\documentclass[10pt]{amsart}
      \usepackage{amsmath,amsfonts}
\def\rg{\hbox to 30pt{\rightarrowfill}}
\def\lg{\hbox to 30pt{\leftarrowfill}}

      \parskip\smallskipamount
          \newtheorem{theorem}{Theorem}[section]
      
      \newtheorem{proposition}[theorem]{Proposition}
      \newtheorem{corollary}[theorem]{Corollary}
      \newtheorem{lemma}[theorem]{Lemma}

      \newtheorem{remark}[theorem]{Remark}
      \makeatletter
      \@addtoreset{equation}{section}
      \makeatother

      \newcommand{\BB}{{\mathbb B}}
      \newcommand{\CC}{{\mathbb C}}
      \newcommand{\NN}{{\mathbb N}}
      
      \newcommand{\ZZ}{{\mathbb Z}}
      \newcommand{\DD}{{\mathbb D}}
      
      \newcommand{\FF}{{\mathbb F}}
      \newcommand{\TT}{{\mathbb T}}

\newcommand{\HH}{{\mathbb H}}
\newcommand{\KK}{{\mathbb K}}

      \newcommand{\cA}{{\mathcal A}}
      
      \newcommand{\cC}{{\mathcal C}}
      \newcommand{\cD}{{\mathcal D}}
      \newcommand{\cE}{{\mathcal E}}
      
      \newcommand{\cG}{{\mathcal G}}
      \newcommand{\cH}{{\mathcal H}}
      \newcommand{\cK}{{\mathcal K}}
      \newcommand{\cL}{{\mathcal L}}
      \newcommand{\cM}{{\mathcal M}}
      \newcommand{\cN}{{\mathcal N}}
      \newcommand{\cQ}{{\mathcal Q}}
      \newcommand{\cP}{{\mathcal P}}
      
      \newcommand{\cS}{{\mathcal S}}

      \newcommand{\cV}{{\mathcal V}}
      \newcommand{\cY}{{\mathcal Y}}

      \newcommand{\supp}{\hbox{\rm{supp}}\,}
      \newcommand{\rank}{\hbox{\rm{rank}}\,}

      \newdimen\expt
      \def\boxit#1{\setbox0\hbox{$\displaystyle{#1}$}
            \hbox{\lower.4\expt
       \hbox{\lower3\expt\hbox{\lower\dp0
            \hbox{\vbox{\hrule height.4\expt
       \hbox{\vrule width.4\expt\hskip3\expt
            \vbox{\vskip3\expt\box0\vskip2\expt}%
       \hskip3\expt\vrule width.4\expt}\hrule height.4\expt}}}}}}
      
\begin{document}
       \pagestyle{myheadings}
      \markboth{ Gelu Popescu}{  Noncommutative polydomains,   Berezin transforms,   and operator model theory    }

      \title [  Berezin transforms on Noncommutative polydomains  ]
      {          Berezin transforms on Noncommutative polydomains }
        \author{Gelu Popescu}
\date{February 24, 2013}
      \thanks{Research supported in part by an NSF grant}

      \address{Department of Mathematics, The University of Texas
      at San Antonio \\ San Antonio, TX 78249, USA}
      \email{\tt gelu.popescu@utsa.edu}

\begin{abstract} This paper is an attempt to unify the multivariable  operator
 model  theory for ball-like domains and commutative polydiscs, and extend it  to
   a more general class of  noncommutative polydomains ${\bf D_q^m}(\cH)$
    in $B(\cH)^n$. An important role in our study is played by  noncommutative
     Berezin transforms associated  with  the  elements of  the polydomain.
      These transforms are  used  to
prove that each such polydomain has a universal model ${\bf W}=\{{\bf W}_{i,j}\}$
 consisting of weighted shifts acting on a tensor product of full Fock spaces.
We introduce the noncommutative Hardy algebra $F^\infty({\bf D_q^m})$ as the
 weakly closed algebra generated by $\{{\bf W}_{i,j}\}$ and the identity,
  and use it to provide a WOT-continuous functional calculus for completely
   non-coisometric   tuples in ${\bf D_q^m}(\cH)$, which are identified.
    It is shown that the Berezin transform  is a completely isometric
    isomorphism   between $F^\infty({\bf D_q^m})$ and  the algebra of
     bounded free holomorphic functions on the radial part of ${\bf D_q^m}(\cH)$.
      A characterization of the Beurling type  joint invariant subspaces
      under $\{{\bf W}_{i,j}\}$  is also provided.

It has been an open problem for quite some time to find  significant
classes of elements in the commutative polidisc   for which a theory
 of characteristic functions and model theory can be developed along
  the lines of  the Sz.-Nagy--Foias theory of  contractions.  We give
   a positive answer  to this question, in our more general setting,
    providing a characterization for the class of tuples of  operators
     in  ${\bf D_q^m}(\cH)$ which admit  characteristic functions. The
      characteristic function is constructed explicitly as an artifact
      of the noncommutative Berezin kernel associated with  the  polydomain,
       and it is proved to be  a complete unitary invariant  for the class of
        completely non-coisometric tuples. Using noncommutative Berezin transforms
         and $C^*$-algebras techniques, we develop a dilation theory on the
          noncommutative polydomain  ${\bf D_q^m}(\cH)$.

\end{abstract}

      \maketitle

\section*{Contents}
{\it

\quad Introduction

\begin{enumerate}
\item[1.]  A class of noncommutative polydomains
   \item[2.]    Noncommutative Berezin transforms  and universal models
   \item[3.]    Noncommutative Hardy algebras  and functional calculus
   \item[4.]  Free holomorphic functions on noncommutative polydomains
   \item[5.]  Joint invariant subspaces and universal models
 \item[6.]    Characteristic functions and operator models
\item[7.]    Dilation theory on noncommutative polydomains
   \end{enumerate}

\quad References

}

\bigskip

\section*{Introduction}

Throughout this paper, we denote by $B(\cH)$ the algebra of bounded
linear operators on a Hilbert space $\cH$.
A polynomial $q\in \CC[Z_1,\ldots, Z_n]$ in $n$ noncommuting indeterminates   is called positive regular if all its coefficients are positive, the constant term is zero, and the coefficients of the linear terms $Z_1,\ldots, Z_n$ are different from zero.  If ${ X}=(X_1,\ldots, X_n)\in B(\cH)^n$ and $q=\sum_{\alpha} a_\alpha Z_\alpha$, we define the map
$\Phi_{q,X}:B(\cH)\to B(\cH)$ by setting $\Phi_{q,X}(Y):=\sum_{\alpha} a_\alpha X_\alpha Y X_\alpha ^*$.

Given two $k$-tuples ${\bf m}:=(m_1,\ldots, m_k)$ and ${\bf n}:=(n_1,\ldots, n_k)$  with $m_i,n_i\in  \NN:=\{1,2,\ldots\}$,  and a $k$-tuple ${\bf q}=(q_1,\ldots, q_k)$ of positive regular polynomials $q_i\in \CC[Z_1,\ldots, Z_{n_i}]$, we associate with each  element ${\bf X}=(X_1,\ldots, X_k)\in B(\cH)^{n_1}\times\cdots \times B(\cH)^{n_k}$ the {\it defect mapping} ${\bf \Delta_{q,X}^m}:B(\cH)\to  B(\cH)$ defined by
$$
{\bf \Delta_{q,X}^m}:=\left(id -\Phi_{q_1, X_1}\right)^{m_1}\circ \cdots \circ\left(id -\Phi_{q_k, X_k}\right)^{m_k}.
$$
We denote by $B(\cH)^{n_1}\times_c\cdots \times_c B(\cH)^{n_k}$
   the set of all tuples  ${\bf X}=({ X}_1,\ldots, { X}_k)\in B(\cH)^{n_1}\times\cdots \times B(\cH)^{n_k}$, where ${ X}_i:=(X_{i,1},\ldots, X_{i,n_i})\in B(\cH)^{n_i}$, $i\in \{1,\ldots, k\}$,
     with the property that, for any $p,q\in \{1,\ldots, k\}$, $p\neq q$, the entries of ${ X}_p$ are commuting with the entries of ${ X}_q$. In this case we say that ${ X}_p$ and ${ X}_q$ are commuting tuples of operators. Note that the operators $X_{i,1},\ldots, X_{i,n_i}$ are not necessarily commuting.

In this paper,  we develop an operator  model  theory and a theory of free holomorphic functions on  the noncommutative polydomains
$$
{\bf D_q^m}(\cH):=\left\{ {\bf X}=(X_1,\ldots, X_k)\in B(\cH)^{n_1}\times_c\cdots \times_c B(\cH)^{n_k}: \ {\bf \Delta_{q,X}^p}(I)\geq 0 \ \text{ for }\ {\bf 0}\leq {\bf p}\leq {\bf m}\right\}.
$$
Our study is an attempt to unify the multivariable operator model theory  for the ball-like domains and  commutative polydiscs,  and to extend it further to the above-mentioned polydomains. The main tool in our investigation is a Berezin \cite{Ber} type transform associated with the {\it abstract noncommutative domain} ${\bf D_q^m}:=\{{\bf D_q^m}(\cH): \ \cH \text{ is a Hilbert space}\}$.

In the last sixty years, this type of  polydomains has been studied in several
particular cases. Most of all, we  mention   the study of
the closed operator unit ball
$$
[B(\cH)]_1^-:=\{X\in B(\cH):\ I-XX^*\geq 0\}
$$
(which corresponds to the case $k=n_1=m_1=1$, and $q_1=Z$) which has
generated the celebrated
 Sz.-Nagy--Foias \cite{SzFBK-book} theory of contractions on Hilbert spaces and has had profound
 implications  in  function theory, interpolation,
  and linear systems theory.
 When   $k=n_1=1$, $m_1\geq 2$, and $q_1=Z$,  the
    corresponding domain coincides with the set of all
    $m$-hypercontractions  studied by Agler  in \cite{Ag1}, \cite{Ag2},
    and  recently by
    Olofsson \cite{O1}, \cite{O2}.

 In several variables, the case when
  $k=1$, $n_1\geq 2$, $m_1=1$, and $q_1=Z_1+\cdots+ Z_n$, corresponds to
 the closed operator  ball
 $$
 [B(\cH)^n]_1^-:=\left\{
(X_1,\ldots, X_n)\in B(\cH)^n:\ I-X_1 X_1^*-\cdots -X_nX_n^*\geq 0
 \right\}
$$
and its study has generated a {\it free } analogue of
Sz.-Nagy--Foias theory (see \cite{F}, \cite{Bu},
\cite{Po-isometric}, \cite{Po-charact},
\cite{Po-von},  \cite{Po-funct}, \cite{Po-analytic},
  \cite{Po-poisson}, \cite{Po-curvature},
\cite{DP2}, \cite{DKS}, \cite{BV},
  \cite{Po-varieties}, \cite{Po-unitary}, \cite{Po-automorphism}, and the references there in).
The commutative case was considered by Drurry \cite{Dru}, extensively
studied by Arveson \cite{Arv1}, \cite{Arv2},   and also in
\cite{Po-poisson}, \cite{Po-varieties}, \cite{BES1}, and \cite{BS}.
  We should  remark that, in recent years, many results
 concerning the theory of row contractions were extended by Muhly
 and Solel (\cite{MuSo1}, \cite{MuSo2}, \cite{MuSo3})
 to representations of tensor algebras over $C^*$-correspondences and Hardy algebras.
We mention that in the particular case when $k=1$  and $q_1$ is a positive regular polynomial, the corresponding   domain was studied in \cite{Po-domains}, if $m_1=1$,  and   in \cite{Po-Berezin}, \cite{Po-similarity-domains}, \cite{Po-classification},  when $m_1\geq 2$. The commutative case when $m_1\geq 2$, $n_1\geq 2$, and  $q_1=Z_1+\cdots + Z_n$,  was studied
   by Athavale \cite{At1}, M\" uller \cite{M}, M\" uller-Vasilescu \cite{MVa},
   Vasilescu \cite{Va1}, and Curto-Vasilescu \cite{CV1}.
   Some  of these results  were extended by
   S.~Pott \cite{Pot} when $q_1$ is a positive regular  polynomial in commuting indeterminates.

The commutative polydisc case, i.e, $k\geq 2$, $n_1=\cdots=n_k=1$,
 and ${\bf q}=(Z_1,\ldots,Z_k)$,  was first considered by Brehmer \cite{Br} in connection with regular dilations. Motivated by Agler's work \cite{Ag2} on weighted shifts as model operators, Curto and Vasilescu developed a theory of standard operator models in
 the polydisc in \cite{CV2}, \cite{CV3}. Timotin \cite{T} was able to obtain some of their results from Brehmer's theorem. The polyball case, when $k\geq 2$ and $q_i=Z_1+\cdots +Z_{n_i}$, $i\in \{1,\ldots,k\}$, was considered in \cite{Po-poisson}  and \cite{BeTi2} for the noncommutative and commutative case, respectively.
  As far as we know, unlike the ball case, there is no theory of characteristic functions, analoguos to the  Sz.-Nagy--Foias theory,  for significant classes of  operators in the    polydisc (or polyball) case.

In Section 1, we work out some basic properties of the  noncommutative polydomains ${\bf D_q^m}(\cH)$. One of the main results, which plays an important role in the present paper,  states that any podydomain ${\bf D_q^m}(\cH)$ is {\it radial}, i.e., $r {\bf X}\in {\bf D_q^m}(\cH)$  whenever ${\bf X}\in {\bf D_q^m}(\cH)$ and $r\in [0,1)$. This fact has also an important consequence
in the particular case when $k=1$, namely, that
all the results from \cite{Po-Berezin}, \cite{Po-similarity-domains}, \cite{Po-classification},  which were proved in the setting of the radial part of  ${\bf D}_{q_1}^{m_1}(\cH)$,  are true for any domain ${\bf D}_{q_1}^{m_1}(\cH)$.

In Section 2, we introduce the {\it  noncommutative Berezin transform}
 at  ${\bf T}\in {\bf D_q^m}(\cH)$ to be the mapping
 $ {\bf B_{T}}: B(\otimes_{i=1}^k F^2(H_{n_i}))\to B(\cH)$
 defined by

 \begin{equation*}
 {\bf B_{T}}[g]:= {\bf K^*_{q,T}} (g\otimes I_\cH) {\bf K_{q,T}},
 \qquad g\in B(\otimes_{i=1}^k F^2(H_{n_i})),
 \end{equation*}
 where $F^2(H_{n_i})$ is the full Fock space on $n_i$ generators and
 $${\bf K_{q,T}}: \cH \to F^2(H_{n_1})\otimes \cdots \otimes  F^2(H_{n_k})
  \otimes  \overline{{\bf \Delta_{q,T}^m}(I) (\cH)}$$
 is the {\it noncommutative Berezin kernel} associated with
   ${\bf T}$, which  is defined in terms of the coefficients of the positive regular polynomials $q_1,\ldots, q_k$.
We remark that in the particular case when $\cH=\CC$, ${\bf q}=(Z_1,\ldots, Z_k)$, ${\bf T}=\lambda=(\lambda_1,\ldots, \lambda_k)\in \DD^k$, and $m_i=n_i=1$ for any $i\in \{1,\ldots,k\}$,   we recover the Berezin   transform   of a bounded
linear operator on the Hardy space $H^2(\DD^k)$, i.e.,
$$
{\bf B}_\lambda [g]=\prod_{i=1}^k(1-|\lambda_i|^2)\left<g k_\lambda,
k_\lambda\right>,\qquad g\in B(H^2(\DD^k)),
$$
where $k_\lambda(z):=\prod_{i=1}^k(1-\overline{\lambda}_i z_i)^{-1}$ and  $z=(z_1,\ldots, z_k)\in \DD^k$.

  The noncommutative Berezin transforms are  used  to
prove the main result of this section (Theorem \ref{Berezin-prop}) which shows  that each   polydomain  ${\bf D_q^m}(\cH)$ has a universal model ${\bf W}=\{{\bf W}_{i,j}\}$ consisting of weighted shifts acting on a tensor product of full Fock spaces. Moreover, we show that a tuple  of operators ${\bf X} $
 is in the
noncommutative polydomain ${\bf D^m_q}(\cH)$ if and only if  there
exists a completely positive linear map $\Psi:C^*({\bf W}_{i,j})\to B(\cH)$ such that  $$
\Psi ( p({\bf W})r({\bf W})^*)= p(X)r(X)^*,
 $$
 for any $p({\bf W}), r({\bf W})$ polynomials   in $ \{{\bf W}_{i,j}\}$ and the identity.

 In Section 3, we introduce the noncommutative Hardy algebra $F^\infty({\bf D_q^m})$ as the weakly closed algebra generated by $\{{\bf W}_{i,j}\}$ and the identity, and use it to provide a WOT-continuous functional calculus for {\it completely non-coisometric}   tuples ${\bf T}=\{T_{i,j}\}$ in ${\bf D_q^m}(\cH)$, which are identified.
 We show that
 $$\Phi(\varphi):=\text{\rm SOT-}\lim_{r\to 1} \varphi(rT_{i,j}), \qquad
  \varphi=\varphi({\bf W}_{i,j})\in F^\infty({\bf D_q^m}),
 $$
 exists in the strong operator topology   and defines a map
 $\Phi:F^\infty({\bf D_q^m})\to B(\cH)$ with the
 property that
 $\Phi(\varphi)=\text{\rm SOT-}\lim\limits_{r\to 1}{\bf B}_{r{\bf T}}[\varphi]$, where ${\bf
B}_{r{\bf T}}$ is the noncommutative  Berezin transform  at  $r{\bf T}\in {\bf D_q^m}(\cH)$.
 Moreover,
$\Phi$ is a unital completely contractive homomorphism, which is  WOT-continuous (resp.~SOT-continuous)  on bounded sets.

  In Section 4, we introduce the algebra  $Hol({\bf D_{q, \text{\rm rad}}^m})$   of all free holomorphic
functions on the {\it abstract    radial polydomain} ${\bf D_{q, \text{\rm rad}}^m}$.  We identify the {\it polydomain algebra} $\cA({\bf D_q^m})$ (the closed algebra generated by $\{{\bf W}_{i,j}\}$ and the identity) and the
Hardy algebra $F^\infty({\bf D_q^m})$ with subalgebras  of $Hol({\bf D_{q, \text{\rm rad}}^m})$.
For example, it is shown that the noncommutative Berezin transform  is a completely isometric isomorphism   between $F^\infty({\bf D_q^m})$ and  the algebra of bounded free holomorphic functions on ${\bf D_{q, \text{\rm rad}}^m}$.
 We remark that there is an
important connection between the theory of free holomorphic
functions on abstract radial polydomains ${\bf D_{q, \text{\rm rad}}^m}$,
 and the theory of holomorphic functions on polydomains in
$\CC^d$ (see \cite{Kr}, \cite{Ru}).
Indeed, if $\cH=\CC^p$ and $p\in\NN$, then
     ${\bf D_{q}^m}(\CC^p)$ can be seen as a subset of $\CC^{(n_1+\cdots +n_k)p^2}$ with
an arbitrary norm. Given    a  free holomorphic function $\varphi$  on the abstract radial polydomain ${\bf D_{q,\text{\rm rad}}^m}$, we prove that
   its representation on $\CC^p$,  i.e., the map $\widehat \varphi$ defined by
$$
  \CC^{(n_1+\cdots +n_k)p^2}\supset {\bf
D_{q,\text{\rm rad}}^m}(\CC^p)\ni (\lambda_{i,j})\mapsto \varphi( \lambda_{i,j})\in
\CC^{p^2}
$$
is a  holomorphic function on the interior of ${\bf
D_{q}^m}(\CC^p)$. In addition,  $\widehat \varphi$  is bounded  when $\varphi\in F^\infty({\bf D_q^m})$, and it
  has continuous extension  to  ${\bf D_{q}^m}(\CC^p)$) when  $\varphi\in \cA({\bf D_q^m})$.

In Section 5, we obtain  a characterization of the Beurling \cite{Be} type  joint invariant subspaces under $\{{\bf W}_{i,j}\}$.   We prove that a subspace $\cM\subset \otimes_{i=1}^k F^2(H_{n_i})\otimes \cH$    has the form  $\cM=\Psi\left((\otimes_{i=1}^k F^2(H_{n_i}))\otimes \cE\right)$ for some {\it inner multi-analytic operator} with respect to the universal model ${\bf W}$,  if and only if
 $$
  {\bf \Delta_{q,W\otimes I}^p}(P_\cM)\geq 0,\qquad \text{ for any }\ {\bf p} \in \ZZ_+^k,  {\bf p}\leq {\bf m},
   $$
 where $P_\cM$ is the orthogonal projection of the Hilbert space
  $\otimes_{i=1}^k F^2(H_{n_i})\otimes \cH$ onto $\cM$. In the particular case when  ${\bf m}=(1,\ldots, 1)$, the latter condition
  is satisfied when ${\bf W}\otimes I|_\cM$  is a doubly  commuting tuple.
   We also characterize the reducing subspaces under
    $\{{\bf W}_{i,j}\}$ and present several results concerning the model theory for pure elements in the noncommutative polydomain ${\bf D_q^m}(\cH)$.

In Section 6, we  provide a characterization for the class of tuples of  operators in  ${\bf D_q^m}(\cH)$ which admit  characteristic functions.
 We say that  ${\bf T} \in {\bf D_q^m}(\cH)$  has   characteristic function if there  is a multi-analytic operator $\Psi$  with respect to the universal model ${\bf W}$  such that
$$
{\bf K_{q,T}}{\bf K_{q,T}^*} +\Psi \Psi^*=I,
$$
where ${\bf K_{q,T}}$ is the noncommutative Berezin kernel associated with ${\bf D_q^m}(\cH)$. In this case, $\Psi$ is essentially unique.
We prove that ${\bf T} \in {\bf D_q^m}(\cH)$  has   characteristic function if and only if
$$
  {\bf \Delta_{q,W\otimes\text{\rm I}}^p}\left(I-{\bf K_{q,T}}{\bf K_{q,T}^*}\right)\geq 0,\qquad \text{ for any }\ {\bf p} \in \ZZ_+^k,  {\bf p}\leq {\bf m}.
   $$
The characteristic function is constructed explicitly  and it is proved to be  a complete unitary invariant  for the class of completely non-coisometric tuples. Moreover, we provide an operator model for this class of elements in  ${\bf D_q^m}(\cH)$ in terms of their   characteristic functions.

 In Section 7,  using  several results from the previous sections and $C^*$-algebras techniques, we develop a dilation theory on the noncommutative polydomain  ${\bf D_q^m}(\cH)$. The main result states that if
  ${\bf T}=\{ T_{i,j}\}$ is  a tuple  in
${\bf D}_{\bf q}^{\bf m}(\cH)$,
then there exists  a $*$-representation $\pi:C^*({\bf W}_{i,j})\to
B(\cK_\pi)$  on a separable Hilbert space $\cK_\pi$,  which
annihilates the compact operators and ${\bf \Delta_{q,\pi(W)}^m}(I_{\cK_\pi})=0$
such that
$\cH$ can be identified with a $*$-cyclic co-invariant subspace of
$$\tilde\cK:=\left[(\otimes _{i=1}^k F^2(H_{n_i}))\otimes
 \overline{{\bf \Delta_{q,T}^m}(I)(\cH)}\right]\oplus
\cK_\pi$$ under  each operator
$$
V_{i,j}:=\left[\begin{matrix} {\bf W}_{i,j}\otimes
I &0\\0&\pi({\bf W}_{i,j})
\end{matrix}\right],
$$
and such that
$ T_{i,j}^*=V_{i,j}^*|_{\cH}$ for all $i,j$. Under a certain additional condition on the universal model ${\bf W}$, the dilation above is  minimal and unique up to unitary equivalence. We also obtain  Wold type
decompositions for non-degenerate $*$-representations of the
$C^*$-algebra $C^*({\bf W}_{i,j})$.

We   mention that  the results of this paper are presented in a more general setting,
 when ${\bf q}$ is replaced by a $k$-tuple ${\bf f}=(f_1,\ldots, f_k)$ of
  positive regular free holomorphic functions in a neighborhood of the origin.
   Also, the  results are used in  \cite{Po-Berezin3} to develop an operator model
   theory for  varieties in  noncommutative polydomains. This includes various
    commutative cases  which are presented in close connection with the theory
     of holomorphic functions in several complex variables.

\bigskip

\section{A class of  noncommutative polydomains}

For each $i\in \{1,\ldots, k\}$,
let $\FF_{n_i}^+$ be the unital free semigroup on $n_i$ generators
$g_{1}^i,\ldots, g_{n_i}^i$ and the identity $g_{0}^i$.  The length of $\alpha\in
\FF_{n_i}^+$ is defined by $|\alpha|:=0$ if $\alpha=g_0^i$  and
$|\alpha|:=p$ if
 $\alpha=g_{j_1}^i\cdots g_{j_p}^i$, where $j_1,\ldots, j_p\in \{1,\ldots, n_i\}$.
If $Z_1,\ldots,Z_{n_i}$  are  noncommuting indeterminates,   we
denote $Z_\alpha:= Z_{j_1}\cdots Z_{j_p}$  and $Z_{g_0^i}:=1$.
 Let  $f_i:= \sum_{\alpha\in
\FF_{n_i}^+} a_{i,\alpha} Z_\alpha$, \ $a_{i,\alpha}\in \CC$,  be a formal power series in $n_i$ noncommuting indeterminates $Z_1,\ldots,Z_{n_i}$. We say that $f_i$ is
a {\it
positive regular free holomorphic function} if the following conditions hold:
$a_{i,\alpha}\geq 0$ for
any $\alpha\in \FF_{n_i}^+$, \ $a_{i,g_{0}^i}=0$,
   $a_{i,g_{j}^i}>0$ for  $j=1,\ldots, n_i$, and
 $$
\limsup_{k\to\infty} \left( \sum_{|\alpha|=k}
|a_{i,\alpha}|^2\right)^{1/2k}<\infty.
 $$
Given $X_i:=(X_{i,1},\ldots, X_{i,n_i})\in B(\cH)^{n_i}$, define the map $\Phi_{f_i,X_i}:B(\cH)\to B(\cH)$  by setting
 $$
\Phi_{f_i,X_i}(Y):=\sum_{k=1}^\infty\sum_{\alpha\in \FF_{n_i}^+,|\alpha|=k} a_{i,\alpha} X_{i,\alpha}
YX_{i,\alpha}^*, \qquad   Y\in B(\cH),$$
 where  the convergence is in the week
operator topology.

Let ${\bf n}:=(n_1,\ldots, n_k)$ and ${\bf m}:=(m_1,\ldots, m_k)$, where $n_i,m_i\in\NN:=\{1,2,\ldots\}$ and $i\in \{1,\ldots, k\}$, and let ${\bf f}:=(f_1,\ldots,f_k)$ be a $k$-tuple of positive regular free holomorphic functions.
We introduce the noncommutative polydomain
${\bf D_f^m}(\cH)$
to be the set of all $k$-tuples
$${\bf X}:=(X_1,\ldots, X_k)\in   B(\cH)^{n_1}\times_c\cdots \times_c B(\cH)^{n_k}$$
with the property that $\Phi_{f_i,X_i}(I)\leq I$ and
$$
(id-\Phi_{f_1,X_1})^{\epsilon_1 m_1}\cdots (id-\Phi_{f_k,X_k})^{\epsilon_k m_k}(I)\geq 0
$$
for any $i\in \{1,\ldots,k\}$ and  $\epsilon_i\in \{0,1\}$.
 We use the convention that $(id-\Phi_{f_i,X_i})^0=id$.
  We remark that ${\bf D_f^m}(\cH)$ contains  a polyball
   $[B(\cH)^{n_1}]_{r_1}^-\times_c \cdots \times_c [B(\cH)^{n_k}]_{r_k}^-$
   for some $r_1,\ldots, r_k>0$, where
$$
[B(\cH)^{n_i}]_{r_i}^-:=\{(Y_1,\ldots, Y_{n_i})\in B(\cH)^{n_i}: \ Y_1 Y_1^*+\cdots +Y_{n_i} Y_{n_i}^*\leq r_i^2 I\}.
$$
 Throughout this paper, we refer to ${\bf D_q^m}:=\{{\bf D_q^m}(\cH): \ \cH \text{is a Hilbert space}\}$ as the abstract noncommutative polydomain, and ${\bf D_q^m}(\cH)$ as its representation on the Hilbert space $\cH$.

A linear map $\varphi:B(\cH)\to
B(\cH)$ is called power bounded if there exists a constant $M>0$
such that $\|\varphi^k\|\leq M$ for any $k\in \NN$. For  information  on completely bounded (resp. positive) maps,  we refer
 to \cite{Pa-book} and  \cite{Pi-book}. If ${\bf p}:=(p_1,\ldots, p_k)\in \ZZ_+^k$ and  ${\bf q}:=(q_1,\ldots, q_k)\in \ZZ_+^k$, we set ${\bf p}\leq {\bf q}$ iff $p_i\leq q_i$ for all $i\in \{1,\ldots,k\}$, where $\ZZ_+:=\{0,1,\ldots\}$.

 \begin{proposition} \label{positive-maps}
 Let $\varphi_i:B(\cH)\to B(\cH)$, $i\in \{1,\ldots,k\}$, be power bounded positive linear maps such that
 $$\varphi_i \varphi_j=\varphi_j \varphi_i,\qquad i,j\in \{1,\ldots,k\}.
  $$
  If  $Y\in B(\cH)$ is a self-adjoint operator and ${\bf p}:=(p_1,\ldots, p_k)\in \ZZ_+^k$ with $p_i\geq 1$, then the following statements are equivalent.
 \begin{enumerate}
 \item[(i)] $(id-\varphi_1)^{\epsilon_1 p_1}\cdots (id-\varphi_k)^{\epsilon_k p_k}(Y)\geq 0 $ for all $\epsilon_i\in \{0,1\}$ with $ \epsilon:=(\epsilon_1,\ldots, \epsilon_k)\neq 0$ and $i\in \{1,\ldots,k\}$.
     \item [(ii)]  $(id-\varphi_1)^{ q_1}\cdots (id-\varphi_k)^{q_k}(Y)\geq 0 $ for all ${\bf q}:=(q_1,\ldots, q_k)\in \ZZ_+^k $ with  ${ \bf q}\leq { \bf p}$ and ${\bf q}\neq 0$.

 \end{enumerate}
\end{proposition}

\begin{proof} Note that it is enough to prove that  $(id-\varphi_1)^{ p_1}\cdots (id-\varphi_k)^{p_k}(Y)\geq 0 $ if and only if $(id-\varphi_1)^{ q_1}\cdots (id-\varphi_k)^{q_k}(Y)\geq 0 $ for all ${\bf q}:=(q_1,\ldots, q_k)\in \ZZ_+^k $ with  ${ q_i}\leq { p_i}$ and ${ q_i}\geq 1$.
We proceed by induction over  $k\in \NN$. Let $k=1$,  and assume that
$(id-\varphi_1)^{p_1}(Y)\geq 0$ and $p_1\geq 2$.
Suppose that there is $h_0\in \cH$ such that $\left<(id-\varphi_1)^{p_1-1}(Y)h_0,h_0\right><0$. Set
$y_j:=\left<\varphi_1^j(id -\varphi_1)^{p_1-1} (Y)h_0,h_0\right>$, $j=0,1,\ldots$, and note that $\{y_j\}_{j=0}^\infty$ is a decreasing  sequence with
$y_j\leq y_0<0$. Consequently, we deduce that $\sum_{j=0}^\infty y_j=-\infty$.
On the other hand, we have
\begin{equation*}
\begin{split}
\left|\sum_{j=0}^p y_j\right| &:=
\left|\left<(id-\varphi_1^{p+1} )(id-\varphi_1)^{p_1-2}(Y)h_0, h_0\right>\right|\\
&\leq \left(1+\|\varphi_1^{p+1}(I)\|\right) \|(id-\varphi_1)^{p_1-2}(Y)\|\|h_0\|.
\end{split}
\end{equation*}
Since $\varphi_1$ is power bounded, we get a contradiction. Therefore, we must have
$(id-\varphi_1)^{p_1-1}(Y)\geq 0$. Continuing this process, we show that
 $(id-\varphi)^{p_1}(Y)\geq 0$  if and only if  $(id-\varphi)^s(Y)\geq
0$ for  $s=1,2,\ldots,p_1$.
Now, assume that $$(id-\varphi_1)^{ p_1}\cdots (id-\varphi_k)^{ p_k}(id-\varphi_{k+1})^{ p_{k+1}}(Y)\geq 0. $$
 Due to the fact that
$\varphi_i \varphi_j=\varphi_j \varphi_i$ for all $i,j\in \{1,\ldots,k\}$,
  we deduce that
$(id-\varphi_{k+1})^{ p_{k+1}}(Y_k)\geq 0$, where
 $Y_k:=(id-\varphi_1)^{ p_1}\cdots (id-\varphi_k)^{ p_k}(Y)$.
On the other hand,
 due to the identity
$$
(id-\varphi_k)^{p_k}(Y)=\sum_{p=0}^{p_k} (-1)^{p} \left(\begin{matrix}
p_k\\p\end{matrix}\right) \varphi_k^p(Y),
$$
the operator $(id-\varphi_k)^{p_k}(Y)$ is self-adjoint whenever $\varphi_k$ is a positive linear map and  $Y$ is a self-adjoint operator. Inductively, one can easily see that $Y_k$ is a self-adjoint operator. Now, applying the case $k=1$, we deduce that $(id-\varphi_{k+1})^{ p_{k+1}}(Y_k)\geq 0$ if and only if
$(id-\varphi_{k+1})^{ q_{k+1}}(Y_k)\geq 0$ for all  $q_{k+1}\in \{0,1,\ldots, p_{k+1}\}$.
Hence,
$$(id-\varphi_1)^{ p_1}\cdots (id-\varphi_k)^{ p_k}(id-\varphi_{k+1})^{ q_{k+1}}(Y)\geq 0. $$
Due to the induction hypothesis, we deduce that
$$(id-\varphi_1)^{ q_1}\cdots (id-\varphi_k)^{ q_k}(id-\varphi_{k+1})^{ q_{k+1}}(Y)\geq 0 $$
for all $(q_1,\ldots, q_{k+1})\in \ZZ_+^{k+1} $ with  ${ q_i}\leq { p_i}$ and ${ q_i}\geq 1$.
This completes the proof.
\end{proof}

Let $\Phi=(\varphi_1,\ldots, \varphi_k)$ be a $k$-tuple of  power bounded,
 positive linear maps on $B(\cH)$ such that
 $\varphi_i \varphi_j=\varphi_j \varphi_i$, $i,j\in \{1,\ldots,k\}$.
  For each ${\bf p}:=(p_1,\ldots, p_k)\in \ZZ_+^k$, we define the linear map
  $\Delta_{\Phi}^{\bf p}:B(\cH)\to B(\cH)$ by setting
 $$
 \Delta_{\Phi}^{p_1,\ldots,p_k}=\Delta_{\Phi}^{\bf p}:= (id-\varphi_1)^{ p_1}\cdots (id-\varphi_k)^{ p_k}.
 $$

\begin{lemma} \label{Delta-ineq} Let
${\bf m}\in \NN^k$  and let $Y\in B(\cH)$ be a self-adjoint operator such that
$\Delta_{\Phi}^{\bf p}(Y)\geq 0$ for all ${\bf p} \in \ZZ_+^k $ with  ${\bf p}\leq {\bf m}$ and ${\bf p}\neq 0$. If ${\bf q}\in \ZZ_+^k$ with ${\bf q}\neq 0$ and ${\bf q}\leq {\bf m}$, then
$$
\Delta_{\Phi}^{\bf m}(Y)\leq \Delta_{\Phi}^{\bf q}(Y).
$$
\end{lemma}
\begin{proof} Set ${\bf m}:=(m_1,\ldots, m_k)\in \NN^k$ and ${\bf m}':=(m_1-1, m_2,\ldots, m_k)$.
Since $\Delta_{\Phi}^{{\bf m}' }(Y)\geq 0$  and $\varphi_1$ is a positive map,  we deduce that
$$
\Delta_{\Phi}^{\bf m}(Y)=\Delta_{\Phi}^{{\bf m}'}(Y) -\varphi_1(\Delta_{\Phi}^{{\bf m}'}(Y))\leq \Delta_{\Phi}^{{\bf m}'}(Y)
$$
 Using  the fact that $\varphi_i \varphi_j=\varphi_j \varphi_i$ for $i,j\in \{1,\ldots,k\}$, one can continue this process and  complete  the proof.
\end{proof}

\begin{proposition} \label{pure} Let $Y\in B(\cH)$ be  a self-adjoint operator, ${\bf m}\in \ZZ_+^k$, ${\bf m}\neq 0$,  and let
$\Phi=(\varphi_1,\ldots, \varphi_k)$ be a $k$-tuple of commuting,   power bounded, positive linear maps on $B(\cH)$  such that
\begin{enumerate}
\item[(i)] $\Delta_{\Phi}^{\bf m}(Y)\geq 0$, and

\item[(ii)] each $\varphi_i$ is pure, i.e., $\varphi_i^p(I)\to 0$ strongly as $p\to\infty$.

\end{enumerate}
Then  $\Delta_{\Phi}^{\bf q}(Y)\geq 0$ for any ${\bf q}\in \ZZ_+^k$ with ${\bf q}\leq {\bf m}$. In particular, $Y\geq 0$.

\end{proposition}
\begin{proof} Set  ${\bf m}':=(m_1-1, m_2,\ldots, m_k)$ and note that due to the fact that  $\Delta_{\Phi}^{\bf m}(Y)\geq 0$ and $\varphi_1$ is a positive linear map, we have
$$0\leq \Delta_{\Phi}^{\bf m}(Y)=\Delta_{\Phi}^{{\bf m}'}(Y)-\varphi_1(\Delta_{\Phi}^{{\bf m}'}(Y)).
$$
Hence,  we deduce that $\varphi_1^p(\Delta_{\Phi}^{{\bf m}'}(Y))\leq \Delta_{\Phi}^{{\bf m}'}(Y)$ for any $p\in \NN$. Since $\Delta_{\Phi}^{{\bf m}'}(Y)$ is a self-adjoint operator, we have
$$
-\|\Delta_{\Phi}^{{\bf m}'}(Y)\| \varphi_1^p(I)\leq \varphi_1^p(\Delta_{\Phi}^{{\bf m}'}(Y))\leq \|\Delta_{\Phi}^{{\bf m}'}(Y)\| \varphi_1^p(I).
$$
Now, taking into account that  $\varphi_i^p(I)\to 0$ strongly as $p\to\infty$, we conclude that $\Delta_{\Phi}^{{\bf m}'}(Y)\geq 0$. Using the commutativity of $\varphi_1,\ldots, \varphi_k$, one can continue this process and complete the proof.
\end{proof}

 For each $i\in \{1,\ldots,k\}$,  let $f_i:=\sum\limits_{\alpha_i\in \FF_{n_i}^+, |\alpha|\geq 1} a_{i,\alpha} Z_\alpha$ be
a positive regular free holomorphic function in $n_i$ variables
 and let $A:=(A_1,\ldots, A_n)\in B(\cH)^{n_i}$ be an $n_i$-tuple of
 operators such that
$\sum_{|\alpha|\geq 1} a_{i,\alpha} A_\alpha A_\alpha^* $ is convergent
in the weak operator topology.
 One can easily prove that the map
$\Phi_{f_i,A}:B(\cH)\to B(\cH)$, defined by
$$\Phi_{f_i,A}(X)=\sum_{|\alpha|\geq 1} a_{i,\alpha} A_\alpha
XA_\alpha^*,\qquad X\in B(\cH),
$$
where the convergence is in the weak operator topology, is a
 completely positive linear map which is WOT-continuous on bounded
 sets. Moreover, if $0<r<1$, then
$$\Phi_{f_i,A}(X)=\text{\rm WOT-}\lim_{r\to1} \Phi_{f_i,rA}(X), \qquad
X\in B(\cH).
$$
These facts will be used in the proof of the next theorem.

Let
${\bf T}=({ T}_1,\ldots, {T}_k)\in  B(\cH)^{n_1}\times_c\cdots \times_c B(\cH)^{n_k}$, where ${ T}_i:=(T_{i,1},\ldots, T_{i,n_i})\in B(\cH)^{n_i} $ for all   $i=1,\ldots, k$,  be such that $\Phi_{f_i, T_i}(I)$ is well-defined in the weak operator topology.  If ${\bf p}:=(p_1,\ldots, p_k)\in \ZZ_+^k$ and ${\bf f}:=(f_1,\ldots,f_k)$, we  define the {\it defect   mapping}
${\bf \Delta_{f,T}^p}: B(\cH)\to  B(\cH)$ by setting
$$
{\bf \Delta_{f,T}^p} :=(id-\Phi_{f_1,T_1})^{p_1}\cdots (id-\Phi_{f_k,T_k})^{p_k}.
$$
Given   $r\geq 0$, we set $r{\bf T}:=(r{ T}_1,\ldots, r{ T}_k)$ and
$r{ T}_i:=(rT_{i,1},\ldots, rT_{i,n_i})$ for $i\in\{1,\ldots,k\}$. We say that the $k$-tuple ${\bf T}$ has the {\it radial property}  with respect to ${\bf D_f^m}(\cH)$  if    there exists $\delta\in (0,1)$ such that $r{\bf T}\in {\bf D_f^m}(\cH)$ for any $r\in (\delta, 1]$.

\begin{theorem}
\label{radial}
Let
${\bf T}=({ T}_1,\ldots, { T}_k)\in B(\cH)^{n_1}\times_c\cdots \times_c B(\cH)^{n_k}$ be such that $\Phi_{f_i, T_i}(I)\leq I$ for any $i\in \{1,\ldots, k\}$, and let ${\bf q}\in \ZZ_+^k$  be with ${\bf q}\neq 0$. Then the following statements are equivalent:
\begin{enumerate}
\item[(i)] ${\bf T}\in {\bf D_f^m}(\cH)$;
 \item[(ii)] for any $p_i\in \{0,1,\ldots, m_i\}$ and
       $i\in \{1,\ldots,k\}$,
   $$
   (id-\Phi_{f_1,T_1})^{p_1}\cdots (id-\Phi_{f_k,T_k})^{p_k}(I)\geq 0;
    $$
    \item[(iii)]    $ {\bf \Delta}_{{\bf f},r{\bf T}}^{\bf m}(I)\geq 0$ for any $r\in [0, 1]$;
\item[(iv)] there exists $\delta\in (0,1)$ such that $ {\bf \Delta}_{{\bf f},r{\bf T}}^{\bf m}(I)\geq 0$ for any $r\in (\delta, 1)$;
    \item[(v)] ${\bf T} $ has the radial property with respect to ${\bf D_f^m}(\cH)$.
\end{enumerate}
\end{theorem}
\begin{proof}
 The equivalence of (i) with (ii) is due to Proposition \ref{positive-maps}, when applied to $\varphi_i= \Phi_{f_i,T_i}$. We prove that (ii) implies (iii).
 First, note that if $D\in B(\cH)$, $D\geq 0$, then, for each $i\in \{1,\ldots, k\}$,
 \begin{equation}
 \label{ine2}
 (id-\Phi_{f_i, T_i})(D)\geq 0\quad \implies  \quad (id-\Phi_{f_i, rT_i})(D)\geq 0, \quad r\in [0,1].
 \end{equation}
 Indeed,  if $\Phi_{f_i, T_i}(D)\leq D$, then $\Phi_{f_i,r T_i}(D)\leq D$ for any $r\in [0,1]$. Now, assume that (ii) holds. If ${\bf p}\in \ZZ_+^k$ with  ${\bf p}\geq e_1:=(1,0,\ldots,0)\in \ZZ_+^k$, then
 $(id-\Phi_{f_1, T_1})({\bf \Delta}_{{\bf f},{\bf T}}^{{\bf p}-e_1}(I))\geq 0$
 for any ${\bf p}\in \ZZ_+^k$  with $e_1\leq {\bf p}\leq {\bf m}$.
 Consequently, due to \eqref{ine2}, we have
 \begin{equation}
 \label{ine3}
 (id-\Phi_{f_1, rT_1})({\bf \Delta}_{{\bf f},{\bf T}}^{{\bf p}-e_1}(I))\geq 0
 \end{equation}
 for any $r\in [0,1]$ and any ${\bf p}\in \ZZ_+^k$  with $e_1\leq {\bf p}\leq {\bf m}$. Due to the commutativity of $\Phi_{f_1, T_1},\ldots,\Phi_{f_k, T_k}$, the latter inequality is equivalent to
 \begin{equation*}
 (id-\Phi_{f_1, T_1})({\bf \Delta}_{{\bf f},{\bf T}}^{{\bf p}-2e_1}(id-\Phi_{f_1, rT_1})(I))\geq 0
 \end{equation*}
 for any $r\in [0,1]$ and any ${\bf p}\in \ZZ_+^k$  with $2e_1\leq {\bf p}\leq {\bf m}$.
 Due to \eqref{ine3}, we have ${\bf \Delta}_{{\bf f},{\bf T}}^{{\bf p}-2e_1}(id-\Phi_{f_1, rT_1})(I)\geq 0$ and, applying again relation \eqref{ine2}, we deduce that
  \begin{equation*}
 (id-\Phi_{f_1, T_1})({\bf \Delta}_{{\bf f},{\bf T}}^{{\bf p}-3e_1}(id-\Phi_{f_1, rT_1})^2(I))\geq 0
 \end{equation*}
 for any $r\in [0,1]$ and any ${\bf p}\in \ZZ_+^k$  with $3e_1\leq {\bf p}\leq {\bf m}$. Continuing this process, we obtain the inequality
 $$
 (id-\Phi_{f_2, T_2})^{p_2}\cdots (id-\Phi_{f_k, T_k})^{p_k}(id-\Phi_{f_1, rT_1})^{p_1}(I)\geq 0
 $$
 for any ${\bf p}\in \ZZ_+^k$  with $e_1\leq {\bf p}\leq {\bf m}$, and any $r\in [0,1]$.
Similar arguments lead to the inequality  $ {\bf \Delta}_{{\bf f},r{\bf T}}^{\bf m}(I)\geq 0$ for any $r\in [0, 1]$. Since the implications (iii)$\implies$(iv) and (v)$\implies$(i) are clear, it remains to prove that (iv)$\implies$(v).

To this end, assume that there exists $\delta\in (0,1)$ such that $ {\bf \Delta}_{{\bf f},r{\bf T}}^{\bf m}(I)\geq 0$ for any $r\in (\delta, 1)$.
Since $\Phi_{f_i, rT_i}(I)\leq r I$, it is clear that $\Phi_{f_i, rT_i}$ is pure for each $i\in\{1,\ldots, k\}$. Applying Proposition \ref{pure}, we deduce that
$ {\bf \Delta}_{{\bf f},r{\bf T}}^{\bf p}(I)\geq 0$ for any  $r\in (\delta, 1)$ and any ${\bf p}\in \ZZ_+^k$ with ${\bf p}\leq {\bf m}$.
 Note that ${\bf \Delta}_{{\bf f},r{\bf T}}^{\bf p}(I)$ is a linear combination  of products of the form
 $\Phi_{f_1, rT_1}^{q_1}\cdots \Phi_{f_k, rT_k}^{q_k}(I)$, where $(q_1,\ldots, q_k)\in \ZZ_+^k$.
 On the other hand
 $$
 \Phi_{f_1, T_1}^{q_1}\cdots \Phi_{f_k, T_k}^{q_k}(I)=\text{\rm WOT-}\lim_{j\to \infty} \sum_{{\alpha_i\in \FF_{n_i}^+}\atop{|\alpha_1|+\cdots +|\alpha_k|\leq j}} c_{\alpha_1,\ldots,\alpha_k} T_{1,\alpha_1}\cdots T_{k,\alpha_k} T_{k,\alpha_k}^*\cdots T_{1,\alpha_1}^*\leq I
 $$
 for some positive constants $c_{\alpha_1,\ldots,\alpha_k} \geq 0$. Given $x\in \cH$ and $\epsilon >0$, there is $N_0\in \NN$ such that
 $$\sum_{{\alpha_i\in \FF_{n_i}^+}\atop {|\alpha_1|+\cdots +|\alpha_k|\leq j}} c_{\alpha_1,\ldots,\alpha_k} r^{2(|\alpha_1|+\cdots +|\alpha_k|)} \left<T_{1,\alpha_1}\cdots T_{k,\alpha_k} T_{k,\alpha_k}^*\cdots T_{1,\alpha_1}^* x,x\right><\epsilon$$ for any $j\geq N_0$
 and $r\in (\delta, 1)$.  This can be used to show that
 $$
 \Phi_{f_1, T_1}^{q_1}\cdots \Phi_{f_k, T_k}^{q_k}(I)=\text{\rm WOT-}\lim_{r\to 1}
 \Phi_{f_1, rT_1}^{q_1}\cdots \Phi_{f_k, rT_k}^{q_k}(I).
 $$
 Hence, we deduce that
 $
  {\bf \Delta}_{{\bf f},{\bf T}}^{\bf p}(I)=\text{\rm WOT-}\lim_{r\to 1}
  {\bf \Delta}_{{\bf f},r{\bf T}}^{\bf p}(I)\geq 0
  $ \
 for any ${\bf p}\in \ZZ_+^k$ with ${\bf p}\leq {\bf m}$. Consequently,
  ${\bf T}\in {\bf D_f^m}(\cH)$ and it has the radial property.  This  completes the proof.
 \end{proof}
As expected, the   domain ${\bf D_f^m}(\cH)$ is called {\it radial} if any ${\bf T}\in {\bf D_f^m}(\cH)$ has the radial property.

\begin{corollary}  The  noncommutative polydomain ${\bf D_f^m}(\cH)$ is radial.
   \end{corollary}

   In the particular case when $k=1$,  Theorem \ref{radial} shows   that   any noncommutative domain ${\bf D}_{f_1}^{m_1}(\cH)$, $m_1\in \NN$,  is radial.  An important consequence is the following

    \begin{corollary} All the results from \cite{Po-Berezin}, \cite{Po-similarity-domains}, \cite{Po-classification},  which were proved in the setting of the radial part of  ${\bf D}_{f_1}^{m_1}(\cH)$,  are true for any domain ${\bf D}_{f_1}^{m_1}(\cH)$.
    \end{corollary}

Another consequence of  Theorem \ref{radial} is the following

\begin{corollary} The following statements hold:
 \begin{enumerate}
\item[(i)]
If   ${\bf f}=(f_1,f_2)$,  and  ${\bf T}=({ T}_1,  { T}_2)\in {\bf D}_{f_1}^{m_1} (\cH)\times_c {\bf D}_{f_2}^{m_2}(\cH)$ with  ${\bf \Delta_{f,T}^m}(I)\geq 0$, then  ${\bf T}\in {\bf D_f^m}(\cH)$.
\item[(ii)] If ${\bf T}=({ T}_1,\ldots, { T}_k)\in B(\cH)^{n_1}\times_c\cdots \times_c B(\cH)^{n_k}$ and
$\Phi_{f_i,T_i}(I) =I$, $ i\in \{1,\ldots,k\},
$
then ${\bf T}$ is in the polydomain ${\bf D_f^m}(\cH)$.
 \end{enumerate}
\end{corollary}

  We say that  a $k$-tuple
${\bf T}=({ T}_1,\ldots, { T}_k)\in {\bf D_f^m}(\cH)$ is {\it pure} if
$$
\lim_{{\bf q}=(q_1,\ldots, q_k)\in \ZZ_+^k}(id-\Phi_{f_k,T_k}^{q_k})\cdots (id-\Phi_{f_1,T_1}^{q_1})(I)=I.
$$
We remark that  $\{(id-\Phi_{f_k,T_k}^{q_k})\cdots (id-\Phi_{f_1,T_1}^{q_1})(I)\}_{{\bf q}=(q_1,\ldots, q_k)\in \ZZ_+^k}$ is an  increasing sequence of positive operators.
 Indeed,  due to Theorem \ref{radial},  $(id-\Phi_{f_k,T_k})\cdots (id-\Phi_{f_1,T_1})(I)\geq 0$. Taking into account that $\Phi_{f_1, T_1}, \ldots, \Phi_{f_k,T_k}$ are commuting, we have
 $$
 (id-\Phi_{f_k,T_k}^{q_k})\cdots (id-\Phi_{f_1,T_1}^{q_1})(I)
 =\sum_{s=0}^{q_k-1}\Phi_{f_k, T_k}^s\cdots \sum_{s=0}^{q_1-1}\Phi_{f_1,T_1}^s
(id-\Phi_{f_k,T_k})\cdots (id-\Phi_{f_1,T_1})(I),
$$
which proves our assertion.
Note also that
$$
 (id-\Phi_{f_k,T_k}^{q_k})\cdots (id-\Phi_{f_1,T_1}^{q_1})(I)\leq  (id-\Phi_{f_{k-1},T_{k-1}}^{q_{k-1}})\cdots (id-\Phi_{f_1,T_1}^{q_1})(I)\leq\cdots \leq   (id-\Phi_{f_1,T_1}^{q_1})(I)\leq I.
$$
Hence, we can deduce the following result.
 \begin{proposition} \label{pure2} A $k$-tuple ${\bf T}=({ T}_1,\ldots, { T}_k)\in {\bf D_f^m}(\cH)$ is pure if and only if,   for each $i\in \{1,\ldots,k\}$,
$\Phi_{f_i,T_i}^p(I)\to 0$ strongly as $p\to\infty$.
\end{proposition}

A $k$-tuple ${\bf T}\in {\bf D_f^m}(\cH)$ is called  {\it doubly commuting}  if
$T_{i,p}T_{j,q}^*=T_{j,q}^*T_{i,p}$ for any  $i,j\in \{1,\ldots,k\}$ with $i\neq j$, and
$p\in \{1,\ldots, n_i\}$, $q\in \{1,\ldots, n_j\}$. The next results provides some classes of elements in ${\bf D_f^m}(\cH)$.

\begin{proposition}\label{exemp} Let ${\bf T}=({ T}_1,\ldots, { T}_k)\in B(\cH)^{n_1}\times_c\cdots \times_c B(\cH)^{n_k}$  be such that $\Phi_{f_i, T_i}(I)\leq I$ and let ${\bf m}=(m_1,\ldots, m_k)\in \NN^k$. Then the following statements hold.
\begin{enumerate}
\item[(i)]
If
${\bf \Delta_{f,T}^m}(I)\geq 0$ and, for each $i\in \{1,\ldots,k\}$,
$\Phi_{f_i,T_i}^p(I)\to 0$ strongly as $p\to\infty$, then  ${\bf T}\in {\bf D_f^m}(\cH)$.

\item[(ii)] If ${\bf T}\in {\bf D}_{f_1}^{m_1}(\cH)\times_c\cdots \times_c {\bf D}_{f_k}^{m_k}(\cH)$ is doubly commuting, then  ${\bf T}\in {\bf D_f^m}(\cH)$.
    \item[(iii)] If  $
m_1 \Phi_{f_1,T_1}(I)+\cdots + m_k\Phi_{f_k,T_k}(I)\leq I,
$
then ${\bf T}\in {\bf D_f^m}(\cH)$.
\item[(iv)]
 If ${\bf T}=({ T}_1,\ldots, { T}_k)\in {\bf D_f^m}(\cH)$, then   ${\bf \Delta_{f,T}^m}(I)= 0$ if and only if
$$(id-\Phi_{f_1,T_1})\cdots (id-\Phi_{f_k,T_k})(I) =0.
$$
\end{enumerate}
\end{proposition}

\begin{proof} Applying Proposition \ref{positive-maps} and Proposition \ref{pure}, when $\Phi=(\Phi_{f_1, T_1},\ldots, \Phi_{f_k,T_k})$, we deduce part (i).
To prove part (ii), note that  since  $T_i\in {\bf D}_{f_i}^{m_i}(\cH)$, we have
$(id-\Phi_{f_i, T_i})^{p_i}(I)\geq 0$ for any $p_i\in \{0,1,\ldots,m_i\}$. Using the fact that ${\bf T}$ is doubly commuting, we deduce that
$$
{\bf \Delta_{f,T}^p}(I)=(id-\Phi_{f_1, T_1})^{p_1}(I)\cdots (id-\Phi_{f_k, T_k})^{p_k}(I)\geq 0
$$
for any ${\bf p}\in \ZZ_+^k$ with ${\bf p}\leq {\bf m}$, which shows that   ${\bf T}\in {\bf D_f^m}(\cH)$.

Now, we prove part (iii). Let $p:=m_1+\cdots +m_k$ and   set $i_j:=1$ if $1\leq j\leq m_1$, $i_j:=2$ if $m_1+1\leq j\leq m_1+m_2,\ldots$, and $i_j:=k$ if $m_1+\cdots +m_{k-1}+1\leq j\leq m_1+\cdots +m_k$. Due to Theorem \ref{radial}, to prove (iii) is equivalent to showing that if $\sum_{j=1}^p \Phi_{f_{i_j}, T_{i_j}}(I)\leq I$, then
$$
(id-\Phi_{f_{i_1}, T_{i_1}})\cdots (id-\Phi_{f_{i_p}, T_{i_p}})(I)\geq 0.
$$
Set $Y_{i_0}=I$ and $Y_{i_j}:=(id -\Phi_{f_{i_j}, T_{i_j}})(Y_{i_{j-1}})$ if $j\in \{1,\ldots, p\}$. We proceed inductively.
Note that
$I=Y_{i_0}\geq Y_{i_1}=(id -\Phi_{f_{i_1}, T_{i_1}})(I)\geq 0$.   Let $n<p$ and assume that
$$
I\geq Y_{i_n}\geq (id-\Phi_{f_{i_1}, T_{i_1}}-\cdots -\Phi_{f_{i_n}, T_{i_n}})(I)\geq 0.
$$
Hence, we deduce that
\begin{equation*}
\begin{split}
I&\geq Y_{i_n}\geq Y_{i_{n+1}}= Y_{i_n}-\Phi_{f_{i_{n+1}}, T_{i_{n+1}}}(Y_{i_n})\\
&\geq (id-\Phi_{f_{i_1}, T_{i_1}}-\cdots -\Phi_{f_{i_n}, T_{i_n}})(I)
-\Phi_{f_{i_{n+1}},T_{i_{n+1}}}(I),
\end{split}
\end{equation*}
which proves our assertion.

Now, we prove part (iv).
If ${\bf T}=({ T}_1,\ldots, { T}_k)\in {\bf D_f^m}(\cH)$, Theorem \ref{radial} implies that
   $$
   (id-\Phi_{f_1,T_1})^{p_1}\cdots (id-\Phi_{f_k,T_k})^{p_k}(I)\geq 0
    $$
for any $p_i\in \{0,1,\ldots, m_i\}$ and
       $i\in \{1,\ldots,k\}$. Due to   Lemma 6.2 from  \cite{Po-Berezin}, if $\varphi:B(\cH)\to B(\cH)$ is a power bounded positive linear map such that $D\in B(\cH)$ is a positive operator with $(id-\varphi)(D)\geq 0$, and $\gamma\geq 1$,  then
$$
(id-\varphi)^\gamma(D)=0 \quad \text{if and only if} \quad (id-\varphi)(D)=0.
$$
Applying this result  in our setting when
$\varphi=\Phi_{f_1,T_1}$, $\gamma=m_1$, and $D= (id-\Phi_{f_2,T_2})^{m_2}\cdots (id-\Phi_{f_k,T_k})^{m_k}(I)\geq 0$, we deduce that   relation ${\bf \Delta_{f,T}^m}(I)= 0$ is equivalent to
$(id-\Phi_{f_1,T_1})(D)=0$. Due to the commutativity of $\Phi_{f_1,T_1},\cdots, \Phi_{f_k,T_k}$, the latter equality is equivalent to
$(id-\Phi_{f_2,T_2})^{m_2}(\Lambda)=0$, where
 $\Lambda:=(id-\Phi_{f_3,T_3})^{m_3}\cdots (id-\Phi_{f_k,T_k})^{m_k}(id-\Phi_{f_1,T_1})(I)\geq 0$.
 Applying again the result mentioned above,  we deduce that the latter equality is equivalent to $(id-\Phi_{f_2,T_2})(\Lambda)=0$. Continuing this process, we can complete the proof of part (iv).
\end{proof}

\bigskip

\section{ Noncommutative  Berezin transforms and universal models}

Noncommutative Berezin transforms are  used  to
  show  that each   polydomain  ${\bf D_f^m}(\cH)$ has a universal model ${\bf W}=\{{\bf W}_{i,j}\}$ consisting of weighted shifts acting on a tensor product of full Fock spaces.

 Let $H_{n_i}$ be
an $n_i$-dimensional complex  Hilbert space with orthonormal basis
$e_1^i,\dots,e_{n_i}^i$.
  We consider the full Fock space  of $H_{n_i}$ defined by
$$F^2(H_{n_i}):=\bigoplus_{p\geq 0} H_{n_i}^{\otimes p},$$
where $H_{n_i}^{\otimes 0}:=\CC 1$ and $H_{n_i}^{\otimes p}$ is the
(Hilbert) tensor product of $p$ copies of $H_{n_i}$. Set $e_\alpha^i :=
e^i_{j_1}\otimes \cdots \otimes e^i_{j_p}$ if
$\alpha=g^i_{j_1}\cdots g^i_{j_p}\in \FF_{n_i}^+$
 and $e^i_{g^i_0}:= 1\in \CC$.
It is  clear that $\{e^i_\alpha:\alpha\in\FF_{n_i}^+\}$ is an orthonormal
basis of $F^2(H_{n_i})$.

Let $m_i, n_i\in \NN:=\{1,2,\ldots\}$, $i\in\{1,\ldots,k\}$, and  $j\in\{1,\ldots,n_i\}$. We  define
 the {\it weighted left creation  operators} $W_{i,j}:F^2(H_{n_i})\to
F^2(H_{n_i})$,   associated with the abstract noncommutative
domain  ${\bf D}_{f_i}^{m_i} $    by setting
\begin{equation}\label{Wij}
W_{i,j} e_\alpha^i:=\frac {\sqrt{b_{i,\alpha}^{(m_i)}}}{\sqrt{b_{i,g_j
\alpha}^{(m_i)}}} e^i_{g_j \alpha}, \qquad \alpha\in \FF_{n_i}^+,
\end{equation}
where
\begin{equation} \label{b-coeff}
  b_{i,g_0}^{(m_i)}:=1\quad \text{ and } \quad
 b_{i,\alpha}^{(m_i)}:= \sum_{p=1}^{|\alpha|}
\sum_{{\gamma_1,\ldots,\gamma_p\in \FF_{n_i}^+}\atop{{\gamma_1\cdots \gamma_p=\alpha }\atop {|\gamma_1|\geq
1,\ldots, |\gamma_p|\geq 1}}} a_{i,\gamma_1}\cdots a_{i,\gamma_p}
\left(\begin{matrix} p+m_i-1\\m_i-1
\end{matrix}\right)
\end{equation}
for all  $ \alpha\in \FF_{n_i}^+$ with  $|\alpha|\geq 1$.

\begin{lemma} \label{univ-model}
For each $i\in \{1,\ldots, k\}$ and $j\in \{1,\ldots, n_i\}$, we
define the operator ${\bf W}_{i,j}$ acting on the tensor Hilbert space
$F^2(H_{n_1})\otimes\cdots\otimes F^2(H_{n_k})$ by setting
$${\bf W}_{i,j}:=\underbrace{I\otimes\cdots\otimes I}_{\text{${i-1}$
times}}\otimes W_{i,j}\otimes \underbrace{I\otimes\cdots\otimes
I}_{\text{${k-i}$ times}},
$$
where the operators $W_{i,j}$ are defined by relation \eqref{Wij}.
 If  ${\bf W}_i:=({\bf W}_{i,1},\ldots,{\bf W}_{i,n_i})$, then the following statements hold.
 \begin{enumerate}

\item[(i)] $(id-\Phi_{f_1,{\bf W}_1})^{m_1}\cdots (id-\Phi_{f_k,{\bf W}_k})^{m_k}(I)={\bf P}_\CC$, where ${\bf P}_\CC$ is the
 orthogonal projection from $\otimes_{i=1}^k F^2(H_{n_i})$ onto $\CC 1\subset \otimes_{i=1}^k F^2(H_{n_i})$, where $\CC 1$ is identified with $\CC 1\otimes\cdots \otimes \CC 1$.
     \item[(ii)]  ${\bf W}:=({\bf W}_1,\ldots, {\bf W}_k)$ is  a pure $k$-tuple  in the
noncommutative polydomain $ {\bf D_f^m}(\otimes_{i=1}^kF^2(H_{n_i}))$.
 \end{enumerate}
 \end{lemma}

 \begin{proof} Note that, due to relation \eqref{Wij},  for each $i\in\{1,\ldots, k\}$ and $\beta_i\in \FF_{n_i}^+$, we have

\begin{equation*}
  W_{i,\beta_i}W_{i,\beta_i}^*
e^i_{\alpha_i} =\begin{cases} \frac
{ {b_{i,\gamma_i}^{(m_i)}}}{ {b_{i,\alpha_i}^{(m_i)}}}\,e^i_{\alpha_i}& \text{ if
}
\alpha_i=\beta_i\gamma_i, \ \, \gamma_i\in \FF_{n_i}^+ \\
0& \text{ otherwise. }
\end{cases}
\end{equation*}
As in Lemma 1.2 from \cite{Po-Berezin}, straightforward computations reveal that   $(id-\Phi_{f_i,{\bf W}_i})^{m_i}(I)=I\otimes\cdots \otimes I\otimes P_\CC\otimes I\otimes \cdots \otimes I$, where $P_\CC$ is  on the $i^{\text{\rm th}}$ position  and denotes the
 orthogonal projection from $F^2(H_{n_i})$ onto $\CC 1\subset F^2(H_{n_i})$.
 Since  the $k$-tuple ${\bf W}:=({\bf W}_1,\ldots, {\bf W}_k)$ is   doubly commuting, we deduce that
$$(id-\Phi_{f_1,{\bf W}_1})^{m_1}\cdots (id-\Phi_{f_k,{\bf W}_k})^{m_k}(I)=(id-\Phi_{f_1,{\bf W}_1})^{m_1}(I)\cdots (id-\Phi_{f_k,{\bf W}_k})^{m_k}(I)={\bf P}_\CC.$$
which proves part (i). To prove part (ii), note first that  relation \eqref{Wij} implies $\Phi_{f_i,{ W}_i}^p(I) e_\alpha^i=0$ if $p>|\alpha|$, $\alpha\in \FF_{n_i}^+$. Since $\|\Phi_{f_i,{\bf W}_i}^p(I)\|\leq 1$ for any $p\in \NN$, we deduce that
 $\lim\limits_{p\to\infty} \Phi^p_{f_i,{\bf W}_i}(I)=0$ in the strong operator
 topology. Taking into account that ${\bf \Delta_{f,W}^m}(I)={\bf P}_\CC$, we can use Proposition \ref{pure} to conclude that  ${\bf W}$ is in  the
noncommutative polydomain $ {\bf D_f^m}(\otimes_{i=1}^kF^2(H_{n_i}))$. Moreover, due to Proposition \ref{pure2}, ${\bf W}$ is a pure $k$-tuple in $ {\bf D_f^m}(\otimes_{i=1}^kF^2(H_{n_i}))$.
\end{proof}

 We mention that one  can   define the {\it weighted right
creation operators} $\Lambda_{i,j}:F^2(H_{n_i})\to F^2(H_{n_i})$ by setting
\begin{equation*}
\Lambda_{i,j} e_\alpha^i:=\frac {\sqrt{b_{i,\alpha}^{(m_i)}}}{\sqrt{b_{i,
\alpha g_j}^{(m_i)}}} e^i_{\alpha g_j}, \qquad \alpha\in \FF_{n_i}^+.
\end{equation*}
As in Lemma \ref{univ-model}, it turns out that
  ${\bf \Lambda}:=({\bf \Lambda}_1,\ldots, {\bf \Lambda}_k)$ is  a pure $k$-tuple
   in the
noncommutative polydomain $ {\bf D_{\widetilde f}^m}(\otimes_{i=1}^kF^2(H_{n_i}))$,
 where ${\bf \widetilde f}=(\widetilde f_1,\ldots, \widetilde f_k)$ with
  $\widetilde{f_i}:=\sum_{|\alpha|\geq 1} a_{i,\widetilde \alpha} Z_\alpha$ and
$\widetilde \alpha=g_{j_p}^i\cdots g_{j_1}^i$ denotes the reverse of
$\alpha=g_{j_1}^i\cdots g_{j_p}^i\in \FF_{n_i}^+$.

 Throughout this paper, the $k$-tuple ${\bf W}:=({\bf W}_1,\ldots, {\bf W}_k)$
 of Lemma \ref{univ-model} will be called the {\it universal model} associated
  with the abstract noncommutative
  polydomain ${\bf D_f^m}$. We introduce the {\it noncommutative Berezin kernel} associated with any element
   ${\bf T}=\{T_{i,j}\}$ in the noncommutative polydomain ${\bf D_f^m}(\cH)$ as the operator
   $${\bf K_{f,T}}: \cH \to F^2(H_{n_1})\otimes \cdots \otimes  F^2(H_{n_k}) \otimes  \overline{{\bf \Delta_{f,T}^m}(I) (\cH)}$$
   defined by
   $$
   {\bf K_{f,T}}h:=\sum_{\beta_i\in \FF_{n_i}^+, i=1,\ldots,k}
   \sqrt{b_{1,\beta_1}^{(m_1)}}\cdots \sqrt{b_{k,\beta_k}^{(m_k)}}
   e^1_{\beta_1}\otimes \cdots \otimes  e^k_{\beta_k}\otimes {\bf \Delta_{f,T}^m}(I)^{1/2} T_{1,\beta_1}^*\cdots T_{k,\beta_k}^*h,
   $$
where the defect operator is defined by
$$
{\bf \Delta_{f,T}^m}(I)  :=(id-\Phi_{f_1,T_1})^{m_1}\cdots (id-\Phi_{f_k,T_k})^{m_k}(I),
$$
and the coefficients  $b_{1,\beta_1}^{(m_1)}, \ldots, b_{k,\beta_k}^{(m_k)}$
are given by relation \eqref{b-coeff}. The fact that ${\bf K_{f,T}}$ is a well-defined bounded operator will be proved  in the next theorem.

\begin{theorem} \label{Berezin-prop}
The noncommutative Berezin kernel associated with a $k$-tuple
${\bf T}=({ T}_1,\ldots, { T}_k)$ in the noncommutative polydomain ${\bf D_f^m}(\cH)$ has the following properties.
\begin{enumerate}
\item[(i)] ${\bf K_{f,T}}$ is a contraction  and
$$
{\bf K_{f,T}^*}{\bf K_{f,T}}=
\lim_{q_k\to\infty}\ldots \lim_{q_1\to\infty}  (id-\Phi_{f_k,T_k}^{q_k})\cdots (id-\Phi_{f_1,T_1}^{q_1})(I),
$$
where the limits are in the weak  operator topology.
\item[(ii)]  If ${\bf T}$ is pure,   then
$${\bf K_{f,T}^*}{\bf K_{f,T}}=I_\cH. $$
\item[(iii)]  For any $i\in \{1,\ldots, k\}$ and $j\in \{1,\ldots, n_i\}$,  $${\bf K_{f,T}} { T}^*_{i,j}= ({\bf W}_{i,j}^*\otimes I)  {\bf K_{f,T}}.
    $$
\end{enumerate}
\end{theorem}
\begin{proof}
 Let ${\bf T}=({ T}_1,\ldots, { T}_k)$ be in the noncommutative polydomain ${\bf D_f^m}(\cH)$ and let $X\in B(\cH)$ be a positive operator such that
   $$
  {\bf \Delta_{f,T}^p}(X)  := (id-\Phi_{f_1,T_1})^{p_1}\cdots (id-\Phi_{f_k,T_k})^{p_k}(X)\geq 0
    $$
 for any  ${\bf p}:=(p_1,\ldots, p_k)\in \ZZ_+^k$ with $p_i\in \{0,1,\ldots, m_i\}$ and
       $i\in \{1,\ldots,k\}$.
Fix $i\in \{1,\ldots,k\}$ and assume that $1\leq p_i\leq m_i$.  Then, due to the commutativity of  $\Phi_{f_1,T_1},\ldots, \Phi_{f_k,T_k}$, we have
$$
(id-\Phi_{f_i,T_i}){\bf \Delta_{f,T}^{p-e_i}}(X) ={\bf \Delta_{f,T}^{p}}(X)\geq 0,
$$
where $\{{\bf e}_i\}_{i=1}^k$ is the canonical basis in $\CC^k$. Hence, and using Lemma \ref{Delta-ineq}, we have
$$
0\leq \Phi_{f_i,T_i}({\bf \Delta_{f,T}^{p-e_i}}(X))\leq {\bf \Delta_{f,T}^{p-e_i}}(X)\leq X,
$$
which proves that $\{\Phi_{f_i,T_i}^s({\bf \Delta_{f,T}^{p-e_i}}(X))\}_{s=0}^\infty$ is a decreasing sequence of positive operators which is convergent in the weak operator topology.  Since $\Phi_{f_i,T_i}$  is WOT-continuous on bounded sets and  $\Phi_{f_1,T_1},\ldots, \Phi_{f_k,T_k}$ are commuting,
  we deduce that
\begin{equation}
\label{wot}
\lim_{s\to\infty}\Phi_{f_i,T_i}^s({\bf \Delta_{f,T}^{p-e_i}}(X))
={\bf \Delta_{f,T}^{p-e_i}}\left(\lim\limits_{s\to\infty}\Phi_{f_i,T_i}^s(X)\right).
\end{equation}
 Then we have
 \begin{equation*}
 \begin{split}
 D_i^{(1)}({\bf \Delta_{f,T}^p}(X))&:=\sum_{s=0}^\infty \Phi_{f_i,T_i}^{s}({\bf \Delta_{f,T}^p}(X))
 =\sum_{s=0}^\infty \Phi_{f_i,T_i}^{s}\left[ {\bf \Delta_{f,T}^{p-e_i}}(X)
 -\Phi_{f_i, T_i}({\bf \Delta_{f,T}^{p-e_i}}(X))\right]\\
 &={\bf \Delta_{f,T}^{p-e_i}}(X)- \lim_{q_1\to \infty}
 \Phi_{f_i,T_i}^{q_1}({\bf \Delta_{f,T}^{p-e_i}}(X))\leq {\bf \Delta_{f,T}^{p-e_i}}(X)\leq X.
 \end{split}
 \end{equation*}
 Due to relation \eqref{wot}, we deduce that
 $$
 0\leq D_i^{(1)}({\bf \Delta_{f,T}^p}(X))={\bf \Delta_{f,T}^{p-e_i}}\left(X- \lim_{q_1\to \infty}
 \Phi_{f_i,T_i}^{q_1}(X)\right), \qquad {\bf p}\leq {\bf m}, 1\leq p_i.
 $$
Define
$D_{i}^{(j)}({\bf \Delta_{f,T}^p}(X))
:=\sum_{s=0}^\infty \Phi_{f_i, T_i}^s(D_{i}^{(j-1)}({\bf \Delta_{f,T}^p}(X)))$,
 where $j=2,\ldots p_i$.
Inductively, we can prove that
\begin{equation}
\label{induct}
0\leq D_{i}^{(j)}({\bf \Delta_{f,T}^p}(X))={\bf \Delta_{f,T}^{ p-{\it j}e_i}}
\left(X- \lim_{q_j\to \infty}
 \Phi_{f_i,T_i}^{q_j}(X)\right)\leq {\bf \Delta_{f,T}^{ p-{\it j}e_i}}(X)\leq X,
 \qquad j\leq p_i.
\end{equation}
Indeed, if $j\leq p_i-1$ and  setting $Y:=X- \lim\limits_{q_j\to \infty}
 \Phi_{f_i,T_i}^{q_j}(X)$, relation
\eqref{induct} implies
\begin{equation*}
\begin{split}
D_{i}^{(j+1)}({\bf \Delta_{f,T}^p}(X))
&=\lim_{q_{j+1}\to \infty}\sum_{s=0}^{q_{j+1}} \Phi_{f_i, T_i}^s\left[{\bf \Delta_{f,T}^{ p-{\it j}e_i}}(Y)\right] \\
&=
{\bf \Delta_{f,T}^{ p-\text{$(j+ 1)$}e_i}}\left[
Y- \lim_{q_{j+1}\to \infty} \Phi_{f_i,T_i}^{q_{j+1}}(Y)\right]\\
&=
{\bf \Delta_{f,T}^{ p-\text{$(j+ 1)$}e_i}}(Y)-{\bf \Delta_{f,T}^{ p-\text{$(j+ 1)$}e_i}}\left( \lim_{q_{j+1}\to \infty} \Phi_{f_i,T_i}^{q_{j+1}}(Y) \right).
\end{split}
\end{equation*}
On the other hand, we have
 \begin{equation*}
 \begin{split}
 \lim_{q_{j+1}\to \infty} \Phi_{f_i,T_i}^{q_{j+1}}(Y)
 &=\lim_{q_{j+1}\to \infty} \Phi_{f_i,T_i}^{q_{j+1}}
 \left(X- \lim\limits_{q_j\to \infty}
 \Phi_{f_i,T_i}^{q_j}(X)\right)\\
 &=
 \lim_{q_{j+1}\to \infty} \Phi_{f_i,T_i}^{q_{j+1}}(X)-
 \lim_{q_{j+1}\to \infty}\lim_{q_{j}\to \infty} \Phi_{f_i,T_i}^{q_{j+1}}\left( \Phi_{f_i,T_i}^{q_{j}}(X)\right)=0.
 \end{split}
 \end{equation*}
Combining these results, we obtain
$$
D_{i}^{(j+1)}({\bf \Delta_{f,T}^p}(X)) =
{\bf \Delta_{f,T}^{ p-\text{$(j+ 1)$}e_i}}\left( X- \lim\limits_{q_j\to \infty}
 \Phi_{f_i,T_i}^{q_j}(X)\right)\leq {\bf \Delta_{f,T}^{ p-\text{$(j+ 1)$}e_i}}(X)\leq X,
$$
for any  ${\bf p}:=(p_1,\ldots, p_k)\in \ZZ_+^k$ with ${\bf p}\leq {\bf m}$ and  $p_i\geq 1$,
which proves our assertion. When $j=p_i$, relation \eqref{induct} becomes
\begin{equation*}
\label{pi}
0\leq D_{i}^{(p_i)}({\bf \Delta_{f,T}^p}(X)) =
{\bf \Delta_{f,T}^{ p-\text{$p_i$}e_i}}\left( X- \lim\limits_{q\to \infty}
 \Phi_{f_i,T_i}^{q}(X)\right)\leq X.
\end{equation*}

On the other hand, taking into account that we can rearrange WOT-convergent series of positive operators, we deduce that, for each $d\in \NN$,
\begin{equation*}
\begin{split}
\Phi_{f_i, T_i}^d({\bf \Delta_{f,T}^p}(X))
 &=\sum_{\alpha_1\in \FF_{n_i}^+, |\alpha_1|\geq 1} a_{i,\alpha_1} T_{i,\alpha_1}\left( \cdots
\sum_{\alpha_d\in \FF_{n_1}^+, |\alpha_d|\geq 1} a_{i,\alpha_d} T_{i,\alpha_d}{\bf \Delta_{f,T}^p}(X)T_{i,\alpha_d}^*\cdots
       \right)T_{i,\alpha_1}^*\\
     &\qquad =\sum_{\gamma\in \FF_{n_i}^+, |\gamma|\geq d}
     \sum_{{\alpha_1,\ldots,\alpha_d\in \FF_{n_i}^+}\atop{{\alpha_1\cdots \alpha_d=\gamma }\atop {|\alpha_1|\geq
1,\ldots, |\alpha_d|\geq 1}}} a_{i,\alpha_1}\cdots a_{i,\alpha_d} T_{i,\gamma}{\bf \Delta_{f,T}^p}(X)T_{i,\gamma}^*
\end{split}
\end{equation*}
and
\begin{equation*}
\begin{split}
D_{i}^{(1)}({\bf \Delta_{f,T}^p}(X))&=\sum_{s=0}^\infty
 \Phi_{f_i,T_i}^{s}({\bf \Delta_{f,T}^p}(X))\\
&
=
{\bf \Delta_{f,T}^p}(X)+\sum_{\gamma\in \FF_{n_i}^+, |\gamma|\geq 1}
\left(\sum_{d=1}^{|\gamma|}
     \sum_{{\alpha_1,\ldots,\alpha_d\in \FF_{n_i}^+}\atop{{\alpha_1\cdots \alpha_d
     =\gamma }\atop {|\alpha_1|\geq
1,\ldots, |\alpha_d|\geq 1}}} a_{i,\alpha_1}\cdots
 a_{i,\alpha_d}\right) T_{i,\gamma}{\bf \Delta_{f,T}^p}(X)T_{i,\gamma}^*.
\end{split}
\end{equation*}
 Since $D_{i}^{(j)}({\bf \Delta_{f,T}^p}(X)):=\sum_{s=0}^\infty
  \Phi_{f_i,T_i}^{s}(D_{i}^{(j-1)}({\bf \Delta_{f,T}^p}(X)))$ for
   $j=2,\ldots, p_i$, using a combinatorial argument and rearranging
    WOT-convergent series  of positive operators, one can prove by induction
    over $p_i$ that
\begin{equation*}
\begin{split}
D_{i}^{(p_i)}({\bf \Delta_{f,T}^p}(X))&={\bf \Delta_{f,T}^p}(X)+
\sum_{\alpha\in \FF_{n_i}^+, |\alpha|\geq 1}\left(
 \sum_{p=1}^{|\alpha|}
\sum_{{\gamma_1,\ldots,\gamma_p\in \FF_{n_i}^+}\atop{{\gamma_1\cdots \gamma_p=\alpha }\atop {|\gamma_1|\geq
1,\ldots, |\gamma_p|\geq 1}}} a_{i,\gamma_1}\cdots a_{i,\gamma_p}
\left(\begin{matrix} p+p_i-1\\p_i-1
\end{matrix}\right)\right)T_{i,\alpha} {\bf \Delta_{f,T}^p}(X) T_{i,\alpha}^*\\
&
=\sum_{\alpha\in \FF_{n_i}^+} b_{i,\alpha}^{(p_i)} T_{i,\alpha} {\bf \Delta_{f,T}^p}(X) T_{i,\alpha}^*.
\end{split}
\end{equation*}
For each $i\in \{1,\ldots,k\}$, let $\Omega_i\subset B(\cH)$ be the
 set of all $Y\in B(\cH)$, $Y\geq0$, such that
the series $\sum_{\beta_i\in \FF_{n_i}^+} b_{i,\beta_i}^{(m_i)} T_{i,\beta_i}Y T_{i,\beta_i}^*$ is convergent in the weak operator topology,
where
\begin{equation*}
  b_{i,g_0}^{(m_i)}:=1\quad \text{ and } \quad
 b_{i,\alpha}^{(m_i)}:= \sum_{p=1}^{|\alpha|}
\sum_{{\gamma_1,\ldots,\gamma_p\in \FF_{n_i}^+}\atop{{\gamma_1\cdots \gamma_p=\alpha }\atop {|\gamma_1|\geq
1,\ldots, |\gamma_p|\geq 1}}} a_{i,\gamma_1}\cdots a_{i,\gamma_p}
\left(\begin{matrix} p+m_i-1\\m_i-1
\end{matrix}\right)
\end{equation*}
for all  $ \alpha\in \FF_{n_i}^+$ with  $|\alpha|\geq 1$.
We define the map $\Psi_i:\Omega_i\to B(\cH)$ by setting
\begin{equation*}
\label{Psi-i}
\Psi_i(Y):=\sum_{\beta_i\in \FF_{n_i}^+} b_{i,\beta_i}^{(m_i)} T_{i,\beta_i}Y T_{i,\beta_i}^*.
\end{equation*}
Due to the  results above,  we have
\begin{equation}
\label{Psi}
\begin{split}
0\leq \Psi_i( {\bf \Delta_{f,T}^p}(X))&=D_{i}^{(m_i)}({\bf \Delta_{f,T}^p}(X))\\
&={\bf \Delta_{f,T}^{ p-\text{$m_i$}e_i}}
\left(id-\lim_{q_{i}\to \infty} \Phi_{f_i,T_i}^{q_{i}} \right)(X)\\
&
\leq {\bf \Delta_{f,T}^{ p-\text{$m_i$}e_i}}(X)\leq X,
\end{split}
\end{equation}
for any  ${\bf p}:=(p_1,\ldots, p_k)\in \ZZ_+^k$ with ${\bf p}\leq {\bf m}$ and
 $p_i=m_i$.
Since ${\bf T}=({ T}_1,\ldots, { T}_k)$ is in the noncommutative polydomain ${\bf D_f^m}(\cH)$, Theorem  \ref{radial}
implies
   $$
  {\bf \Delta_{f,T}^p}(I)  := (id-\Phi_{f_1,T_1})^{p_1}\cdots (id-\Phi_{f_k,T_k})^{p_k}(I)\geq 0
    $$
 for any  ${\bf p}:=(p_1,\ldots, p_k)\in \ZZ_+^k$ with $p_i\in \{0,1,\ldots, m_i\}$ and
       $i\in \{1,\ldots,k\}$.
Applying relation \eqref{Psi} in the  particular case when $i=1$, $p_1=m_1$, and $X=I$, we have
$$
0\leq \Psi_1( {\bf \Delta_{f,T}^{p'}}(I))=D_{1}^{(m_1)}({\bf \Delta_{f,T}^{p'}}(I))
={\bf \Delta_{f,T}^{ p'-\text{$m_1$}e_1}}\left(I-\lim_{q_{1}\to \infty} \Phi_{f_1,T_1}^{q_{1}} (I)\right)
\leq {\bf \Delta_{f,T}^{ p'-\text{$m_1$}e_1}}(I)\leq I
$$
for any ${\bf p}'=(m_1,p_2,\ldots, p_k)$ with ${\bf p}'\leq {\bf m}$. Hence and
using again relation \eqref{Psi}, when $i=2$, ${\bf p}=(0,m_2, p_3\ldots, p_k)$, and
$X=\lim_{q_{1}\to \infty} \left(id-\Phi_{f_1,T_1}^{q_{1}} \right)(I)\geq  0$, we obtain
\begin{equation*}
\begin{split}
0\leq \Psi_2(\Psi_1( {\bf \Delta_{f,T}^{p''}}(I)))&=\Psi_2\left({\bf \Delta_{f,T}^{ p''-\text{$m_1$}e_1}}\left(I-\lim_{q_{1}\to \infty} \Phi_{f_1,T_1}^{q_{1}} (I)\right)\right)\\
&={\bf \Delta_{f,T}^{ p''-\text{$m_1$}e_1-\text{$m_2$}e_2}}
\lim_{q_{2}\to \infty}\lim_{q_{1}\to \infty}\left(id- \Phi_{f_2,T_2}^{q_{2}}\right)
\left(id- \Phi_{f_1,T_1}^{q_{1}}\right)(I)\\
&\leq {\bf \Delta_{f,T}^{ p''-\text{$m_1$}e_1-\text{$m_2$}e_2}}(I)\leq I
\end{split}
\end{equation*}
for any ${\bf p}''=(m_1,m_2, p_3,\ldots, p_k)$. Continuing  this process, a repeated application of \eqref{Psi}, leads to the relation
\begin{equation*}
\label{PsiPsi}
0\leq (\Psi_k\circ\cdots \circ \Psi_1)({\bf \Delta_{f,T}^m}(I))=
\lim_{q_k\to\infty}\ldots \lim_{q_1\to\infty}  (id-\Phi_{f_k,T_k}^{q_k})\cdots (id-\Phi_{f_1,T_1}^{q_1})(I)\leq I,
\end{equation*}
where ${\bf m}=(m_1,\ldots,m_k)$.
To prove item (i), note  that the results above imply

 \begin{equation*}
\begin{split}
\|{\bf K_{f,T}}h\|^2 &=
\sum_{\beta_k\in \FF_{n_k}}\cdots \sum_{\beta_1\in \FF_{n_1}}  b_{1,\beta_1}^{(m_1)}\cdots b_{k,\beta_k}^{(m_k)} \left<T_{k,\beta_k}\cdots T_{1,\beta_1}{\bf \Delta_{f,T}^m}(I) T_{1,\beta_1}^*\cdots T_{k,\beta_k}^*h, h\right>\\
&=\left< (\Psi_k\circ\cdots \circ\Psi_1)({\bf \Delta_{f,T}^m}(I))h,h\right>\leq \|h\|^2
\end{split}
\end{equation*}
for any  $h\in \cH$, and
$$
{\bf K_{f,T}^*}{\bf K_{f,T}}=
\lim_{q_k\to\infty}\ldots \lim_{q_1\to\infty}  (id-\Phi_{f_k,T_k}^{q_k})\cdots (id-\Phi_{f_1,T_1}^{q_1})(I).
$$
Now, item (ii) is clear.
To prove part (iii),  note that
\begin{equation}\label{WbWb}
  W_{i,j}^*
e^i_{\beta_i} =\begin{cases} \frac
{\sqrt{b_{i,\gamma_i}^{(m_i)}}}{\sqrt{b_{i,\beta_i}^{(m_i)}}}\,e^i_{\gamma_i}& \text{ if
}
\beta_i=g_j^i\gamma_i, \ \, \gamma_i\in \FF_{n_i}^+ \\
0& \text{ otherwise }
\end{cases}
\end{equation}
 for any   $\beta_i\in \FF_{n_i}^+$ and
 $j\in \{1,\ldots, n_i\}$.
 Hence, and using the definition of the noncommutative Berezin kernel, we have
 \begin{equation*}
 \begin{split}
  &({\bf W}_{i,j}^*\otimes I){\bf K_{f,T}}h\\
  & =\sum_{\beta_p\in \FF_{n_p}^+, p\in\{1,\ldots,k\}}
   \sqrt{b_{1,\beta_1}^{(m_1)}}\cdots \sqrt{b_{k,\beta_k}^{(m_k)}}
   e^1_{\beta_1}\otimes \cdots \otimes e^{i-1}_{\beta_{i-1}}\otimes W_{i,j}^*e^i_{\beta_i}\otimes e^{i+1}_{\beta_{i+1}}\otimes \cdots \otimes  e^k_{\beta_k}\otimes {\bf \Delta_{f,T}^m}(I)^{1/2} T_{1,\beta_1}^*\cdots T_{k,\beta_k}^*h\\
   & =\sum_{{\beta_p\in \FF_{n_p}^+, p\in \{1,\ldots,k\}\backslash\{i\}} \atop{\gamma_i\in \FF_{n_i}}}
   \sqrt{b_{1,\beta_1}^{(m_1)}}\cdots \sqrt{b_{i,\gamma_i}^{(m_i)}}\cdots \sqrt{b_{k,\beta_k}^{(m_k)}}
   e^1_{\beta_1}\otimes \cdots \otimes  e^{i-1}_{\beta_{i-1}}\otimes e^i_{\gamma_i}\otimes e^{i+1}_{\beta_{i+1}} \otimes \cdots \otimes e^k_{\beta_k}\\
   &\qquad \qquad\qquad \otimes {\bf \Delta_{f,T}^m}(I)^{1/2} T_{1,\beta_1}^*\cdots  T_{i-1,\beta_i}^* T_{i,g_j^i \gamma_i}^* T_{i+1, \beta_{i+1}}^*\cdots T_{k,\beta_k}^*h
 \end{split}
 \end{equation*}
for any $h\in \cH$. Using the commutativity of  the tuples $T_1,\ldots,T_k$,
we deduce that
$$
({\bf W}_{i,j}^*\otimes I){\bf K_{f,T}}={\bf K_{f,T}} T_{i,j}^*
$$
for any $i\in \{1,\ldots, k\}$ and $j\in \{1,\ldots, n_i\}$.
The proof is complete.
\end{proof}

\begin{remark} If ${\bf \Delta_{f,T}^m}(I)=0$, then ${\bf K_{f,T}}=0$ and
$(id-\Phi_{f_k,T_k})\cdots (id-\Phi_{f_1,T_1})(I)=0$.

\end{remark}

We can define now the {\it  noncommutative Berezin transform}
 at   ${\bf T}\in {\bf D_f^m}(\cH)$ to be the mapping
 $ {\bf B_{T}}: B(\otimes_{i=1}^k F^2(H_{n_i}))\to B(\cH)$
 given by

 \begin{equation*}
 \label{def-Be2}
 {\bf B_{T}}[g]:= {\bf K^*_{f,T}} (g\otimes I_\cH) {\bf K_{f,T}},
 \qquad g\in B(\otimes_{i=1}^k F^2(H_{n_i})).
 \end{equation*}

We denote by $\cP({\bf W})$ the set of all polynomials $p({\bf W}_{i,j})$  in  the operators ${\bf W}_{i,j}$, $i\in \{1,\ldots, k\}$, $j\in \{1,\ldots, n_i\}$,  and the identity. We introduce the polydomain algebra $\cA({\bf D_f^m})$ to be the norm closed algebra generated by ${\bf W}_{i,j}$ and the identity.

\begin{theorem}\label{Poisson-C*}
Let ${\bf T}=\{{T}_{i,j}\}$ be in  the noncommutative polydomain ${\bf D_f^m}(\cH)$  and let
  $$\cS:=\overline{\text{\rm  span}} \{ p({\bf W}_{i,j})q({\bf W}_{i,j})^*:\
p({\bf W}_{i,j}),q({\bf W}_{i,j}) \in  \cP({\bf W})\},
$$
where the closure is in the operator norm.
   Then there is
    a unital completely contractive linear map
$
\Psi_{{\bf f,T}}:\cS \to B(\cH)
$
such that
\begin{equation*}
\Psi_{{\bf f,T}}(g)=\lim_{r\to 1} {\bf B}_{r{\bf T}}[g],\qquad g\in \cS,
\end{equation*}
where the limit exists in the norm topology of $B(\cH)$, and
$$
\Psi_{\bf f,T}\left(\sum_{\gamma=1}^s p_\gamma({\bf W}_{i,j})q_\gamma({\bf W}_{i,j})^*\right)= \sum_{\gamma=1}^s p_\gamma(T_{i,j})q_\gamma(T_{i,j})^*
$$
for any $p_\gamma({\bf W}_{i,j}),q_\gamma({\bf W}_{i,j}) \in  \cP({\bf W})$ and $ s\in \NN$.
 In particular, the restriction of
$\Psi_{\bf f,T}$ to the polydomain algebra $\cA({\bf D_f^m})$ is a
completely contractive homomorphism.
 If, in addition,
 ${\bf T}$ is a  pure $k$-tuple,
then  $$\lim_{r\to 1} {\bf B}_{r{\bf T}}[g]= {\bf B}_{T}[g],\qquad g\in \cS.$$
\end{theorem}

\begin{proof}  According to Theorem \ref{radial},  $r{\bf T}=(rT_1,\ldots,
 rT_k)\in  {\bf D_f^m}(\cH)$ for   any $r\in (0,1)$.
Since we have
$\Phi_{f_i,rT_i}^n(I)\leq r^n \Phi_{f_i,T_i}^n(I)\leq r^n I$ for any  $n\in \NN$, Proposition \ref{pure2} shows   that $ r{\bf T}$  is a pure $k$-tuple in ${\bf D_f^m}(\cH)$.
 Using  Theorem \ref{Berezin-prop}, we deduce that the noncommutative Berezin kernel
  ${\bf K}_{{\bf f},r{\bf T}}$ is an isometry and
\begin{equation}
\label{K-r}
{\bf K}_{{\bf f},r{\bf T}}^* \left[p({\bf W}_{i,j})q({\bf W}_{i,j})^*\otimes
I_\cH\right]{\bf K}_{{\bf f},r{\bf T}}= p(rT_{i,j})q(rT_{i,j})^*, \qquad p({\bf W}_{i,j}),q({\bf W}_{i,j}) \in  \cP({\bf W}).
\end{equation}
Hence, we obtain the von Neumann \cite{von} type inequality
\begin{equation}
\label{vn1}
\left\|\sum_{\gamma=1}^s p_\gamma(rT_{i,j})q_\gamma(rT_{i,j})^*\right\|
\leq \left\|\sum_{\gamma=1}^s p_\gamma({\bf W}_{i,j})q_\gamma({\bf W}_{i,j})^*\right\|
 \end{equation}
 for any
  $ p_\gamma({\bf W}_{i,j}),q_\gamma({\bf W}_{i,j}) \in  \cP({\bf W})$, $s\in \NN$, and $r\in [0,1]$.
  Fix $g\in \cS $ and  let
$\{\chi_n({\bf W}_{i,j}, {\bf W}_{i,j}^*)\}_{n=0}^\infty$ be a sequence of operators in the span of $\cP({\bf W})\cP({\bf W})^*$  which converges to $g$ in norm, as $n\to\infty$. Define
 $\Psi_{{\bf f,T}}(g):=\lim_{n\to\infty} \chi_n({ T}_{i,j}, { T}_{i,j}^*)$. The inequality
 \eqref{vn1} shows that $\Psi_{{\bf f,T}}(g)$ is  well-defined and
 $\|\Psi_{{\bf f,T}}(g)\|\leq \|g\|$. Using the matrix version of \eqref{K-r},
  we deduce that $\Psi_{{\bf f,T}}$ is  a unital completely contractive linear map.
Now we  prove  that $\Psi_{{\bf f,T}}(g)=\lim_{r\to 1} {\bf B}_{r{\bf T}}[g]$.
    Note that  relation \eqref{K-r}
implies
$$
 \chi_n(rT_i,rT_i^*)={\bf K}_{{\bf f},r{\bf T}}^*(\chi_n({\bf W}_{i,j}, {\bf W}_{i,j}^*)\otimes I_\cH){\bf K}_{{\bf f},r{\bf T}}={\bf B}_{r{\bf T}}[\chi_n({\bf W}_{i,j}, {\bf W}_{i,j}^*)]
$$
for any $n\in \NN$ and $r\in (0,1)$. Using the fact that   $\Psi_{f,rT}(g):=\lim_{n\to\infty}
\chi_n(rT_i,rT_i^*)$ exists in norm,  we deduce that
\begin{equation}
\label{I-tensor} \Psi_{{\bf f},r{\bf T}}(g)={\bf K}_{{\bf f},r{\bf T}}^* (g\otimes I_\cH){\bf K}_{{\bf f},r{\bf T}}={\bf B}_{r{\bf T}}[g].
\end{equation}
Given  $\epsilon>0$ let $s\in \NN$   be such that
$\|\chi_s(W_i,W_i^*)-g\|<\frac{\epsilon}{3}. $ Due to the first part of
the theorem, we have
$$
\|\Psi_{{\bf f},r{\bf T}}(g)-\chi_s(rT_i,rT_i^*)\|\leq
\|g-\chi_s({\bf W}_i,{\bf W}_i^*)\|<\frac{\epsilon}{3}
$$
for any $r\in [0,1]$.
On the other hand, since $\chi_s({\bf W}_i,{\bf W}_i^*)$  has a finite number of
terms, there exists $\delta\in (0,1)$ such that
$$
\|\chi_s(rT_i,rT_i^*)-\chi_s(T_i,T_i^*)\|<\frac{\epsilon}{3}
$$
for any $r\in (\delta,1)$.
Now, using these inequalities and  relation \eqref{I-tensor}, we
deduce that
\begin{equation*}
\begin{split}
\left\|\Psi_{\bf f,T}(g)-{\bf B}_{r{\bf T}}[g]\right\| &=
\left\|
\Psi_{\bf f,T}(g)- \Psi_{{\bf f},r{\bf T}}(g)\right\|\\
&
\leq \left\| \Psi_{\bf f,T}(g)-\chi_s(T_i,
T_i^*)\right\|
  +\left\|\chi_s(T_i, T_i^*)- \chi_s(rT_i,
rT_i^*)\right\|\\
&\qquad
+ \left\|\chi_s(rT_i, rT_i^*)- \Psi_{{\bf f},r{\bf T}}(g)\right\|
< \epsilon
\end{split}
\end{equation*}
for any $r\in (\delta, 1)$, which proves our assertion.
Now, we assume that  ${\bf T}=({ T}_1,\ldots, { T}_k)$ is a pure $k$-tuple in  ${\bf D_f^m}(\cH)$.
Due to Theorem  \ref{Berezin-prop}, we have
$$
{\bf B_{T}}[\chi_n({\bf W}_{i,j}, {\bf W}_{i,j}^*)]:= {\bf K^*_{f,T}} (\chi_n({\bf W}_{i,j}, {\bf W}_{i,j}^*)\otimes I_\cH) {\bf K_{f,T}}=\chi_n({ T}_{i,j}, { T}_{i,j}^*).
$$
Taking into account that $\{\chi_n({\bf W}_{i,j}, {\bf W}_{i,j}^*)\}_{n=0}^\infty$ is a sequence of operators in the span of $\cP({\bf W})\cP({\bf W})^*$  which converges to $g$ in norm, we conclude that
$$
{\bf B_{T}}[g]=\Psi_{\bf f,T}(g)=\lim_{r\to 1} {\bf B}_{r{\bf T}}[g],\qquad g\in \cS.
$$
This
completes the proof.
\end{proof}

We remark that Theorem \ref{Poisson-C*} shows that  the noncommutative polydomain algebra $\cA({\bf D_f^m})$ is the universal algebra generated by the identity and a  doubly commuting $k$-tuple in  the abstract polydomain domain ${\bf D_f^m}$.

  We denote by  $C^*({\bf W}_{i,j})$ the $C^*$-algebra generated by  the operators ${\bf W}_{i,j}$,  where $i\in \{1,\ldots, k\}$,  $j\in \{1,\ldots, n_i\}$, and the identity.

\begin{corollary}\label{C*-charact} Let ${\bf q}=(q_1,\ldots, q_k)$ be a $k$-tuple of   positive regular  noncommutative polynomials and let
$${\bf X}:=(X_1,\ldots, X_k)\in   B(\cH)^{n_1}\times\cdots \times B(\cH)^{n_k}.$$
 Then  ${\bf X}$ is in the
noncommutative polydomain ${\bf D^m_q}(\cH)$ if and only if  there
exists a unital completely positive linear map $\Psi:C^*({\bf W}_{i,j})\to B(\cH)$ such that  $$
\Psi_{\bf q,T}( p({\bf W}_{i,j})r({\bf W}_{i,j})^*)= p(X_{i,j})r(X_{i,j})^*, \qquad
p({\bf W}_{i,j}),r({\bf W}_{i,j}) \in  \cP({\bf W}), $$
where ${\bf W}:=\{{\bf W}_{i,j}\}$  is the  universal model associated with the abstract noncommutative polydomain ${\bf D_q^m}$.
\end{corollary}
\begin{proof}
The direct implication is due to Theorem \ref{Poisson-C*} and
Arveson's extension theorem \cite{Arv-acta}. For the converse,  note that, due to Lemma \ref{univ-model}, Proposition \ref{pure2}, and Proposition \ref{pure}, we have
$$ (I-\Phi_{q_1,{ X}_1})^{p_1}\cdots (I-\Phi_{q_k,{ X}_k})^{p_k}(I)=\Psi_{\bf q,T}\left[ (I-\Phi_{q_1,{\bf W}_1})^{p_1}\cdots (I-\Phi_{q_k,{\bf W}_k})^{p_k}(I)\right]\geq 0
$$
for any $p_i\in \{0,1,\ldots, m_i\}$ and
       $i\in \{1,\ldots,k\}$.    Using now
  Theorem \ref{radial} we can complete the proof.
\end{proof}

We remark that under the condition
  \begin{equation*}
\overline{\text{\rm span}}\,\{ p({\bf W}_{i,j})r({\bf W}_{i,j})^*) :\
 p({\bf W}_{i,j}),r({\bf W}_{i,j}) \in  \cP({\bf W})\}=C^*({\bf W}_{i,j}),
\end{equation*}
 Corollary \ref{C*-charact} shows that  $C^*({\bf W}_{i,j})$ is the universal
  $C^*$-algebra generated by the identity and a  doubly commuting $k$-tuple in
    the abstract polydomain domain ${\bf D_f^m}$. We remark that the condition
     above holds, for example, if  ${\bf D_f^m}(\cH)$ is the noncommutative polyball
 $[B(\cH)^{n_1}]_{1}^-\times_c \cdots \times_c [B(\cH)^{n_k}]_{1}^-$.

\bigskip

\section{Noncommutative Hardy algebras and functional calculus}

 We introduce the noncommutative Hardy algebra $F^\infty({\bf D_f^m})$
  and   provide a WOT-continuous functional calculus for
  completely non-coisometric   tuples
 in  in the noncommutative polydomain ${\bf D_f^m}(\cH)$.

Let $\varphi({\bf W}_{i,j})=\sum\limits_{\beta_1\in \FF_{n_1}^+,\ldots, \beta_k\in \FF_{n_k}^+} c_{\beta_1,\ldots, \beta_k}
{\bf W}_{1,\beta_1}\cdots {\bf W}_{k,\beta_k} $ be a formal sum with $c_{\beta_1,\ldots, \beta_k}\in \CC$ and such that
$$\sum\limits_{\beta_1\in \FF_{n_1}^+,\ldots, \beta_k\in \FF_{n_k}^+} |c_{\beta_1,\ldots, \beta_k}|^2 \frac{1}{b_{1,\beta_1}^{(m_1)}\cdots b_{k,\beta_k}^{(m_k)}}<\infty.
$$
We prove that
 $\varphi({\bf W}_{i,j})(e_{\gamma_1}^1\otimes\cdots \otimes e_{\gamma_k}^k)$
 is in $\otimes_{i=1}^k F^2(H_{n_i})$,
for any $\gamma_1\in \FF_{n_1}^+,\ldots, \gamma_k\in \FF_{n_k}^+$.
 Indeed, due to relation \eqref{Wij},
    we have
\begin{equation*}
\begin{split}
\sum\limits_{\beta_1\in \FF_{n_1}^+,\ldots, \beta_k\in \FF_{n_k}^+} c_{\beta_1,\ldots, \beta_k}
{\bf W}_{1,\beta_1}&\cdots {\bf W}_{k,\beta_k}
(e_{\gamma_1}^1\otimes\cdots \otimes e_{\gamma_k}^k)\\
& =
\sum_{\beta_1\in \FF_{n_1}^+,\ldots, \beta_k\in \FF_{n_k}^+}
c_{\beta_1,\ldots, \beta_k}\sqrt{\frac{b_{1,\gamma_1}
^{(m_1)}}{b_{1,\beta_1\gamma_1}^{(m_1)}}}\cdots \sqrt{\frac{b_{k,\gamma_k}
^{(m_k)}}{b_{k,\beta_k\gamma_k}^{(m_k)}}} e_{\beta_1
\gamma_1}^1 \otimes\cdots \otimes  e_{\beta_k \gamma_k}^k.
\end{split}
\end{equation*}
Let $i\in \{1,\ldots,k\}$ and $\alpha, \beta\in \FF_{n_i}$ be such that   $ |\alpha|\geq 1$ and
$|\beta|\geq 1$.  Note that,  for any $j\in\{1,\ldots, |\alpha|\}$ and
 $k\in\{1,\ldots,|\beta|\}$,
$$
\left(\begin{matrix} j+m_i-1\\m_i-1
\end{matrix}\right)
\left(\begin{matrix} k+m_i-1\\m_i-1
\end{matrix}\right)
\leq C_{i,|\beta|}^{( m_i)} \left(\begin{matrix} j+k+m_i-1\\m_i-1
\end{matrix}\right),
$$
where $C_{i,|\beta|}^{( m_i)}:=\left(\begin{matrix} |\beta|+m_i-1\\m_i-1
\end{matrix}\right)$.
Using relation \eqref{b-coeff}
and comparing the product
$b_{i,\alpha}^{(m_i)} b_{i, \beta}^{(m_i)} $ with
$b_{i,\alpha \beta}^{(m)}$, we deduce that
 \begin{equation}
\label{Mb} b_{i,\alpha}^{(m_i)} b_{i,\beta}^{(m_i)}\leq C_{i,|\beta|}^{( m_i)} b_{i,\alpha
\beta}^{(m_i)} \quad \text{ and } \quad  b_{i,\alpha}^{(m_i)} b_{i,\beta}^{(m_i)}\leq C_{i,|\alpha|}^{(m_i)} b_{i,\alpha \beta}^{(m_i)}
\end{equation}
for any $\alpha, \beta\in \FF_{n_i}^+$.
 Hence,  we deduce that
\begin{equation*}
\begin{split}
\sum_{\beta_1\in \FF_{n_1}^+,\ldots, \beta_k\in \FF_{n_k}^+}
&|c_{\beta_1,\ldots, \beta_k}|^2\frac{b_{1,\gamma_1}
^{(m_1)}}{b_{1,\beta_1\gamma_1}^{(m_1)}}\cdots \frac{b_{k,\gamma_k}
^{(m_k)}}{b_{k,\beta_k\gamma_k}^{(m_k)}} \\
&\leq C_{1,|\gamma_1|}^{(
m_1)}\cdots C_{k,|\gamma_k|}^{(
m_k)}\sum_{\beta_1\in \FF_{n_1}^+,\ldots, \beta_k\in \FF_{n_k}^+}
|c_{\beta_1,\ldots, \beta_k}|^2
\frac{1}{b_{1,\beta_1}^{(m_1)}\cdots b_{k,\beta_k}^{(m_k)}}<\infty,
\end{split}
\end{equation*}
which proves our assertion.
Let $\cP$ be the linear span of  the vectors   $e_{\gamma_1}\otimes\cdots \otimes e_{\gamma_k}$  for  $\gamma_1\in \FF_{n_1}^+,\ldots, \gamma_k\in \FF_{n_k}^+$.
If
$$
\sup_{p\in\cP, \|p\|\leq 1} \left\|\sum\limits_{\beta_1\in \FF_{n_1}^+,\ldots, \beta_k\in \FF_{n_k}^+} c_{\beta_1,\ldots, \beta_k}
{\bf W}_{1,\beta_1}\cdots {\bf W}_{k,\beta_k}(p)\right\|<\infty,
$$
then there is a unique bounded operator acting on $F^2(H_{n_1})\otimes\cdots \otimes F^2(H_{n_k})$, which
we denote by $\varphi({\bf W}_{i,j})$, such that
$$
\varphi({\bf W}_{i,j})p=\sum\limits_{\beta_1\in \FF_{n_1}^+,\ldots, \beta_k\in \FF_{n_k}^+} c_{\beta_1,\ldots, \beta_k}
{\bf W}_{1,\beta_1}\cdots {\bf W}_{k,\beta_k}(p)\quad \text{ for any } \ p\in \cP.
$$
The set of all operators $\varphi({\bf W}_{i,j})\in B(\otimes_{i=1}^k F^2(H_{n_i}))$
satisfying the above-mentioned properties is denoted by
$F^\infty({\bf D^m_f})$.   One can prove that $F^\infty({\bf D^m_f})$
is a Banach algebra, which we call Hardy algebra associated with the
noncommutative polydomain ${\bf D^m_f}$.

 In a similar manner,   one can   define
     the Hardy algebra $R^\infty({\bf D^m_f})$.
For each $i\in \{1,\ldots, k\}$ and $j\in \{1,\ldots, n_i\}$, we
define the operator ${\bf \Lambda}_{i,j}$ acting on the Hilbert space
$F^2(H_{n_1})\otimes\cdots\otimes F^2(H_{n_k})$ by setting
$${\bf \Lambda}_{i,j}:=\underbrace{I\otimes\cdots\otimes I}_{\text{${i-1}$
times}}\otimes \Lambda_{i,j}\otimes \underbrace{I\otimes\cdots\otimes
I}_{\text{${k-i}$ times}}.
$$
 Set ${\bf \Lambda}_i:=({\bf \Lambda}_{i,1},\ldots,{\bf \Lambda}_{i,n_i})$.
  As in Lemma \ref{univ-model}, one can  prove that,
 ${\bf \Lambda}:=({\bf \Lambda}_1,\ldots, {\bf \Lambda}_k)$  is in the
noncommutative polydomain ${\bf D_{{\tilde f}}^m}(\otimes_{i=1}^kF^2(H_{n_i}))$,
 where $\tilde{\bf f}=(\tilde{f_1},\ldots, \tilde{f_k})$.

Let
 $\chi({\bf \Lambda}_{i,j})=\sum\limits_{\beta_1\in \FF_{n_1}^+,\ldots,
 \beta_k\in \FF_{n_k}^+} c_{\tilde{\beta}_1,\ldots, \tilde{\beta}_k}
{\bf \Lambda}_{1,\beta_1}\cdots {\bf \Lambda}_{k,\beta_k} $ be a formal sum with
$c_{\tilde{\beta}_1,\ldots, \tilde{\beta}_k}\in \CC$ and such   that
$$\sum\limits_{\beta_1\in \FF_{n_1}^+,\ldots, \beta_k\in \FF_{n_k}^+}
 |c_{\beta_1,\ldots, \beta_k}|^2 \frac{1}{b_{1,\beta_1}^{(m_1)}\cdots b_{k,\beta_k}^{(m_k)}}<\infty
$$
and
$$
\sup_{p\in\cP, \|p\|\leq 1} \left\|\sum\limits_{\beta_1\in \FF_{n_1}^+,\ldots, \beta_k\in \FF_{n_k}^+} c_{\tilde{\beta}_1,\ldots, \tilde{\beta}_k}
{\bf \Lambda}_{1,\beta_1}\cdots {\bf \Lambda}_{k,\beta_k}(p)\right\|<\infty.
$$
Then there is a unique bounded operator acting on $F^2(H_{n_1})\otimes\cdots \otimes F^2(H_{n_k})$, which
we denote by $\chi({\bf \Lambda}_{i,j})$, such that
$$
\chi({\bf \Lambda}_{i,j})p=\sum\limits_{\beta_1\in \FF_{n_1}^+,\ldots, \beta_k\in \FF_{n_k}^+} c_{\tilde{\beta}_1,\ldots, \tilde{\beta}_k}
{\bf \Lambda}_{1,\beta_1}\cdots {\bf \Lambda}_{k,\beta_k}(p)\quad \text{ for any } \ p\in \cP.
$$
The set of all operators $\chi({\bf \Lambda}_{i,j})\in
B(\otimes_{i=1}^kF^2(H_{n_i}))$
satisfying the above-mentioned properties is  a Banach algebra which is denoted by $R^\infty({\bf D^m_f})$.

\begin{proposition}\label{tilde-f2}
 The following statements hold:
\begin{enumerate}
\item[(i)] $F^\infty({\bf D^m_f})'=R^\infty({\bf D^m_f})$, where $'$ stands for the commutant;
 \item[(ii)] $F^\infty({\bf D^m_f})''=F^\infty({\bf D^m_f})$;
 \item[(iii)] $F^\infty({\bf D^m_f})$  is WOT-closed   in
$B(\otimes_{i=1}^kF^2(H_{n_i}))$.
 \end{enumerate}
\end{proposition}
\begin{proof}   Let $U\in B(\otimes_{i=1}^kF^2(H_{n_i}))$ be the unitary operator
defined by equation
 $$ U(e_{\gamma_1}^1\otimes\cdots \otimes e_{\gamma_k}^k):= (e_{\widetilde{\gamma}_1}^1\otimes\cdots \otimes e_{\widetilde{\gamma}_k}^k),\qquad \gamma_1\in \FF_{n_1}^+,\ldots, \gamma_k\in \FF_{n_k}^+,
 $$
 and note that
 $U^* {\bf \Lambda}_{i,j} U={\bf W^{\tilde f}}_{i,j}$ for any
$i=1,\ldots,k$ and  $j\in \{1,\ldots, n_i\}$,  where ${\bf W^{\tilde f}}_{i,j}$ is the universal model associated with ${\bf D^m_{\tilde
f}}$.
Consequently,  we have $U^*(F^\infty({\bf D^m_{\tilde
f}}))U=R^\infty({\bf D^m_f})$.
On the other hand, since ${\bf W}_{i_1,j_1}
{\bf \Lambda}_{i_2,j_2} ={\bf \Lambda}_{i_2,j_2}{\bf W}_{i_1,j_1} $ for any $i_1,i_2\in \{1,\ldots,k\}$, $j_1\in \{1,\ldots, n_{i_1}\}$, and $j_2\in \{1,\ldots, n_{i_2}\}$.
we deduce that $R^\infty({\bf D^m_f})\subseteq F^\infty({\bf
D^m_f})'$. Now, we prove the reverse inclusion. Let $G\in F^\infty({\bf
D^m_f})'$ and note that,
 since $G(1)\in \otimes_{i=1}^k F^2(H_{n_i})$, we have
  $$G(1)=\sum\limits_{\beta_1\in \FF_{n_1}^+,\ldots, \beta_k\in \FF_{n_k}^+} c_{\tilde{\beta}_1,\ldots, \tilde{\beta}_k}\frac{1}{\sqrt{b_{1,\tilde \beta_1}^{(m_1)}}}\cdots \frac{1}{\sqrt{b_{k,\tilde \beta_k}^{(m_k)}}}
e_{\tilde \beta_1}^1\otimes \cdots \otimes e_{\tilde \beta_k}^k$$
for some coefficients $c_{\tilde{\beta}_1,\ldots, \tilde{\beta}_k}\in \CC $ with
$$
\sum\limits_{\beta_1\in \FF_{n_1}^+,\ldots, \beta_k\in \FF_{n_k}^+} |c_{\beta_1,\ldots, \beta_k}|^2 \frac{1}{b_{1,\beta_1}^{(m_1)}\cdots b_{k,\beta_k}^{(m_k)}}
<\infty.
$$
Taking into account that $ G{\bf W}_{i,j} ={\bf W}_{i,j}G$ for
$i\in\{1,\ldots,k\}$ and $j\in \{1,\ldots,n_i\}$, relations \eqref{WbWb} and its analogue  for ${\bf \Lambda}_{i,j}$  imply
\begin{equation*}
\begin{split}
G(e_{\alpha_1}^1\otimes\cdots \otimes e_{\alpha_k}^k)
&=\sqrt{b_{1,\alpha_1}^{(m_1)}}\cdots \sqrt{b_{k,\alpha_k}^{(m_k)}}GW_{1,\alpha_1}\cdots W_{k,\alpha_k}(1)\\
&=\sqrt{b_{1,\alpha_1}^{(m_1)}}\cdots \sqrt{b_{k,\alpha_k}^{(m_k)}}W_{1,\alpha_1}\cdots W_{k,\alpha_k}
G(1)\\
&=\sum\limits_{\beta_1\in \FF_{n_1}^+,\ldots, \beta_k\in \FF_{n_k}^+} c_{\tilde{\beta}_1,\ldots, \tilde{\beta}_k}
\frac{\sqrt{b^{(m_1)}_{1,\alpha_1}}}{\sqrt{b^{(m_1)}_{1,\alpha_1 \tilde\beta_1}}}\cdots \frac{\sqrt{b^{(m_k)}_{k,\alpha_k}}}{\sqrt{b^{(m_k)}_{k,\alpha_k \tilde\beta_k
}}}
e_{\alpha_1 \tilde \beta_1}^1\otimes \cdots \otimes e_{\alpha_k \tilde \beta_k}^k\\
& =\sum\limits_{\beta_1\in \FF_{n_1}^+,\ldots, \beta_k\in \FF_{n_k}^+} c_{\tilde{\beta}_1,\ldots, \tilde{\beta}_k}
{\bf \Lambda}_{1,\beta_1}\cdots {\bf \Lambda}_{k,\beta_k}(e_{\alpha_1}^1\otimes \cdots \otimes e_{\alpha_k}^k)
\end{split}
\end{equation*}
for any $\alpha_1\in \FF_{n_1}^+, \ldots, \alpha_k\in \FF_{n_k}^+$.
Therefore,
$$G(p)=\sum\limits_{\beta_1\in \FF_{n_1}^+,\ldots, \beta_k\in \FF_{n_k}^+} c_{\tilde{\beta}_1,\ldots, \tilde{\beta}_k}
{\bf \Lambda}_{1,\beta_1}\cdots {\bf \Lambda}_{k,\beta_k}(p)
 $$
 for any polynomial for any  $ p\in \cP$. Since $G$ is a bounded operator,
  $$
g({\bf \Lambda}_{i,j}):=\sum\limits_{\beta_1\in \FF_{n_1}^+,\ldots, \beta_k\in \FF_{n_k}^+} c_{\tilde{\beta}_1,\ldots, \tilde{\beta}_k}
{\bf \Lambda}_{1,\beta_1}\cdots {\bf \Lambda}_{k,\beta_k}
$$
 is
in $R^\infty({\bf D^m_f})$ and $G=g({\bf \Lambda}_{i,j})$.
Therefore, $R^\infty({\bf D^m_f})= F^\infty({\bf D^m_f})'$. The
item (ii) follows easily applying part (i). Now, item (iii) is clear.  This completes the
proof.
\end{proof}

Similarly to the proof of Proposition \ref{tilde-f2}, one can prove that if $\cS\subset B(\cK)$ and $I_\cK\in \cS$, then
$$
\left(F^\infty({\bf D_f^m})\otimes \cS\right)'=R^\infty({\bf D_f^m})\bar\otimes \cS'\quad \text{ and } \quad
\left(R^\infty({\bf D_f^m})\otimes \cS\right)'=F^\infty({\bf D_f^m})\bar\otimes \cS',
$$
where $F^\infty({\bf D_f^m})\bar\otimes \cS'$ is the WOT-closed algebra generated by the spatial tensor product of the two algebras.  Moreover,
for each $i\in \{1,\ldots,k\}$, the commutant of the set
$$
\{W_{i,j}\otimes I_\cH:\ j\in \{1,\ldots,n_i\}\}\cup \{I_{F^2(H_{n_i}}\otimes Y:\ Y\in \cS\}
$$
is equal to $R^\infty({\bf D}_{f_1}^{m_1})\bar \otimes \cS'$.
A repeated appplication of these  results  shows that,
  if ${\bf f}=(f_1,\ldots, f_k)$ and ${\bf m}=(m_1,\ldots, m_k)$, then
$$
F^\infty({\bf D_f^m})\bar \otimes  B(\cH)=F^\infty({\bf D}_{f_1}^{m_1})\bar
 \otimes \cdots \bar \otimes F^\infty({\bf D}_{f_k}^{m_k})\bar \otimes B(\cH)
$$
In  the same manner, one can prove the corresponding result for $R^\infty({\bf D_f^m})\bar \otimes  B(\cH)$.
Another consequence of the results above is the following Tomita-type theorem in our non-selfadjoint setting:
  if $\cM$ is a von Neumann algebra, then
$$
\left(F^\infty({\bf D_f^m})\bar\otimes \cM\right)''=F^\infty({\bf D_f^m})\bar\otimes \cM.
$$

\begin{proposition}\label{density-pol}
 The noncommutative Hardy  algebra $F^\infty({\bf D^m_f})$ is the sequential  SOT-(resp.~WOT-, $w^*$-) closure of all polynomials in ${\bf W}_{i,j}$   and
the identity, where   $i\in \{1,\ldots,k\}$, $j\in \{1,\ldots, n_k\}$.
\end{proposition}
\begin{proof}
Let $P_n$, $n\geq 0$,  be the orthogonal projection of $F^2(H_{n_1})\otimes \cdots \otimes F^2(H_{n_k})$ on the the
subspace  $\text{\rm span}\,\{e_{\alpha_1}\otimes \cdots \otimes e_{\alpha_k}: \ |\alpha_1|+\cdots +|\alpha_k|=n, \, \alpha_1\in \FF_{n_1}^+,\ldots, \alpha_k\in \FF_{n_k}^+\}$.  Define the completely contractive projection
$\Gamma_j:B(\otimes_{i=1}^k F^2(H_{n_i}))\to B(\otimes_{i=1}^k F^2(H_{n_i}))$, $j\in \ZZ$, by
$$
\Gamma_j(A):=\sum_{n\geq\max\{0,-j\}} P_nAP_{n+j}.
$$
The Cesaro operators on
$B(\otimes_{i=1}^k F^2(H_{n_i}))$, defined by
$$
\chi_n(A):=\sum_{|j|<n}\left(1-\frac{|j|}{n}\right)
\Gamma_j(A),\quad n\geq 1,
$$
 are completely contractive and $\chi_n(A)$ converges to $A$ in
 the strong operator topology.
 Let $A\in F^\infty({\bf D^m_f})$ have the Fourier
 representation $\sum\limits_{\beta_1\in \FF_{n_1}^+,\ldots, \beta_k\in \FF_{n_k}^+} c_{\beta_1,\ldots, \beta_k}
{\bf W}_{1,\beta_1}\cdots {\bf W}_{k,\beta_k}$.  Taking into account the
 definition of the  operators ${\bf W}_{i,j}$, one can easily  check that
 $$
 P_{n+j} AP_j=\left(\sum\limits_{{|\beta_1|+\cdots|\beta_k|=n}\atop{\beta_1\in \FF_{n_1}^+,\ldots, \beta_k\in \FF_{n_k}^+}} c_{\beta_1,\ldots, \beta_k}
{\bf W}_{1,\beta_1}\cdots {\bf W}_{k,\beta_k} \right)
 P_j,\quad n\geq 0, j\geq 0,
 $$
 and $P_j A P_{n+j}=0$ if $n\geq 1$  and  $j\geq 0$. Therefore,
 $$
\chi_k(A)=\sum_{0\leq q\leq
n-1}\left(1-\frac{q}{n}\right)\left(\sum\limits_{{|\beta_1|+\cdots|\beta_k|=q}\atop{\beta_1\in \FF_{n_1}^+,\ldots, \beta_k\in \FF_{n_k}^+}} c_{\beta_1,\ldots, \beta_k}
{\bf W}_{1,\beta_1}\cdots {\bf W}_{k,\beta_k} \right)
$$
converges to $A$, as $ k\to \infty$, in the strong operator
topology. The proof is complete.
\end{proof}

\begin{lemma}
\label{gKKg} Let  ${\bf W}:=({\bf W}_1,\ldots, {\bf W}_k)$ be
the universal model associated to the abstract  noncommutative domain
 ${\bf D_f^m}$, where ${\bf W}_i:=({\bf W}_{i,1},\ldots, {\bf W}_{i,n_i})$
 for  $i\in \{1,\ldots,k\}$. If
  $$\varphi({\bf W}_{i,j})=\sum\limits_{\beta_1\in \FF_{n_1}^+,\ldots, \beta_k\in \FF_{n_k}^+} c_{\beta_1,\ldots, \beta_k}
{\bf W}_{1,\beta_1}\cdots {\bf W}_{k,\beta_k}$$ is  in  the noncommutative Hardy algebra $F^\infty({\bf D_f^m})$,
 then the following statements hold.
 \begin{enumerate}
 \item[(i)]
The series  $$\varphi(r{\bf W}_{i,j})
:=\sum\limits_{q=0}^\infty \sum\limits_{{\beta_1\in \FF_{n_1}^+,\ldots,
 \beta_k\in \FF_{n_k}^+}\atop{|\beta_1|+\cdots +|\beta_k|=q}}
  r^{q}c_{\beta_1,\ldots, \beta_k}
{\bf W}_{1,\beta_1}\cdots {\bf W}_{k,\beta_k}$$ converges  in the operator norm topology for any $r\in [0,1)$.
\item[(ii)] The operator $\varphi(r{\bf W}_{i,j})$ is in the noncommutative domain algebra $\cA({\bf D_f^m})$ and
    $$
    \|\varphi(r{\bf W}_{i,j})\|\leq \|\varphi({\bf W}_{i,j})\|.
    $$
    \item[(iii)]
     $\varphi({\bf W}_{i,j})
     =\text{\rm SOT-}\lim_{r\to 1} \varphi(r{\bf W}_{i,j})
     $
     and
     $$ \|\varphi({\bf W}_{i,j})\|=\sup_{0\leq r<1}\|\varphi(r{\bf W}_{i,j})\|=\lim_{r\to 1}\|\varphi(r{\bf W}_{i,j})\|.
     $$
\end{enumerate}
\end{lemma}
\begin{proof} First, we prove that
\begin{equation} \label{ineq-comin}
\begin{split}
\sum\limits_{{\beta_1\in \FF_{n_1}^+,\ldots, \beta_k\in \FF_{n_k}^+}\atop{|\beta_1|=p_1,\ldots, |\beta_k|=p_k}}
  b_{1,\beta_1}^{(m_1)}\cdots b_{k,\beta_k}^{(m_k)}
  &{\bf W}_{1,\beta_1}\cdots {\bf W}_{k,\beta_k}{\bf W}_{k,\beta_k}^*\cdots {\bf W}_{1,\beta_1}^*\\
  &
  \leq
 \left(\begin{matrix}
p_1+m_1-1\\m_1-1 \end{matrix}\right) \cdots \left(\begin{matrix}
p_k+m_k-1\\m_k-1 \end{matrix}\right) I.
  \end{split}
  \end{equation}
According to  relations \eqref{Wij} and   \eqref{Mb}, for each $i\in \{1,\ldots,k\}$, and $p_i\in \NN$,
the operators $\{W_{i,\beta_i}\}_{\beta_i\in \FF_{n_i}, |\beta_i|=p_i}$ have orthogonal ranges and

$$\|W_{i,\beta_i} x\|\leq \frac{1}{\sqrt{b^{(m_i)}_{i,\beta_i}}} \left(\begin{matrix} |\beta_i|+m_i-1\\m_i-1
\end{matrix}\right)^{1/2}\|x\|,
\qquad x\in F^2(H_{n_i}).
$$
Consequently,  we deduce that
\begin{equation*}
  \sum\limits_{\beta_i\in \FF_{n_i}^+,|\beta_i|=p_i} b_{i,\beta_i}^{(m_i)} W_{i,\beta_i}
W_{i,\beta_i}^* \leq \left(\begin{matrix} p_i+m_i-1\\m_i-1
\end{matrix}\right)I\quad \text{  for any } \quad p_i\in\NN.
\end{equation*}
A repeated application of this inequality proves our assertion.
Since $\varphi({\bf W}_{i,j}) \in F^\infty({\bf D_f^m})$, we have
  \begin{equation}
  \label{ell2}
  \sum\limits_{\beta_1\in \FF_{n_1}^+,\ldots, \beta_k\in \FF_{n_k}^+} |c_{\beta_1,\ldots, \beta_k}|^2 \frac{1}{b_{1,\beta_1}^{(m_1)}\cdots b_{k,\beta_k}^{(m_k)}}<\infty.
\end{equation}
 Hence,  using relation \eqref{ineq-comin} and Cauchy-Schwarz inequality,
 we deduce that, for $0\leq r<1$,
\begin{equation*}
 \begin{split}
 &\sum_{p=0}^\infty r^p \left\|\sum_{{p_1,\ldots, p_k\in \NN\cup\{0\}}\atop{p_1+\cdots +p_k=p}}\sum\limits_{{\beta_1\in \FF_{n_1}^+,\ldots, \beta_k\in \FF_{n_k}^+}\atop{|\beta_1|=p_1,\ldots, |\beta_k|=p_k}} c_{\beta_1,\ldots, \beta_k}{\bf W}_{1,\beta_1}\cdots {\bf W}_{k,\beta_k}\right\|\\
 &\leq
 \sum_{p=0}^\infty r^p \sum_{{p_1,\ldots, p_k\in \NN\cup\{0\}}\atop{p_1+\cdots +p_k=p}} \left(\sum\limits_{{\beta_1\in \FF_{n_1}^+,\ldots, \beta_k\in \FF_{n_k}^+}\atop{|\beta_1|=p_1,\ldots, |\beta_k|=p_k}} |c_{\beta_1,\ldots, \beta_k}|^2 \frac{1}{b_{1,\beta_1}^{(m_1)}\cdots b_{k,\beta_k}^{(m_k)}}\right)^{1/2}\\
 &\qquad \qquad \left\|
 \sum\limits_{{\beta_1\in \FF_{n_1}^+,\ldots, \beta_k\in \FF_{n_k}^+}\atop{|\beta_1|=p_1,\ldots, |\beta_k|=p_k}}
  b_{1,\beta_1}^{(m_1)}\cdots b_{k,\beta_k}^{(m_k)}
  {\bf W}_{1,\beta_1}\cdots {\bf W}_{k,\beta_k}{\bf W}_{k,\beta_k}^*\cdots {\bf W}_{1,\beta_1}^*\right\|^{1/2}  \\
  &\leq
  \sum_{p=0}^\infty  r^p\sum_{{p_1,\ldots, p_k\in \NN\cup\{0\}}\atop{p_1+\cdots +p_k=p}} \left(\begin{matrix}
p_1+m_1-1\\m_1-1 \end{matrix}\right)^{1/2}\cdots \left(\begin{matrix}
p_k+m_k-1\\m_k-1 \end{matrix}\right)^{1/2}  \left(\sum\limits_{{\beta_1\in \FF_{n_1}^+,\ldots, \beta_k\in \FF_{n_k}^+}\atop{|\beta_1|=p_1,\ldots, |\beta_k|=p_k}} |c_{\beta_1,\ldots, \beta_k}|^2 \frac{1}{b_{1,\beta_1}^{(m_1)}\cdots b_{k,\beta_k}^{(m_k)}}\right)^{1/2} \\
 &\leq
   \left(\sum_{p=0}^\infty  r^{2p}\sum_{{p_1,\ldots, p_k\in \NN\cup\{0\}}\atop{p_1+\cdots +p_k=p}} \left(\begin{matrix}
p_1+m_1-1\\m_1-1 \end{matrix}\right)\cdots \left(\begin{matrix}
p_k+m_k-1\\m_k-1 \end{matrix}\right)\right)^{1/2}  \left(\sum\limits_{{\beta_1\in \FF_{n_1}^+,\ldots, \beta_k\in \FF_{n_k}^+} } |c_{\beta_1,\ldots, \beta_k}|^2 \frac{1}{b_{1,\beta_1}^{(m_1)}\cdots b_{k,\beta_k}^{(m_k)}}\right)^{1/2}.
 \end{split}
 \end{equation*}
 Now,  using  relation \eqref{ell2} we obtain
 \begin{equation}
 \label{se-ine}
 \sum_{p=0}^\infty  r^{2p}\sum_{{p_1,\ldots, p_k\in \NN\cup\{0\}}\atop{p_1+\cdots +p_k=p}} \left(\begin{matrix}
p_1+m_1-1\\m_1-1 \end{matrix}\right)\cdots \left(\begin{matrix}
p_k+m_k-1\\m_k-1 \end{matrix}\right)\leq \sum_{p=0}^\infty  r^{2p} (p+M)^{Mk-k} (p+1)^{k}<\infty,
 \end{equation}
where $M:=\max\{m_1,\ldots,m_k\}$,
and deduce  that
the series
\begin{equation*}
\varphi(r{ \bf W}_{i,j}):=\sum\limits_{q=0}^\infty \sum\limits_{{\beta_1\in \FF_{n_1}^+,\ldots, \beta_k\in \FF_{n_k}^+}\atop{|\beta_1|+\cdots +|\beta_k|=q}} r^{q}c_{\beta_1,\ldots, \beta_k}
{\bf W}_{1,\beta_1}\cdots {\bf W}_{k,\beta_k}
 \end{equation*}
 converges  in the operator norm topology. Therefore $\varphi(r{ \bf W}_{i,j})$
 is in the  noncommutative domain  algebra
$\cA({\bf D^m_f})$.
In what follows, we show that
\begin{equation}\label{c0sot}
\varphi({\bf W}_{i,j})=\text{\rm SOT-}\lim_{r\to 1} \varphi(r{\bf W}_{i,j})
\end{equation}
for any  $\varphi({\bf W}_{i,j})=\sum\limits_{\beta_1\in \FF_{n_1}^+,\ldots, \beta_k\in \FF_{n_k}^+} c_{\beta_1,\ldots, \beta_k}
{\bf W}_{1,\beta_1}\cdots {\bf W}_{k,\beta_k}$    in  the noncommutative Hardy algebra $F^\infty({\bf D_f^m})$.
 According to the first part of this lemma,
 \begin{equation}
\label{limm}
 \varphi(r{\bf W}_{i,j})=\lim_{n\to \infty}p_n(r{\bf W}_{i,j}),
\end{equation}
where $p_n({\bf W}_{i,j}):=\sum_{q=0}^n \sum\limits_{{\beta_1\in \FF_{n_1}^+,\ldots, \beta_k\in \FF_{n_k}^+}\atop{|\beta_1|+\cdots +|\beta_k|=q}} c_{\beta_1,\ldots, \beta_k}
{\bf W}_{1,\beta_1}\cdots {\bf W}_{k,\beta_k}$ and the convergence is in the operator norm topology.
 For each $i\in \{1,\ldots,k\}$, let   $\gamma_i, \sigma_i, \epsilon_i\in \FF_{n_i}^+$ and set  $n:=|\gamma_1|+\cdots +|\gamma_k|$.
 Since ${\bf W}_{1,\beta_1}^*\cdots {\bf W}_{k,\beta_k}^*(e_{\gamma_1}^1\otimes \cdots \otimes e_{\gamma_k}^k) =0$ for any
$\beta_i\in \FF_{n_i}^+$ with $|\beta_1|+\cdots + |\beta_k|>n$,  we have
$$
\varphi(r{\bf W}_{i,j})^* (e_{\alpha_1}^1\otimes \cdots \otimes e_{\alpha_k}^k) =p_n(r{\bf W}_{i,j})^*  (e_{\alpha_1}^1\otimes \cdots \otimes e_{\alpha_k}^k)$$
for any $\alpha_i\in \FF_{n_i}^+$ with $|\alpha_1|+\cdots +|\alpha_k|\leq n$ and any
$r\in [0,1)$. Due to Lemma \ref{univ-model} and Theorem \ref{radial},   $r{\bf W}:=(r{\bf W}_1,\ldots, r{\bf W}_n)$ is
  a pure $k$-tuple  in the
noncommutative polydomain $ {\bf D_f^m}(\otimes_{i=1}^kF^2(H_{n_i}))$
for any $r\in [0, 1)$. Applying Theorem \ref{Berezin-prop}, we obtain
$$
{\bf K_{f,rW}}p_n(r{\bf W}_{i,j})^*=[p_n({\bf W}_{i,j})^*\otimes
I_{\otimes_{i=1}^k F^2(H_{n_i})}]{\bf K_{f,rW}}
$$
for any $r\in [0,1)$. Using all these facts and the definition of the noncommutative Berezin kernel,  careful
calculations reveal that
\begin{equation*}
\begin{split}
&\left<{\bf K_{f,rW}}\varphi(r{\bf W}_{i,j})^*(e_{\gamma_1}^1\otimes \cdots \otimes e_{\gamma_k}^k),
(e_{\sigma_1}^1\otimes \cdots \otimes e_{\sigma_k}^k)\otimes (e_{\epsilon_1}^1\otimes \cdots \otimes e_{\epsilon_k}^k)\right> \\
&=\left<{\bf K_{f,rW}}p_n(r{\bf W}_{i,j})^*(e_{\gamma_1}^1\otimes \cdots \otimes e_{\gamma_k}^k),
(e_{\sigma_1}^1\otimes \cdots \otimes e_{\sigma_k}^k)\otimes (e_{\epsilon_1}^1\otimes \cdots \otimes e_{\epsilon_k}^k)\right>\\
&=\left<[(p_n({\bf W}_{i,j})^*\otimes
I_{\otimes_{i=1}^kF^2(H_{n_i})})]{\bf K_{f,rW}}(e_{\gamma_1}^1\otimes \cdots \otimes e_{\gamma_k}^k),
(e_{\sigma_1}^1\otimes \cdots \otimes e_{\sigma_k}^k)\otimes (e_{\epsilon_1}^1\otimes \cdots \otimes e_{\epsilon_k}^k)\right>\\
&=\sum_{\beta_i\in \FF_{n_i}^+, i=1,\ldots,k} r^{|\beta_1|+\cdots|\beta_k|}
   \sqrt{b_{1,\beta_1}^{(m_1)}}\cdots \sqrt{b_{k,\beta_k}^{(m_k)}}\left<p_n({\bf W}_{i,j})^*
(e^1_{\beta_1}\otimes \cdots \otimes  e^k_{\beta_k}),e_{\sigma_1}^1\otimes \cdots \otimes e_{\sigma_k}^k\right>\\
&
\qquad \qquad \times\left<{\bf W}_{1,\beta_1}^*\cdots {\bf W}_{k,\beta_k}^*(e_{\gamma_1}^1\otimes \cdots \otimes e_{\gamma_k}^k),  {\bf \Delta_{f,rW}^m}(I)^{1/2} (e_{\epsilon_1}^1\otimes \cdots \otimes e_{\epsilon_k}^k)\right>\\
&=\sum_{\beta_i\in \FF_{n_i}^+, i=1,\ldots,k} r^{|\beta_1|+\cdots|\beta_k|}
   \sqrt{b_{1,\beta_1}^{(m_1)}}\cdots \sqrt{b_{k,\beta_k}^{(m_k)}}\left<\varphi({\bf W}_{i,j})^*
(e^1_{\beta_1}\otimes \cdots \otimes  e^k_{\beta_k}),e_{\sigma_1}^1\otimes \cdots \otimes e_{\sigma_k}^k\right>\\
&
\qquad \qquad \times\left<{\bf W}_{1,\beta_1}^*\cdots {\bf W}_{k,\beta_k}^*(e_{\gamma_1}^1\otimes \cdots \otimes e_{\gamma_k}^k),  {\bf \Delta_{f,rW}^m}(I)^{1/2} (e_{\epsilon_1}^1\otimes \cdots \otimes e_{\epsilon_k}^k)\right>\\
&= \left<[(\varphi({\bf W}_{i,j})^*\otimes
I_{\otimes_{i=1}^kF^2(H_{n_i})})]{\bf K_{f,rW}}(e_{\gamma_1}^1\otimes \cdots \otimes e_{\gamma_k}^k),
(e_{\sigma_1}^1\otimes \cdots \otimes e_{\sigma_k}^k)\otimes (e_{\epsilon_1}^1\otimes \cdots \otimes e_{\epsilon_k}^k)\right>
\end{split}
\end{equation*}
for any $r\in [0,1)$ and $\gamma_i, \sigma_i, \epsilon_i\in \FF_{n_i}^+$, $i\in \{1,\ldots,k\}$.
Hence, since $\varphi(r{\bf W}_{i,j})$ and  $\varphi({\bf W}_{i,j})$  are
bounded operators on $\otimes_{i=1}^k F^2(H_{n_i})$, we deduce that
$$
{\bf K_{f,rW}}\varphi(r{\bf W}_{i,j})^*=[\varphi({\bf W}_{i,j})^*\otimes
I_{\otimes_{i=1}^k F^2(H_{n_i})}]{\bf K_{f,rW}}
$$
for any $r\in[0,1)$.
  Since  $r{\bf W}:=(r{\bf W}_1,\ldots, r{\bf W}_n)$ is
  a pure $k$-tuple  in the
noncommutative polydomain $ {\bf D_f^m}(\otimes_{i=1}^kF^2(H_{n_i}))$
for any $r\in[0, 1)$,  Theorem \ref{Berezin-prop} shows that the Berezin  kernel
${\bf K_{f,rW}}$ is an isometry and, therefore, the equality above
implies
\begin{equation}
\label{gr} \|\varphi(r{\bf W}_{i,j}))\|\leq \|\varphi({\bf W}_{i,j}))\|\quad
\text{ for any } r\in [0,1).
\end{equation}
 Hence, and due to the fact that
$\varphi({\bf W}_{i,j})(e_{\alpha_1}^1\otimes \cdots \otimes e_{\alpha_k}^k) = \lim\limits_{r\to 1}\varphi(r{\bf W}_{i,j})(e_{\alpha_1}^1\otimes \cdots \otimes e_{\alpha_k}^k) $ for any  $\alpha_i\in \FF_{n_i}^+$, an approximation
argument  implies relation \eqref{c0sot}.  Note that if   $0<r_1<r_2<1$, then
$$
\|\varphi(r_1{\bf W}_{i,j})\|\leq \|\varphi(r_2{\bf W}_{i,j})\|.
$$
Indeed, since $\varphi(r_2{\bf W}_{i,j})$ is in
the polydomain  algebra $ \cA({\bf D_f^m})$, Theorem \ref{Poisson-C*} implies
$\|\varphi(r r_2{\bf W}_{i,j})\|\leq
\|\varphi(r_2{\bf W}_{i,j})\|$  for any  $ r\in [0,1)$.
Taking $r:=\frac{r_1}{r_2}$, we prove our assertion.
Now one can easily complete the proof of the theorem.
\end{proof}

\begin{lemma}\label{Berezin-relations}
Let ${\bf T}=({ T}_1,\ldots, { T}_k)$ be in  the noncommutative polydomain ${\bf D_f^m}(\cH)$ and let  $\varphi({\bf W}_{i,j})$ be  in the Hardy algebra $F^\infty({\bf D_f^m})$. Then the noncommutative Berezin kernel satisfies the relations
 \begin{equation*}
 \varphi(r{ T}_{i,j})  {\bf K_{f,T}^*}
= {\bf K_{f,T}^*}(\varphi(r{\bf W}_{i,j})\otimes I_\cH)
\end{equation*}
and
\begin{equation*}
 \varphi(r{T}_{i,j})
  {\bf K_{f,rT}^*}={\bf K_{f,rT}^*} (\varphi({\bf W}_{i,j})\otimes I_\cH)
 \end{equation*}
 for any   $r\in[0,1)$.
\end{lemma}
\begin{proof}
 Due to Theorem \ref{Berezin-prop}, we have
   $${ T}_{i,j}{\bf K_{f,T}^*} =   {\bf K_{f,T}^*}({\bf W}_{i,j}\otimes I_\cH)
    $$
    for any
$i\in \{1,\ldots, k\}$ and $j\in \{1,\ldots, n_i\}$.
Hence, using Theorem \ref{Poisson-C*} and part (i) of Lemma \ref{gKKg}, we deduce that
\begin{equation*}
\varphi(r{  T}_{i,j}):=\sum\limits_{q=0}^\infty \sum\limits_{{\beta_1\in \FF_{n_1}^+,\ldots, \beta_k\in \FF_{n_k}^+}\atop{|\beta_1|+\cdots +|\beta_k|=q}} r^{q}c_{\beta_1,\ldots, \beta_k}
{T}_{1,\beta_1}\cdots {T}_{k,\beta_k}
 \end{equation*}
 converges  in the operator norm topology and  $\varphi(r{ T}_{i,j})  {\bf K_{f,T}^*}
= {\bf K_{f,T}^*}(\varphi(r{\bf W}_{i,j})\otimes I_\cH)$ for all $r\in [0,1)$.
Now, we prove the second part of this lemma.
Using again Theorem \ref{Berezin-prop}, we obtain
 \begin{equation}\label{eq-ker2}
{\bf K_{f,rT}^*}[p({\bf W}_{i,j})\otimes I_\cH]=p(r{ T}_{i,j})
 {\bf K_{f,rT}^*}
 \end{equation}
for any polynomial $p({\bf W}_{i,j})$ and $r\in [0,1)$.
  Since $r{\bf T}:=(r{T}_1,\ldots, r{ T}_n)\in {\bf D_f^m}
  (\cH)$ (see Theorem \ref{radial}), relation \eqref{limm} and
 Theorem \ref{Poisson-C*}   imply
\begin{equation*}
 \varphi(rt T_{i,j})=\lim_{n\to \infty}\sum_{q=0}^n \sum\limits_{{\beta_1\in \FF_{n_1}^+,\ldots, \beta_k\in \FF_{n_k}^+}\atop{|\beta_1|+\cdots +|\beta_k|=q}} (rt)^q c_{\beta_1,\ldots, \beta_k}
{T}_{1,\beta_1}\cdots {T}_{k,\beta_k}, \qquad r,t\in[0,1),
\end{equation*}
where  the convergence is in the operator norm topology. Consequently, an approximation argument shows that relation
 \eqref{eq-ker2} implies
 \begin{equation}
 \label{eq-ker3}
{\bf K_{f,rT}^*} [\varphi (t{\bf W}_{i,j})\otimes I_\cH]=\varphi(rt{ T}_{i,j})
 {\bf K_{f,rT}^*}\quad \text{ for } \ r,t\in (0, 1).
 \end{equation}
 On the other hand, let us prove  that
\begin{equation}\label{lim-t}
 \lim_{t\to 1} \varphi(rt{T}_{i,j})=\varphi(rT_{i,j}),
 \end{equation}
 where the convergence is in the operator norm topology.
  Notice that, due to relation \eqref{se-ine}, if  $\epsilon>0$,  there is   $m_0\in \NN$
 such that $\sum_{p=m_0}^\infty  r^{2p}\sum_{{p_1,\ldots, p_k\in \NN\cup\{0\}}\atop{p_1+\cdots +p_k=p}} \left(\begin{matrix}
p_1+m_1-1\\m_1-1 \end{matrix}\right)\cdots \left(\begin{matrix}
p_k+m_k-1\\m_k-1 \end{matrix}\right)  <\frac {\epsilon^2} {4 K^2}$, where $K:=\|\varphi({\bf W}_{i,j})(1)\|$.
  Since   ${\bf T}:=({T}_1,\ldots, { T}_n)\in {\bf D_f^m}
  (\cH)$ ,  Theorem \ref{Poisson-C*}
  and relation \eqref{ineq-comin} imply
  \begin{equation*}
\begin{split}
\sum\limits_{{\beta_1\in \FF_{n_1}^+,\ldots, \beta_k\in \FF_{n_k}^+}\atop{|\beta_1|=p_1,\ldots, |\beta_k|=p_k}}
  b_{1,\beta_1}^{(m_1)}\cdots b_{k,\beta_k}^{(m_k)}
  &T_{1,\beta_1}\cdots T_{k,\beta_k}T_{k,\beta_k}^*\cdots T_{1,\beta_1}^*\\
  &
  \leq
 \left(\begin{matrix}
p_1+m_1-1\\m_1-1 \end{matrix}\right) \cdots \left(\begin{matrix}
p_k+m_k-1\\m_k-1 \end{matrix}\right) I.
  \end{split}
  \end{equation*}
  Now, as in the proof of Lemma \ref{gKKg}, we can deduce that

\begin{equation*}
 \begin{split}
 &\sum_{p=m_0}^\infty r^p
  \left\|\sum_{{p_1,\ldots, p_k\in \NN\cup\{0\}}\atop{p_1+\cdots +p_k=p}}
  \sum\limits_{{\beta_1\in \FF_{n_1}^+,\ldots, \beta_k\in \FF_{n_k}^+}
  \atop{|\beta_1|=p_1,\ldots, |\beta_k|=p_k}} c_{\beta_1,\ldots, \beta_k}
  {T}_{1,\beta_1}\cdots { T}_{k,\beta_k}\right\|\\
 &\leq
  \left(\sum_{p=m_0}^\infty  r^{2p}\sum_{{p_1,\ldots, p_k\in \NN\cup\{0\}}\atop{p_1+\cdots +p_k=p}} \left(\begin{matrix}
p_1+m_1-1\\m_1-1 \end{matrix}\right)\cdots \left(\begin{matrix}
p_k+m_k-1\\m_k-1 \end{matrix}\right)\right)^{1/2} \|\varphi({\bf W}_{i,j})(1)\|\\
&\leq \frac{\epsilon}{2}.
  \end{split}
 \end{equation*}
 Consequently,   setting $T_{(\beta)}:={T}_{1,\beta_1}\cdots { T}_{k,\beta_k}$, there exists $0< d<1$ such that
 \begin{equation*}
 \begin{split}
 &\left\|\sum_{p=0}^\infty
  \sum\limits_{{\beta_1\in \FF_{n_1}^+,\ldots, \beta_k\in \FF_{n_k}^+}
  \atop{|\beta_1|+\ldots +|\beta_k|=p}}(rt)^{|\beta_1|+\ldots +|\beta_k|}
   c_{\beta_1,\ldots, \beta_k}T_{(\beta)} -
  \sum_{p=0}^\infty
  \sum\limits_{{\beta_1\in \FF_{n_1}^+,\ldots, \beta_k\in \FF_{n_k}^+}
  \atop{|\beta_1|+\ldots +|\beta_k|=p}}r^{|\beta_1|+\ldots +|\beta_k|}
   c_{\beta_1,\ldots, \beta_k}T_{(\beta)}\right\|\\
    &\leq   {\epsilon} +
    \left\| \sum_{p=1}^{m_0-1} r^p(t^p-1)
    \sum\limits_{{\beta_1\in \FF_{n_1}^+,\ldots,
     \beta_k\in \FF_{n_k}^+}\atop{|\beta_1|+\ldots +|\beta_k|=p}}
      c_{\beta_1,\ldots, \beta_k}
      T_{(\beta)}\right\|\|\varphi({\bf W}_{i,j})(1)\| \\
    &\leq   2\epsilon
\end{split}
 \end{equation*}
   for any  $ t\in (d, 1)$.
  Hence, we deduce  relation \eqref{lim-t}.
  On the other hand, due to Lemma \ref{gKKg}, we have $\varphi({\bf W}_{i,j})=\text{\rm SOT-}\lim_{t\to 1} \varphi(t{\bf W}_{i,j}).
     $
 Since  the map $Y\mapsto Y\otimes I_\cH$ is
SOT-continuous on bounded sets, we deduce  that
\begin{equation}
\label{SOT-lim}
\text{\rm SOT-}\lim_{t\to 1}[\varphi(t{\bf W}_{i,j}) )\otimes
I_\cH]=\varphi({\bf W}_{i,j}) \otimes I_\cH.
\end{equation}
 Consequently, using relation \eqref{lim-t} and passing to limit in \eqref{eq-ker3}, as $t\to 1$,
  we   complete the proof.
\end{proof}

In what follows we show that the restriction of the noncommutative
Berezin transform to the Hardy algebra $F^\infty({\bf D_f^m})$
provides a functional calculus  associated with each   pure
tuple of operators in the noncommutative domain ${\bf
D_f^m}(\cH)$.
Moreover, we obtain a
Fatou type result.

\begin{theorem}
\label{funct-calc} Let ${\bf T}=({ T}_1,\ldots, { T}_k)$ be a pure $k$-tuple in  the noncommutative polydomain ${\bf D_f^m}(\cH)$  and define the map
$$\Psi_{\bf T}:F^\infty({\bf D^m_f})\to B(\cH)\quad \text{ by} \quad
 \Psi_{\bf T}(\varphi):= {\bf B_T}[\varphi],$$
 where ${\bf B_T}$ is the noncommutative Berezin transform  at ${\bf T}\in {\bf
 D_f^m}(\cH)$.
 Then
\begin{enumerate}
\item[(i)] $\Psi_{\bf T}$ is    WOT-continuous (resp.
SOT-continuous)  on bounded sets;
\item[(ii)]
$\Psi_{\bf T}$ is a unital completely contractive homomorphism and
$$\Psi_{\bf T}({\bf W}_{1,\beta_1}\cdots {\bf W}_{k,\beta_k})={T}_{1,\beta_1}\cdots {T}_{k,\beta_k}, \qquad \beta_i\in \FF_{n_i}^+,  i\in \{1,\ldots, k\}
$$
\item[(iii)] for any\  $\varphi({\bf W}_{i,j}) \in F^\infty({\bf D^{m}_f})$,
$${\bf B}_{r{\bf T}}[\varphi({\bf W}_{i,j})]=\varphi(r{T}_{i,j})={\bf B}_{\bf T}[\varphi({r\bf W}_{i,j})]$$ and
 $$
\Psi_{\bf T}(\varphi({\bf W}_{i,j}) )=\text{\rm SOT-}\lim_{r\to 1}\varphi(r{ T}_{i,j}).
$$
 \end{enumerate}
 \end{theorem}

\begin{proof}
Since
\begin{equation}\label{def-be}
\Psi_{\bf T}(\varphi({\bf W}_{i,j}) )={\bf K_{f,T}^*} (\varphi({\bf W}_{i,j}) \otimes I_\cH) {\bf K_{f,T}}, \qquad \varphi({\bf W}_{i,j}) \in
F^\infty({\bf D}_f^m),
\end{equation}
 using  standard
facts in functional analysis, we deduce part (i).

 Now, we prove part (ii). Since ${\bf T}$ is pure,
 Theorem \ref{Berezin-prop} shows that
    ${\bf K_{f,T}}$ is an isometry. Consequently, relation
    \eqref{def-be}
    implies
  $$
   \left\|\left[\Psi_{\bf T}(\varphi_{ij}) \right]_{k\times k}\right\|
   \leq
\left\|\left[\varphi_{ij} \right]_{k\times k}\right\|
$$
for any operator-valued matrix $\left[\varphi_{ij} \right]_{k\times k}$ in
$M_{k\times k}(F^\infty({\bf D}^{m}_f))$, which proves that $\Psi_{\bf T}$ is a
unital completely contractive  linear map. Due to Theorem
\ref{Poisson-C*}, $\Psi_{\bf T}$ is a homomorphism  on  the set $\cP({\bf W})$
of  polynomials in $\{{\bf W}_{i,j}\}$.
   By Proposition \ref{density-pol}, the
polynomials in ${\bf W}_{i,j}$ and the identity  are
 sequentially
 WOT-dense in $F^\infty({\bf D}^m_f)$. On the other hand, due to part (i),  $\Psi_{T}$  is
 WOT- continuous on bounded sets. Using  the principle of
 uniform boundedness  we  deduce that $\Psi_T$  is also a homomorphism on
 $F^\infty({\bf D}^m_f)$.

Due to Lemma \ref{Berezin-relations} and taking into account that ${\bf K_{f,T}}$ and ${\bf K_{f,rT}}$ are isometries, we have
\begin{equation*}
\begin{split}
 {\bf B}_{r{\bf T}}[\varphi({\bf W}_{i,j})]&={\bf K_{f,rT}^*} (\varphi({\bf W}_{i,j}) \otimes I) {\bf K_{f,rT}}\\
 &=\varphi(r{T}_{i,j})
 ={\bf K_{f,T}^*} (\varphi(r{\bf W}_{i,j}) \otimes I) {\bf K_{f,T}}\\
 &={\bf B}_{{\bf T}}[\varphi(r{\bf W}_{i,j})].
\end{split}
\end{equation*}
 Now, due to relation \eqref{SOT-lim} we have
 \begin{equation*}
\text{\rm SOT-}\lim_{t\to 1}[\varphi(t{\bf W}_{i,j}) )\otimes
I_\cH]=\varphi({\bf W}_{i,j}) \otimes I_\cH.
\end{equation*}
 Hence, and using the equalities above, we deduce that

\begin{equation*}
 \begin{split}
 {\bf B}_{{\bf T}}[\varphi({\bf W}_{i,j})]&:={\bf K_{f,T}^*}
  (\varphi({\bf W}_{i,j}) \otimes I) {\bf K_{f,T}}\\
 &=\text{\rm SOT-} \lim_{r\to 1}{\bf K_{f,T}^*}
  (\varphi(r{\bf W}_{i,j}) \otimes I) {\bf K_{f,T}}\\
 &=\text{\rm SOT-} \lim_{r\to 1}\varphi(r{T}_{i,j}).
\end{split}
\end{equation*}
This completes the proof.
\end{proof}

We say that ${\bf T}=({ T}_1,\ldots, { T}_k)\in {\bf D_f^m}(\cH)$ is
 completely non-coisometric    if there is no $h\in \cH$, $h\neq 0$ such that
$$
\left< (id-\Phi_{f_1,T_1}^{q_1})\cdots (id-\Phi_{f_k,T_k}^{q_k})(I_\cH)h, h\right>=0$$
for any $(q_1,\ldots, q_k)\in \NN^k$. This is equivalent to the condition
$$
\lim_{q=(q_1,\ldots, q_k)\in \NN^k} \left< (id-\Phi_{f_1,T_1}^{q_1})\cdots (id-\Phi_{f_k,T_k}^{q_k})(I_\cH)h, h\right>=0.
$$
In what follows we present an $F^\infty({\bf D_f^m})$-functional calculus
 for the completely non-coisometric  part of the noncommutative polydomain ${\bf D^m_f}(\cH)$.

\begin{theorem} \label{funct-calc2}
Let ${\bf T}=({ T}_1,\ldots, { T}_k)$ be  a completely
non-coisometric $k$-tuple
      in  the noncommutative polydomain ${\bf D_f^m}(\cH)$.  Then
 $$\Phi(\varphi):=\text{\rm SOT-}\lim_{r\to 1} \varphi(rT_{i,j}), \qquad
  \varphi=\varphi({\bf W}_{i,j})\in F^\infty({\bf D_f^m}),
 $$
 exists in the strong operator topology   and defines a map
 $\Phi:F^\infty({\bf D_f^m})\to B(\cH)$ with the following
 properties:
\begin{enumerate}
\item[(i)]
$\Phi(\varphi)=\text{\rm SOT-}\lim\limits_{r\to 1}{\bf B}_{r{\bf T}}[\varphi]$, where ${\bf
B}_{r{\bf T}}$ is the noncommutative  Berezin transform  at  $r{\bf T}\in {\bf D_f^m}(\cH)$;
\item[(ii)] $\Phi$ is    WOT-continuous (resp.
SOT-continuous)  on bounded sets;
\item[(iii)]
$\Phi$ is a unital completely contractive homomorphism.
\end{enumerate}
 \end{theorem}

\begin{proof} According to Theorem \ref{radial}, $r{\bf T}\in {\bf D_f^m}(\cH)$
 and $r{\bf W}\in {\bf D_f^m}(\otimes_{i=1}^k F^2(H_{n_i}))$ for any   $r\in [0,1)$.
 Due to relations  \eqref{gr} and \eqref{SOT-lim}, we have
 $\text{\rm SOT-}\lim_{t\to 1}[\varphi(t{\bf W}_{i,j}) )\otimes
I_\cH]=\varphi({\bf W}_{i,j}) \otimes I_\cH$.
 Taking the limit in the first  relation  of Lemma \ref{Berezin-relations}, as
  $r\to 1$,  we deduce that the map $\Lambda:
\text{\rm range} \,{\bf K_{f,T}^*}\to \cH$ given by
$\Lambda y:=\lim\limits_{r\to 1}\varphi(t T_{i,j})y$, $y\in \text{\rm
range} \,{\bf K_{f,T}^*}$,
 is well-defined, linear, and
$$
\|\Lambda {\bf K_{f,T}^*}x\|\leq \limsup_{r\to 1} \|\varphi(r{T}_{i,j})\|\|{\bf K_{f,T}^*}x\|\leq \|\varphi({\bf W}_{i,j})\|\|{\bf K_{f,T}^*}x\|
$$
   for any $x\in
 (\otimes_{i=1}^kF^2(H_{n_i}))\otimes \cH$.

Since  ${\bf T}=({ T}_1,\ldots, { T}_k)$ is completely
non-coisometric, Theorem \ref{Berezin-prop} implies that
  the noncommutative Berezin  kernel ${\bf K_{f,T}}$ is
one-to-one and, therefore, the  range  of ${\bf K_{f,T}^*}$ is dense in $\cH$.
Consequently, the map $\Lambda$ has a unique extension to a bounded linear
operator on $\cH$, denoted also by $\Lambda$, with  $\|\Lambda\|\leq
\|\varphi({\bf W}_{i,j})\|$.
We show that
\begin{equation}\label{A}
\lim_{r\to 1} \varphi(rT_{i,j})h =\Lambda h\quad \text{ for any }\ h\in
\cH.
\end{equation}
Let $h\in \cH$ and let  $\{y_k\}_{k=1}^\infty$  be a sequence of vectors
in the range of ${\bf K_{f,T}^*}$, which converges to $h$.
According to Theorem \ref{Poisson-C*} and relations \eqref{limm}, \eqref{gr},  we have
$$
\| \varphi(rT_{i,j})\|\leq\|\varphi(r{\bf W}_{i,j})\|\leq
\|\varphi({\bf W}_{i,j})\|
$$
for any $r\in [0, 1)$.
 Note that
\begin{equation*}
\begin{split}
\|\Lambda h-\varphi(rT_{i,j}))h\|&\leq \|\Lambda h-\Lambda
y_k\|+\|\Lambda y_k-\varphi(rT_{i,j}))y_k\|+\|\varphi(rT_{i,j}))y_k-\varphi(rT_{i,j}))h\|\\
&\leq 2\|\varphi({\bf W}_{i,j})\| \|h-y_k\|+\|\Lambda y_k-\varphi(rT_{i,j}))y_k\|.
\end{split}
\end{equation*}
Consequently, since $\lim\limits_{r\to 1} \varphi(rT_{i,j})y_k =\Lambda y_k$, relation \eqref{A}
 follows.
 Due to Lemma \ref{Berezin-relations},   we have
\begin{equation}
\label{anot} \varphi(rT_{i,j})={\bf K_{f,rT}^*}[\varphi({\bf W}_{i,j})\otimes I_\cH]{\bf K_{f,rT}},
\end{equation}
which together with relation \eqref{A} imply part (i) of the theorem.

Now we prove part (ii). Since
$\|\varphi(rT_{i,j})\|\leq \|\varphi({\bf W}_{i,j})\|$ we deduce that
$\|\Phi(\varphi)\|\leq \|\varphi\|$ for  $\varphi\in F^\infty({\bf D_f^{m}})$.
Taking $r\to 1$ in  the first relation   of Lemma \ref{Berezin-relations} and using the first part of this theorem, we obtain
\begin{equation}
\label{Phi-Kf}
\Phi(\varphi){\bf K_{f,T}^*}={\bf K_{f,T}^*}(\varphi\otimes I),\quad
\varphi\in F^\infty({\bf D_f^m}).
\end{equation}
If $\{g_\iota\}$
be a bounded net in $F^\infty({\bf D_f^m})$ such that
 $g_\iota\to g\in F^\infty({\bf D_f^m})$ in the weak (resp. strong)
 operator topology, then $g_\iota\otimes I$ converges to $ g\otimes I$ in the same topologies.
 By relation \eqref{Phi-Kf}, we have
  $\Phi(g_\iota){\bf K_{f,T}^*}={\bf K_{f,T}^*}(g_\iota\otimes I)$.
Since the range of ${\bf K_{f,T}^*}$ is dense in $\cH$ and
 $\{\Phi(g_\iota)\}$ is bounded, an approximation argument shows that
 $\Phi(g_\iota)\to \Phi(g)$ in the weak (resp. strong)
 operator topology.

 Now, we prove (iii). Relation \eqref{anot} and the fact that  ${\bf K_{f,rT}}$
 is an isometry for $r\in [0, 1)$ imply
$$
\|[\varphi_{st}(rT_{i,j})]_{k\times k}\|\leq \|[\varphi_{st}]_{k\times k}\|
$$
for any operator-valued matrix $[\varphi_{st}]_{k\times k}\in M_{k\times k}(F^\infty({\bf D_f^m}))$
and $r\in [0, 1)$.
Hence, and  using the fact that
 $\Phi(\varphi_{st})=\text{\rm SOT-}\lim_{r\to 1} \varphi_{st}(rT_{i,j})$,
we deduce that $\Phi$ is completely contractive map.
Due to Theorem \ref{Poisson-C*}, $\Phi$ is a homomorphism
 on polynomials in ${\bf W}_{i,j}$ and the identity.
 Since, due to Proposition \ref{density-pol}, these polynomials are sequentially WOT-dense in
 $F^\infty({\bf D_f^m})$
 and $\Phi$ is WOT-continuous  on bounded sets, we deduce part (iii) of the theorem.
  The proof is complete.
\end{proof}

\bigskip

\section{Free holomorphic functions on  noncommutative  polydomains}

We introduce the algebra  $Hol({\bf D_{f, \text{\rm rad}}^m})$   of all free holomorphic
functions on the abstract    radial polydomain ${\bf D_{f, \text{\rm rad}}^m}$.  We identify the polydomain algebra $\cA({\bf D_f^m})$ and the
Hardy algebra $F^\infty({\bf D_f^m})$ with subalgebras  of $Hol({\bf D_{f, \text{\rm rad}}^m})$ .

For each $i\in\{1,\ldots, k\}$, let $Z_i:=(Z_{i,1},\ldots, Z_{i,n_i})$ be
an  $n_i$-tuple of noncommuting indeterminates and assume that, for any
$p,q\in \{1,\ldots, k\}$, $p\neq q$, the entries in $Z_p$ are commuting
 with the entries in $Z_q$. We set $Z_{i,\alpha_i}:=Z_{i,j_1}\cdots Z_{i,j_p}$
  if $\alpha_i\in \FF_{n_i}^+$ and $\alpha_i=g_{j_1}^i\cdots g_{j_p}^i$, and
   $Z_{i,g_0^i}:=1$, where $g_0^i$ is the identity in $\FF_{n^i}^+$.
We consider formal power series
$$
\varphi=\sum_{\alpha_1\in \FF_{n_1}^+,\ldots, \alpha_k\in \FF_{n_k}^+} a_{\alpha_1,\ldots, \alpha_k} Z_{1,\alpha_1}\cdots Z_{k,\alpha_k},\qquad a_{\alpha_1,\ldots, \alpha_k}\in \CC,
$$
in ideterminates $Z_{i,j}$. Denoting  $(\alpha):=(\alpha_1,\ldots, \alpha_k)\in \FF_{n_1}^+\times \cdots \times\FF_{n_k}^+$, $Z_{(\alpha)}:= Z_{1,\alpha_1}\cdots Z_{k,\alpha_k}$, and $a_{(\alpha)}:=a_{\alpha_1,\ldots, \alpha_k}$, we can also use
   the abbreviation
$\varphi=\sum\limits_{(\alpha)} a_{(\alpha)} Z_{(\alpha)}$.

Given a Hilbert space $\cH$, we introduce   the radial polydomain
$$
{\bf D_{f,\text{\rm rad}}^m}(\cH):= \bigcup_{0\leq r<1} r{\bf D_{f}^m}(\cH)\subseteq {\bf D_{f}^m}(\cH).
$$
A formal power series $\varphi$, having the representation above, is called  free holomorphic function on the
{\it abstract radial polydomain}
${\bf D_{f,\text{\rm rad}}^m}:=
\{{\bf D_{f,\text{\rm rad}}^m}(\cH):\ \cH \text{ is a Hilbert space}\}$ if the series
$$
\varphi(X_{i,j}):=\sum_{q=0}^\infty \sum_{{(\alpha)\in \FF_{n_1}^+\times \cdots \times\FF_{n_k}^+ }\atop {|\alpha_1|+\cdots +|\alpha_k|=q}} a_{(\alpha)} X_{(\alpha)}
$$
is convergent in the operator norm topology for any $X=(X_{i,j})\in {\bf D_{f, \text{\rm rad}}^m}(\cH)$ with $i\in \{1,\ldots, k\}$, $j\in \{1,\ldots, n_i\}$,  and any Hilbert space $\cH$.
We denote by $Hol({\bf D_{f, \text{\rm rad}}^m})$ the set of all free holomorphic functions on the abstract radial polydomain ${\bf D_{f,\text{\rm rad}}^m}$.

\begin{lemma}\label{free-ho}
Let $\varphi=\sum\limits_{(\alpha)\in  \FF_{n_1}^+\times \cdots
 \times\FF_{n_k}^+} a_{(\alpha)} Z_{(\alpha)}$ be a
   formal power series and let  ${\bf W}=\{{\bf W}_{i,j}\}$
    be the universal model associated with the abstract noncommutative
polydomain ${\bf D_f^m}$.
Then the following statements are equivalent.
\begin{enumerate}
\item [(i)] $\varphi$  is a free holomorphic function on the abstract radial polydomain ${\bf D_{f,\text{\rm rad}}^m}$.
\item[(ii)] For  any  $ r\in
[0,1)$, the series
$$
 \varphi(r{\bf W}_{i,j}):=\sum_{q=0}^\infty \sum_{{(\alpha)\in \FF_{n_1}^+\times \cdots \times\FF_{n_k}^+ }\atop {|\alpha_1|+\cdots +|\alpha_k|=q}} a_{(\alpha)} r^{|\alpha_1|+\cdots +|\alpha_k|} {\bf W}_{(\alpha)}
$$
is convergent in the operator norm topology.
 \item[(iii)] The inequality
 $$
 \limsup_{n\to\infty}\left\| \sum_{{(\alpha)\in \FF_{n_1}^+\times
  \cdots \times\FF_{n_k}^+ }\atop {|\alpha_1|+\cdots +|\alpha_k|=n}}
  a_{(\alpha)} {\bf W}_{(\alpha)}\right\|^{1/n}\leq 1.
 $$
\end{enumerate}
\end{lemma}
\begin{proof} The equivalence of (i) with (ii) is due to Theorem \ref{Poisson-C*}.
Using standard arguments, one can  easily prove that (ii) is equivalent to (iii).
\end{proof}

We remark  that  the coefficients  of a  free holomorphic function
are uniquely determined by its representation on
   an infinite dimensional   Hilbert space. Indeed,
    under the above notations,
   let $0<r<1$ and assume that
   $\varphi(r{\bf W}_{i,j})=0$.  Taking into account  relation  \eqref{WbWb},
    we have
\begin{equation*}
\left< \varphi(r{\bf W}_{i,j})1 ,  {\bf W}_{(\alpha)} 1 \right>= r^{|\alpha_1|+\cdots +|\alpha_k|} a_{(\alpha)} \frac{1}{ b_{1,  \alpha_1}^{(m_1)}}\cdots \frac{1}{ b_{k, \alpha_k}^{(m_k)}}=0
\end{equation*}
for any   $(\alpha)=(\alpha_1,\ldots, \alpha_k)\in \FF_{n_1}^+\otimes \cdots \otimes \FF_{n_k}^+$. Therefore
$a_{(\alpha)}=0$, which proves our
assertion.

   Due to
Lemma \ref{free-ho}, if $\varphi\in Hol({\bf D_{f,\text{\rm rad}}^m})$,
then $ \varphi(r{\bf W}_{i,j})$ is in the domain algebra $\cA({\bf
D_f^m})$ for any $r\in [0,1)$. Using the results from the previous section, one can see
that $Hol({\bf D_{f,\text{\rm rad}}^m})$ is an algebra. Let
$H^\infty({\bf D_{f,\text{\rm rad}}^m})$  denote the set of  all
elements $\varphi$ in $Hol({\bf D_{f,\text{\rm rad}}^m})$     such
that
$$\|\varphi\|_\infty:= \sup \|\varphi(X_{i,j} )\|<\infty,
$$
where the supremum is taken over all   $ (X_{i,j})\in {\bf D_{f,\text{\rm rad}}^m(\cH)}$ and any Hilbert space
$\cH$. One can  show that $H^\infty({\bf D_{f,\text{\rm
rad}}^m)}$ is a Banach algebra under pointwise multiplication and the
norm $\|\cdot \|_\infty$.
For each $p\in \NN$, we define the norms $\|\cdot
\|_p:M_{p\times p}\left(H^\infty({\bf D_{f,\text{\rm rad}}^m})\right)\to
[0,\infty)$ by setting
$$
\|[\varphi_{st}]_{p\times p}\|_p:= \sup \|[\varphi_{st}(X_{i,j})]_{p\times p}\|,
$$
where the supremum is taken over all  $ (X_{i,j})\in {\bf D_{f,\text{\rm rad}}^m(\cH)}$ and any Hilbert space
$\cH$.
It is easy to see that the norms  $\|\cdot\|_p$, $p\in \NN$,
determine  an operator space structure  on $H^\infty({\bf
D_{f,\text{\rm rad}}^m})$,
 in the sense of Ruan (\cite{Pa-book}).
Let $\varphi$  be  a free holomorphic function on the abstract radial polydomain ${\bf D_{f,\text{\rm rad}}^m}$.
  Note that if   $0<r_1<r_2<1$, then $r_1{\bf D_{f}^m}(\cH)\subset r_2{\bf D_{f}^m}(\cH)\subset {\bf D_{f}^m}(\cH)$.
Since $\varphi(r_2{\bf W}_{i,j})$ is in
the polydomain  algebra $ \cA({\bf D_f^m})$, Theorem \ref{Poisson-C*} implies
$\|\varphi(r r_2{\bf W}_{i,j})\|\leq
\|\varphi(r_2{\bf W}_{i,j})\|$  for any  $ r\in [0,1)$.
Taking $r:=\frac{r_1}{r_2}$, we deduce that
$$
\|\varphi(r_1{\bf W}_{i,j})\|\leq \|\varphi(r_2{\bf W}_{i,j})\|.
$$
On the other hand,
 if $0<r<1$, then we can use again Theorem \ref{Poisson-C*} to show that  the mapping
  $g:r{\bf D_{f}^m}(\cH)\to B(\cH)$
defined by
$$g(X_{i,j}):=\varphi(X_{i,j}),\qquad (X_{i,j}) \in r{\bf D_{f}^m}(\cH),
$$
is continuous and \ $\|g(X_{i,j})\|\leq \|g(r{\bf W}_{i,j})\|$.  Moreover, the series defining $g$  converges uniformly on
$r{\bf D_{f}^m}(\cH)$ in the operator norm topology.

Given  $\varphi\in F^\infty({\bf D^m_f})$ and a Hilbert space $\cH$, the noncommutative Berezin
transform  associated with the abstract noncommutative polydomain ${\bf D^m_f}$
generates a function whose representation on $\cH$ is
$$
{\bf B}[\varphi]:{\bf D_{f,\text{\rm rad}}^m}(\cH)\to B(\cH)
$$
defined by
$$
{\bf B}[\varphi](X_{i,j} ):={\bf B}_{X}[\varphi],\qquad
X:=(X_{i,j})\in {\bf D_{f,\text{\rm rad}}^m}(\cH),
$$
where ${\bf B}_X$  is the Berezin transform at $X$.
We call ${\bf B}[\varphi]$  the Berezin transform of $\varphi$. In
what follows,  we identify the  noncommutative algebra
$F^\infty({\bf D^m_f})$ with the  Hardy subalgebra $H^\infty({\bf
D_{f,\text{\rm rad}}^m})$ of   bounded free holomorphic functions on
${\bf D_{f,\text{\rm rad}}^m}$.

\begin{theorem}\label{f-infty} The map
$ \Phi:H^\infty({\bf D_{f,\text{\rm rad}}^m})\to F^\infty({\bf
D^m_f}) $ defined by
$$
\Phi\left(\sum\limits_{(\alpha)} a_{(\alpha)} Z_{(\alpha)}\right):=\sum\limits_{(\alpha)} a_{(\alpha)} {\bf W}_{(\alpha)}
$$
is a completely isometric isomorphism of operator algebras.
Moreover, if  $g:=\sum\limits_{(\alpha) } a_{(\alpha)} Z_{(\alpha)}$
is  a free holomorphic function on the abstract radial polydomain ${\bf
D_{f,\text{\rm rad}}^m}$, then the following statements are equivalent:
 \begin{enumerate}
 \item[(i)]$g\in H^\infty({\bf D_{f,\text{\rm
rad}}^m})$;
\item[(ii)] $\sup\limits_{0\leq r<1}\|g(r{\bf W}_{i,j})\|<\infty$, where $g(r{\bf W}_{i,j}):=\sum_{q=0}^\infty \sum_{{(\alpha)\in \FF_{n_1}^+\times \cdots \times\FF_{n_k}^+ }\atop {|\alpha_1|+\cdots +|\alpha_k|=q}} r^q a_{(\alpha)} {\bf W}_{(\alpha)}$;
\item[(iii)]
there exists $\varphi\in F^\infty({\bf D^m_f})$ with $g={\bf
B}[\varphi]$, where ${\bf B}$ is the  noncommutative Berezin
transform  associated with the abstract   polydomain ${\bf D^m_f}$.
\end{enumerate}

In this case,
$$
\Phi(g)=\text{\rm SOT-}\lim_{r\to 1}g(r{\bf W}_{i,j}),   \quad
 \quad  \Phi^{-1}(\varphi)={\bf B}[\varphi],\quad
\varphi\in F^\infty({\bf D^m_f}),
$$
and
\begin{equation*}
\|\Phi(g)\|=\sup_{0\leq r<1}\|g(r{\bf W}_{i,j})\|=
\lim_{r\to 1}\|g(r{\bf W}_{i,j})\|.
\end{equation*}

\end{theorem}

\begin{proof} To show that the map $\Phi$ is well-defined, let
$g:=\sum\limits_{(\beta) } a_{(\beta)} Z_{(\beta)}$ be in the Hardy algebra $ H^\infty({\bf D_{f,\text{\rm rad}}^m})$.
 Since $(r{\bf W}_{i,j})\in {\bf D_{f,\text{\rm
rad}}^m}(F^2(H_n))$,  Lemma \ref{free-ho} shows that $g(r{\bf W}_{i,j})$ is well-defined for any $r\in [0,1)$ and
$
\sup_{0\leq r<1}\|g(r{\bf W}_{i,j})\|\leq \|g\|_\infty<\infty.
$
We need to show that $g({\bf W}_{i,j}):=\sum\limits_{(\beta) } a_{(\beta)} {\bf W}_{(\beta)}$ is the Fourier representation of an element in $F^\infty({\bf D_f^m})$.
 Taking into account
relation \eqref{WbWb}, we deduce that
\begin{equation*}
\begin{split}
\sum\limits_{\beta_1\in \FF_{n_1}^+,\ldots, \beta_k\in \FF_{n_k}^+} r^{|\beta_1|+\cdots + |\beta_k|} |a_{\beta_1,\ldots, \beta_k}|^2 \frac{1}{b_{1,\beta_1}^{(m_1)}\cdots b_{k,\beta_k}^{(m_k)}}=
\left\|g(r{\bf W}_{i,j})(1)\right\|
\leq \sup\limits_{0\leq r<1}\|g(r{\bf W}_{i,j})\|<\infty
\end{split}
\end{equation*}
for any $0\leq r<1$. Consequently, $\sum\limits_{\beta_1\in \FF_{n_1}^+,\ldots, \beta_k\in \FF_{n_k}^+}  |a_{\beta_1,\ldots, \beta_k}|^2 \frac{1}{b_{1,\beta_1}^{(m_1)}\cdots b_{k,\beta_k}^{(m_k)}}<\infty$. As in Section 3, the latter relation implies  that $
g({\bf W}_{i,j})p$ is in the tensor product $\otimes_{i=1}^k F^2(H_{n_i})$ for any
polynomial $p\in \otimes_{i=1}^k F^2(H_{n_i})$. Now assume that
  $g({\bf W}_{i,j})\notin F^\infty({\bf
D^m_f})$.
 According to the definition of $F^\infty({\bf
D^m_f})$,  for any fixed  positive
  number $M$, there exists a polynomial $q\in \otimes_{i=1}^k F^2(H_{n_i})$ with $\|q\|=1$ such that
$
\|g({\bf W}_{i,j})q\|>M.
$
Since $\|g(r{\bf W}_{i,j})(1)-g({\bf W}_{i,j})(1)\|\to 0$ as
$r\to 1$, we have
$\|g({\bf W}_{i,j})q-g(r{\bf W}_{i,j})q\| \to 0$, as $\ r\to 1$.
Consequently, there is $r_0\in (0,1)$ such that $
\|g(r_0{\bf W}_{i,j}) q\|> M$,  which implies $
\|g(r_0{\bf W}_{i,j})\|
>M$.  This contradicts the fact  that
$
\sup_{0\leq r<1}\|g(r{\bf W}_{i,j})\|<\infty.
$
 Therefore, $g({\bf W}_{i,j})\in
F^\infty({\bf D^m_f})$, which proves that the map $\Phi$ is
well-defined.

Moreover,  due to Theorem \ref{Poisson-C*} , we have
$\|g(X_{i,j}))\|\leq \|g(r{\bf W}_{i,j})\|$  for any
$(X_{i,j})\in r{\bf D_{f}^m}(\cH)$.  Using  now   Lemma \ref{gKKg}, we deduce that
$$
\|g({\bf W}_{i,j})\|=\sup_{0\leq r<1}\|g(r{\bf W}_{i,j}))\|=\|g\|_\infty
$$
and
$$
\Phi(g)=g({\bf W}_{i,j})=\text{\rm SOT-}\lim\limits_{r\to
1}g(r{\bf W}_{i,j}).
$$
Therefore, $\Phi$ is  a well-defined isometric linear map.
We show now
that $\Phi$ is a surjective map. To this end, let $\varphi({\bf W}_{i,j}):=\sum\limits_{(\beta) } a_{(\beta)} {\bf W}_{(\beta)}$ be in
$F^\infty({\bf D^m_f)}$.
Using  Lemma \ref{gKKg} and  Theorem \ref{free-ho}, we deduce that $g:=\sum\limits_{(\alpha) } a_{(\alpha)} Z_{(\alpha)}$
  is a free holomorphic function on the noncommutative domain ${\bf D_{f,\text{\rm
  rad}}^m}$
  and
  $$
  \|g({ X}_{i,j})\|\leq \|g(r{\bf W}_{i,j})\|\leq \|g({\bf W}_{i,j})\|
  $$
  for any $(X_{i,j})\in r{\bf D_{f}^m}(\cH)$ and $r\in [0,1)$. Hence, we
  deduce that
 $$ \sup_{(X_{i,j})\in {\bf D_{f,\text{\rm rad}}^m}(\cH)} \|g(X_{i,j})\|\leq \|g({\bf W}_{i,j}))\|<\infty,
$$
  which proves that $g\in H^\infty({\bf D_{f,\text{\rm rad}}^m})$.
   This  shows that the map $\Phi$ is surjective.
   Therefore, we have proved
  that $\Phi$ is an isometric isomorphism of operator algebras.
  Using the same techniques and passing to matrices, one can prove that $\Phi$ is a
  completely isometric isomorphism.
Moreover,  note
   that if $X:=(X_{i,j})\in {\bf D_{f,\text{\rm rad}}^m}$,
   then   there is  $r\in (0,1)$ such that $X=rY$ with $Y=(Y_{i,j})\in {\bf D_f^m}(\cH)$. Applying Theorem \ref{funct-calc} part (iii), we deduce that
   $\varphi(X)={\bf B}_{X}[\varphi]$.
   Now, the equivalences mentioned in the theorem  can be easily deduced
from the considerations above.
  The proof is complete.
\end{proof}

For the rest of this section, we assume that  ${\bf D_f^m} (\cH)$  is
closed in the operator norm topology for any Hilbert space $\cH$. Then
 we have ${\bf D_{f, \text{\rm rad}}^m}(\cH)^- ={\bf D_f^m} (\cH)$.
 Note that the
interior  of ${\bf D_f^m} (\cH)$, which we denote by $Int({\bf
D_f^m} (\cH))$, is a subset of ${\bf D_{f, \text{\rm rad}}^m}(\cH)$.
 We remark that if  ${\bf q}=(q_1,\ldots, q_k)$ is a $k$-tuple of   positive regular  noncommutative polynomials, then ${\bf D_q^m} (\cH)$  is
closed in the operator norm topology.

 We  denote by  $A({\bf
D_{f,\text{\rm rad}}^m})$   the set of all  elements $g$
  in $Hol({\bf
D_{f,\text{\rm rad}}^m})$   such that the mapping
$${\bf
D_{f,\text{\rm rad}}^m}(\cH)\ni (X_{i,j})\mapsto
g(X_{i,j})\in B(\cH)$$
 has a continuous extension to  $[{\bf
D_{f,\text{\rm rad}}^m}(\cH)]^-={\bf
D_{f}^m}(\cH)$ for any Hilbert space $\cH$. One can show that  $A({\bf
D_{f,\text{\rm rad}}^m})$ is a  Banach algebra under pointwise
multiplication and the norm $\|\cdot \|_\infty$, and it has an operator space structure under the norms $\|\cdot \|_p$, $p\in \NN$.  Moreover, we can
identify the polydomain algebra $\cA({\bf D^m_f})$ with the subalgebra
 $A({\bf D_{f,\text{\rm rad}}^m})$. Using Theorem \ref{Poisson-C*},
 Theorem  \ref{f-infty}, and an approximation argument, one can obtain the following result.

\begin{corollary}\label{A-infty} The map
$ \Phi:A({\bf D_{f,\text{\rm rad}}^m})\to \cA({\bf D^m_f}) $
defined by
$$
\Phi\left(\sum\limits_{(\alpha)} a_{(\alpha)} Z_{(\alpha)}\right):=\sum\limits_{(\alpha)} a_{(\alpha)} {\bf W}_{(\alpha)}
$$
is a completely isometric isomorphism of operator algebras.
Moreover, if  $g:=\sum\limits_{(\alpha)} a_{(\alpha)} Z_{(\alpha)}$
is  a free holomorphic function on the abstract radial  polydomain ${\bf
D_{f,\text{\rm rad}}^m}$,
then the following statements are equivalent:
 \begin{enumerate}
 \item[(i)]$g\in A({\bf D_{f,\text{\rm
rad}}^m})$;
\item[(ii)] $g(r{\bf W}_{i,j}):=\sum_{q=0}^\infty \sum_{{(\alpha)\in \FF_{n_1}^+\times \cdots \times\FF_{n_k}^+ }\atop {|\alpha_1|+\cdots +|\alpha_k|=q}} r^q  a_{(\alpha)} {\bf W}_{(\alpha)}$ is convergent in the  norm topology  as $r\to 1$;
\item[(iii)]
there exists $\varphi\in \cA({\bf D^m_f})$ with $g={\bf
B}[\varphi]$, where ${\bf B}$ is the  noncommutative Berezin
transform  associated with the abstract  polydomain ${\bf D}^m_f$.
\end{enumerate}

In this case,
$$
\Phi(g)=\lim_{r\to 1}g(r{\bf W}_{i,j})  \quad \text{ and } \quad  \Phi^{-1}(\varphi)={\bf B}[\varphi],\quad \varphi\in \cA({\bf D^m_f}).
$$

\end{corollary}

 We remark that there is an
important connection between the theory of free holomorphic
functions on abstract radial polydomains ${\bf D_{f, \text{\rm rad}}^m}$,
 and the theory of holomorphic functions on polydomains in
$\CC^d$.
Indeed, consider the case when $\cH=\CC^p$ and $p=1,2\ldots.$ Then
     ${\bf D_{f}^m}(\CC^p)$ can be seen as a subset of $\CC^{(n_1+\cdots +n_k)p^2}$ with
an arbitrary norm. We denote by $Int({\bf D_{f}^m}(\CC^p))$ the
interior of the closed set $ {\bf D_{f}^m}(\CC^p)$.   In the particular case when
$p=1$, the interior $Int({\bf D_{f}^m}(\CC))$ is a Reinhardt domain,
i.e., $(\xi_{i,j}\lambda_{i,j})\in Int({\bf D_{f}^m}(\CC))$ for any $( \lambda_{i,j})\in Int({\bf D_{f}^m}(\CC))$ and $\xi_{i,j}\in \TT$.
 Let  $M_{p\times p}(\CC)$  denote the set of all $p\times p$ matrices with entries
in $\CC$.

\begin{proposition} \label{rep-finite1} If  $p\in \NN$ and
$\varphi$ is  a  free holomorphic function   on the abstract radial polydomain ${\bf D_{f,\text{\rm rad}}^m}$,
  then its representation on $\CC^p$,  i.e., the map $\widehat \varphi$ defined by

  $$
  \CC^{(n_1+\cdots +n_k)p^2}\supset {\bf
D_{f,\text{\rm rad}}^m}(\CC^p)\ni (\lambda_{i,j})\mapsto \varphi( \lambda_{i,j})\in M_{ p\times p}(\CC)\subset
\CC^{p^2}
$$
is a  holomorphic function on the interior of ${\bf
D_{f}^m}(\CC^p)$. Moreover, the following statements hold:
\begin{enumerate}
\item[(i)] if $\varphi\in F^\infty({\bf D_{f,\text{\rm rad}}^m})$,
 then $\widehat \varphi$ is bounded on the interior of ${\bf D_{f}^m}(\CC^p)$;
\item[(ii)] if $\varphi\in A({\bf D_{f,\text{\rm rad}}^m})$,
 then $\widehat \varphi$ is  continuous  on  ${\bf D_{f}^m}(\CC^p)$ and
  holomorphic
 on the interior of ${\bf D_{f}^m}(\CC^p)$.
\end{enumerate}
\end{proposition}
\begin{proof}

If $K$ is a compact subset in the interior of
${\bf D_{f}^m}(\CC^p)$, then there exists  $r\in (0,1)$ such that
$K\subset r{\bf D_{f}^m}(\CC^p)$. Indeed,  if $\lambda:=(\lambda_{i,j})\in  Int({\bf D_{f}^m}(\CC^p))\subset \CC^{(n_1+\cdots +n_k)p^2}$, then there exists $\epsilon_\lambda>0$ and $r_\lambda\in (0,1)$ such that
$\frac{1}{r_\lambda} \mu\in Int({\bf D_{f}^m}(\CC^p))$ for any
$\mu\in B_{\epsilon_\lambda}(\lambda):=\{z \in \CC^{(n_1+\cdots +n_k)p^2}: \ \|\lambda-z\|<\epsilon_\lambda\}$. Since $K$ is a compact set and $K\subset \cup_{\lambda\in K}B_{\epsilon_\lambda}(\lambda)$, there exists $\lambda_1,\ldots \lambda_l\in K$ such that  $K\subset \cup_{i=1}^l B_{\epsilon_{\lambda_i}}(\lambda_i)$. Consequently, for any $\mu\in K$, we have
$\frac{1}{r_{\lambda_i}} \mu\in Int({\bf D_{f}^m}(\CC^p))\subset {\bf
D_{f}^m}(\CC^p)  $ for some $i\in \{1,\ldots, n\}$.
Taking into account that  $r_1{\bf
D_{f}^m}(\CC^p)\subset r_2{\bf
D_{f}^m}(\CC^p)$ if  $r_1,r_2\in(0,1)$ and $r_1\leq r_2$, we conclude that $K\subset r{\bf D_{f}^m}(\CC^p)$, where $r:=\max\{r_1,\ldots, r_l\}$.

Note that if  $\varphi:=\sum_{q=0}^\infty \sum_{{(\alpha)\in \FF_{n_1}^+\times \cdots \times\FF_{n_k}^+ }\atop {|\alpha_1|+\cdots +|\alpha_k|=q}} a_{(\alpha)} Z_{(\alpha)}$, then

$$
\left\|\varphi( \lambda_{i,j})- \sum_{{(\alpha)\in \FF_{n_1}^+\times \cdots \times\FF_{n_k}^+ }\atop {|\alpha_1|+\cdots +|\alpha_k|\leq n}} a_{(\alpha)} {\bf \lambda}_{(\alpha)} \right\|\leq
\sum_{s=n+1}^\infty \left\|  \sum_{{(\alpha)\in \FF_{n_1}^+\times \cdots \times\FF_{n_k}^+ }\atop {|\alpha_1|+\cdots +|\alpha_k|=s}}r^{|\alpha_1|+\cdots +|\alpha_k|} a_{(\alpha)} {\bf W}_{(\alpha)}  \right\|
$$
for any $(\lambda_{i,j})\in K$. Using  Theorem \ref{free-ho}, we deduce that
$  \sum_{{(\alpha)\in \FF_{n_1}^+\times \cdots \times\FF_{n_k}^+ }\atop {|\alpha_1|+\cdots +|\alpha_k|\leq n}} a_{(\alpha)} {\bf \lambda}_{(\alpha)} $ converges
 to $\varphi(\lambda_{i,j})$ uniformly on $K$, as
$n\to\infty$. Therefore, the map $ (\lambda_{i,j})\mapsto \varphi( \lambda_{i,j})$ is holomorphic on
the interior of  ${\bf
D_{f}^m}(\CC^p)$. Now, the items (i) and (ii) are consequences of Theorem \ref{f-infty}
and Corollary \ref{A-infty}.
The proof is complete.
\end{proof}

We remark that one can obtain   versions  of all the results of this section
 in the setting of free holomorphic functions with operator-valued coefficients.
  Since the proofs   are very similar we shall omit
them. We also mention that,  in the particular case when
 $k=m_1=1$ and $f_1=Z_1+\cdots +Z_n$,
we recover some of the results concerning  the free holomorphic functions on
 the unit ball of $B(\cH)^n$ (see \cite{Po-holomorphic}, \cite{Po-holomorphic2},
  \cite{Po-automorphism}).

\section{Joint invariant subspaces and universal models}

We obtain  a Beurling type factorization and a characterization of the Beurling \cite{Be} type  joint invariant subspaces under $\{{\bf W}_{i,j}\}$.
We also characterize the reducing subspaces under $\{{\bf W}_{i,j}\}$ and present several results concerning the model theory for pure elements in the noncommutative polydomain ${\bf D_f^m}(\cH)$.

We recall that a   subspace $\cH\subseteq \cK$ is
called co-invariant under $\cS\subset B(\cK)$ if $X^*\cH\subseteq
\cH$ for any $X\in \cS$.

\begin{theorem}\label{cyclic} Let ${\bf W}:=({\bf W}_1,\ldots, {\bf W}_k)$ be the universal model
 associated to the abstract  noncommutative domain ${\bf D_f^m}$.
  If $\cK$  be a Hilbert space and
$\cM\subseteq (\otimes_{i=1}^k F^2(H_{n_i}))\otimes \cK$ is a co-invariant subspace under each operator
${\bf W}_{i,j}\otimes I_\cK$ for  $i\in \{1,\ldots, k\}$, $j\in \{1,\ldots, n_i\}$, then  there exists a subspace
$\cE\subseteq \cK$ such that
\begin{equation*}
\overline{\text{\rm span}}\,\left\{\left({\bf W}_{1,\beta_1}\cdots {\bf W}_{k,\beta_k}\otimes
I_\cK\right)\cM:\ \beta_1\in \FF_{n_1}^+,\ldots, \beta_k\in \FF_{n_k}^+\right\}=(\otimes_{i=1}^k F^2(H_{n_i}))\otimes \cE.
\end{equation*}
Consequently, a subspace
$\cM\subseteq (\otimes_{i=1}^k F^2(H_{n_i}))\otimes \cK$ is
 reducing under each operator
${\bf W}_{i,j}\otimes I_\cK$ for  $i\in \{1,\ldots, k\}$,
$j\in \{1,\ldots, n_i\}$,
if and only if   there exists a subspace $\cE\subseteq \cK$ such
that
\begin{equation*}
 \cM=(\otimes_{i=1}^k F^2(H_{n_i}))\otimes \cE.
\end{equation*}
\end{theorem}

\begin{proof}
  Define the subspace $\cE\subseteq \cK$ by
 $\cE:=({\bf P}_\CC\otimes I_\cK)\cM$, where ${\bf P}_\CC$ is the
 orthogonal projection from $\otimes_{i=1}^k F^2(H_{n_i})$ onto $\CC 1\subset \otimes_{i=1}^k F^2(H_{n_i})$. Let $\varphi$ be a
nonzero element of $\cM$ with   representation
$$\varphi=\sum\limits_{{\beta_1\in \FF_{n_1}^+,\ldots, \beta_k\in \FF_{n_k}^+}} e^1_{\beta_1}\otimes\cdots \otimes  e^k_{\beta_k}\otimes h_{\beta_1,\ldots, \beta_k},
$$
where $h_{\beta_1,\ldots, \beta_k}\in \cK$ and $\sum\limits_{{\beta_1\in \FF_{n_1}^+,\ldots, \beta_k\in \FF_{n_k}^+}}\|h_{\beta_1,\ldots, \beta_k}\|^2<\infty$.
Let $\sigma_1\in \FF_{n_1}^+,\ldots, \sigma_k\in \FF_{n_k}^+$ be such that $h_{\sigma_1,\ldots, \sigma_k}\neq 0$ and note that
\begin{equation*}\label{PB*}
({\bf P}_\CC \otimes I_\cK) ({\bf W}_{1,\sigma_1}^*\cdots {\bf W}_{k,\sigma_k}^*\otimes I_\cK)\varphi= 1\otimes\frac{1}{\sqrt{b_{1,\sigma_1}^{(m_1)}}}\cdots \frac{1}{\sqrt{b_{k,\sigma_k}^{(m_k)}}}  h_{\sigma_1,\ldots, \sigma_k}.
\end{equation*}
 Consequently, since $\cM$ is  a
co-invariant subspace under ${\bf W}_{i,j}\otimes I_\cK$ for  $i\in \{1,\ldots, k\}$, $j\in \{1,\ldots, n_i\}$,
we deduce  that $ h_{\sigma_1,\ldots, \sigma_k}\in \cE$.   This implies
$$
 ({\bf W}_{1,\sigma_1}\cdots {\bf W}_{k,\sigma_k}\otimes I_\cK)(1\otimes
h_{\sigma_1,\ldots, \sigma_k})=\frac{1}{\sqrt{b_{1,\sigma_1}^{(m_1)}}}\cdots \frac{1}{\sqrt{b_{k,\sigma_k}^{(m_k)}}}e^1_{\sigma_1}\otimes\cdots \otimes  e^k_{\sigma_k}\otimes h_{\sigma_1,\ldots, \sigma_k}
$$
is a vector in $\otimes_{i=1}^k F^2(H_{n_i})\otimes \cE$.
 Therefore,
\begin{equation}\label{phi-series}
\varphi= \lim_{n\to\infty}\sum_{q=0}^n \sum\limits_{{\beta_1\in \FF_{n_1}^+,\ldots, \beta_k\in \FF_{n_k}^+}\atop{|\beta_1|+\cdots +|\beta_k|=q}} e^1_{\beta_1}\otimes\cdots \otimes  e^k_{\beta_k}\otimes h_{\beta_1,\ldots, \beta_k}
\end{equation}
is in $\otimes_{i=1}^k F^2(H_{n_i})\otimes \cE$. Hence,
 $\cM\subset\otimes_{i=1}^k F^2(H_{n_i})\otimes \cE$ and
$$
\cY:= \overline{\text{\rm span}}\,\left\{({\bf W}_{1,\sigma_1}\cdots
 {\bf W}_{k,\sigma_k}\otimes I_\cK)\cM:\ \sigma_1\in \FF_{n_1}^+,\ldots,
  \sigma_k\in \FF_{n_k}^+\right\}\subset (\otimes_{i=1}^k F^2(H_{n_i}))\otimes \cE.
$$

To prove  the reverse inclusion, we show first that $\cE\subset
\cY$.
 If
$h_0\in \cE$, $h_0\neq 0$, then there exists $g\in \cM\subset
\otimes_{i=1}^k F^2(H_{n_i})\otimes \cE$ such that $$g=1\otimes
h_0+\sum\limits_{{\beta_1\in \FF_{n_1}^+,\ldots, \beta_k\in \FF_{n_k}^+}\atop{|\beta_1|+\cdots +|\beta_k|\geq 1}} e^1_{\beta_1}\otimes\cdots \otimes  e^k_{\beta_k}\otimes h_{\beta_1,\ldots, \beta_k}
$$
and $1\otimes h_0=({\bf P}_\CC\otimes I_\cK) g$.
Consequently, due to Lemma \ref{univ-model}, we have
\begin{equation*}
\begin{split}
1\otimes h_0=({\bf P}_\CC\otimes I_\cK) g
=(I-\Phi_{q_1,{\bf W}_1\otimes I_\cK})^{m_1}\cdots (I-\Phi_{q_k,{\bf W}_k\otimes I_\cK})^{m_k}(I)g.
\end{split}
\end{equation*}
 Taking into account that $\cM$ is co-invariant under ${\bf W}_{i,j}\otimes I_\cK$ for  $i\in \{1,\ldots, k\}$, $j\in \{1,\ldots, n_i\}$,
we deduce that $h_0\in\cY$ for any $h_0\in \cE$, i.e., $\cE\subset
\cY$. The latter inclusion  shows that $({\bf W}_{1,\sigma_1}\cdots {\bf W}_{k,\sigma_k}\otimes I_\cK)(1\otimes \cE)\subset \cY$ for any $\sigma_1\in \FF_{n_1}^+,\ldots, \sigma_k\in \FF_{n_k}^+$, which
implies
\begin{equation*}\label{PN}
\frac{1}{\sqrt{b_{1,\sigma_1}^{(m_1)}}}\cdots \frac{1}{\sqrt{b_{k,\sigma_k}^{(m_k)}}}e^1_{\sigma_1}\otimes\cdots \otimes  e^k_{\sigma_k}\otimes \cE\subset
\cY.
\end{equation*}
Hence, if $\varphi\in (\otimes_{i=1}^k F^2(H_{n_i}))\otimes \cE$
has   the  representation \eqref{phi-series},   we deduce that
$\varphi \in \cY.$
Therefore, $ (\otimes_{i=1}^k F^2(H_{n_i}))\otimes \cE\subseteq \cY$.
The last part of the theorem is now obvious.
 The proof is complete.
\end{proof}

Let ${\bf W}:=({\bf W}_1,\ldots, {\bf W}_k)$ be the universal model
 associated to the abstract  noncommutative domain ${\bf D_f^m}$.
  An operator $M:(\otimes_{i=1}^k F^2(H_{n_i}))\otimes \cH\to
(\otimes_{i=1}^k F^2(H_{n_i}))\otimes \cK$  is called {\it  multi-analytic}  with respect to ${\bf W}$
 if $$M({\bf W}_{i,j}\otimes I_\cH)=({\bf W}_{i,j}\otimes I_\cK)M$$
for any $i\in \{1,\ldots,k\}$ and  $j\in \{1,\ldots, n_i\}$. In case $M$ is a partial isometry, we call
it {\it inner multi-analytic} operator.

\begin{theorem}\label{Beur-fact} Let ${\bf W}:=({\bf W}_1,\ldots, {\bf W}_k)$ be the universal model associated to the abstract  noncommutative domain ${\bf D_f^m}$ and  let ${\bf W}_i\otimes I_\cH:=({\bf W}_{i,1}\otimes I_\cH,\ldots, {\bf W}_{i,n_i}\otimes I_\cH)$ for  $i\in \{1,\ldots,k\}$, where $\cH$ is a Hilbert space.
If $Y\in B((\otimes_{i=1}^k F^2(H_{n_i}))\otimes \cH)$ then the following statements  are equivalent.
\begin{enumerate}
\item[(i)]
There is  a multi-analytic operator
 $M:(\otimes_{i=1}^k F^2(H_{n_i}))\otimes \cE \to
 (\otimes_{i=1}^k F^2(H_{n_i}))\otimes \cH$ with respect to ${\bf W}$,
 where  $\cE$ is a Hilbert space, such that
$$Y=M M^*.$$
\item[(ii)] For any ${\bf p}:=(p_1,\ldots, p_k)\in \ZZ_+^k$
 such that ${\bf p}\leq {\bf m}$, ${\bf p}\neq 0$,
     $$
   (id-\Phi_{f_1,{\bf W}_1\otimes I_\cH})^{p_1}\cdots
    (id-\Phi_{f_k,{\bf W}_k\otimes I_\cH})^{p_k}(Y)\geq 0.
   $$
\end{enumerate}
\end{theorem}
\begin{proof} Setting  ${\bf \Delta}_{{\bf f,W}\otimes I_\cH}^{\bf p}
:=(id-\Phi_{f_1,{\bf W}_1\otimes I_\cH})^{p_1}\cdots
 (id-\Phi_{f_k,{\bf W}_k\otimes I_\cH})^{p_k}$, it is
   easy to see that if  item (i) holds, then
$$
{\bf \Delta}_{{\bf f,W}\otimes I_\cH}^{\bf p}
(Y)=M {\bf \Delta}_{{\bf f,W}\otimes I_\cE}^{\bf p}(I)M^*\geq 0
$$
for any ${\bf p}:=(p_1,\ldots, p_k)\in \ZZ_+^k$ such that ${\bf p}\leq {\bf m}$,
 ${\bf p}\neq 0$.

To prove the converse, assume that (ii) holds. In particular, we have
$\Phi_{f_1,{\bf W}_1\otimes I_\cH}({\bf \Delta}_{{\bf f,W}\otimes I_\cH}^{\bf m'}(Y))
\leq {\bf \Delta}_{{\bf f,W}\otimes I_\cH}^{\bf m'}(Y)$, where ${\bf m}'
=(m_1-1,m_2,\ldots,m_k)$. Consequently,
 $\Phi_{f_1,{\bf W}_1\otimes I_\cH}^n({\bf \Delta}_{{\bf f,W}\otimes I_\cH}^{\bf m'}(Y))
 \leq {\bf \Delta}_{{\bf f,W}\otimes I_\cH}^{\bf m'}(Y)$ for any $n\in \NN$. Since
$ {\bf W}:=({\bf W}_1,\ldots, {\bf W}_k)$ is a pure $k$-tuple, we have
SOT-$\lim_{n\to\infty}\Phi_{f_1,{\bf W}_1
\otimes I_\cH}^n({\bf \Delta}_{{\bf f,W}\otimes I_\cH}^{\bf m'}(Y))=0$, which implies
 ${\bf \Delta}_{{\bf f,W}\otimes I_\cH}^{\bf m'}(Y)\geq 0$.
Continuing this process, we deduce that $Y\geq 0$.

 Let $\cM:=\overline{\text{\rm range}\, Y^{1/2}}$ and define
\begin{equation}\label{ai}
A_{i,j}(Y^{1/2} x):=Y^{1/2} ({\bf W}_{i,j}^*\otimes I_\cH)x,\quad x\in
(\otimes_{i=1}^k F^2(H_{n_i}))\otimes \cH,
\end{equation}
 for any $i\in \{1,\ldots, k\}$ and $j\in \{1,\ldots, n_i\}$. Since
  $\Phi_{f_i,{\bf W}_i\otimes I_\cH}(Y)\leq Y$, we have
\begin{equation*}
\begin{split}
\sum_{\alpha\in \FF_{n_i}^+, |\alpha|\geq 1}
 a_{i,\alpha}\|A_{i,{\tilde \alpha}}Y^{1/2}x\|^2
=\left< \Phi_{f_i,{\bf W}_i\otimes I_\cH}(Y)x,x\right>\leq \|Y^{1/2} x\|^2
\end{split}
\end{equation*}
for any $x\in (\otimes_{i=1}^k F^2(H_{n_i}))\otimes \cH$.  Consequently, we deduce that
$a_{i,g_j^i}\|A_{i,j}Y^{1/2} x\|^2\leq \|Y^{1/2} x\|^2$, for any $x\in
(\otimes_{i=1}^k F^2(H_{n_i}))\otimes \cH$.  Since $a_{i,g_j^i}\neq 0$ each  $A_{i,j}$ can be uniquely
be extended to a bounded operator (also denoted by $A_{i,j}$) on the
subspace $\cM$. Denoting $X_{i,j}:=A_{i,j}^*$ for  $i\in \{1,\ldots, k\}$, $j\in\{1,\ldots, n_i\}$,   an
approximation argument shows that $\Phi_{f_i,X_i}(I_\cM)\leq I_\cM$
and relation \eqref{ai} implies
\begin{equation*}X_{i,j}^*(Y^{1/2} x)=Y^{1/2} ({\bf W}_{i,j}^*\otimes I_\cH)x,\qquad x\in
(\otimes_{i=1}^k F^2(H_{n_i}))\otimes \cH,
\end{equation*}
 for any $i\in \{1,\ldots, k\}$ and $j\in \{1,\ldots, n_i\}$. This implies
 $$
 Y^{1/2} {\bf \Delta_{f,X}^p}(I_\cM) Y^{1/2}
 ={\bf \Delta}_{{\bf f,W}\otimes I_\cH}^{\bf p}(Y)\geq 0
 $$
for any ${\bf p}:=(p_1,\ldots, p_k)\in \ZZ_+^k$ such that ${\bf p}\leq {\bf m}$, ${\bf p}\neq 0$.
 On the other hand,  we have
\begin{equation*}
\begin{split}
\left< \Phi_{f_i, X_i}^n(I_\cM) Y^{1/2}x, Y^{1/2} x\right> &= \left<
\Phi_{f_i,
{\bf W}_i\otimes I_\cH}^n(Y)x,  x\right>
\leq \|Y\| \left< \Phi_{f_i, {\bf W}_i\otimes I_\cH}^n(I)x,  x\right>
\end{split}
\end{equation*}
for any $ x\in
(\otimes_{i=1}^k F^2(H_{n_i}))\otimes \cH$ and $n\in \NN$. Since \
SOT-$\lim\limits_{n\to\infty}\Phi_{f_i,{\bf W}_i\otimes I_\cH}^n(I)=0$,
we have \
SOT-$\lim\limits_{m\to\infty}\Phi_{f_i,X_i}^n(I_\cM)=0$, which, due to Proposition
 \ref{pure2} shows that
${\bf X}:=(X_1,\ldots, X_k)$ is a pure $k$-tuple in the noncommutative polydomain
${\bf D_f^m}(\cM)$. Set $\cE:=\overline{{\bf \Delta_{f,X}^m}(I_\cM)(\cM)}$.
According to Theorem \ref{Berezin-prop}, the noncommutative Berezin
kernel ${\bf K_{f,X}}:\cM\to (\otimes_{i=1}^k F^2(H_{n_i}))\otimes \cE$
 is an isometry with the property that
\begin{equation*}
  X_{i,j}{\bf K_{f,X}^*}={\bf K_{f,X}^*} ({\bf W}_{i,j}\otimes I_\cE)
\end{equation*}
for any $ i\in \{1,\ldots,k\}$ and  any $ j\in \{1,\ldots, n_i\}$.
Now, define the bounded linear  operator  $M:=Y^{1/2} {\bf K_{f,X}^*}: (\otimes_{i=1}^k F^2(H_{n_i}))\otimes \cE\to (\otimes_{i=1}^k F^2(H_{n_i}))\otimes
\cH$ and note that
\begin{equation*}
\begin{split}
M({\bf W}_{i,j}\otimes I_\cE)&=Y^{1/2}{\bf K_{f,X}^*}
({\bf W}_{i,j}\otimes I_\cE)=Y^{1/2}
X_{i,j} {\bf K_{f,X}^*}\\
&=({\bf W}_{i,j}\otimes I_\cH) Y^{1/2} {\bf K_{f,X}^*} =({\bf W}_{i,j}\otimes I_\cH) M
\end{split}
\end{equation*}
for any $i\in\{1,\ldots, k\}$ and $j\in \{1,\ldots, n_i\}$, which proves that $M$ is a multi-analytic operator with respect to ${\bf W}_{i,j}$. We also have $MM^*=Y^{1/2}
{\bf K_{f,X}^*} {\bf  K_{f,X}} Y^{1/2} =Y$. This completes the proof.
\end{proof}

We say that $\cM\subset (\otimes_{i=1}^k F^2(H_{n_i}))\otimes \cH$ is
a Beurling type invariant subspace under the operators
 ${\bf W}_{i,j}\otimes I_\cH$, $i\in \{1,\ldots, k\}$, $j\in \{1,\ldots, n_i\}$,
  if there is an inner multi-analytic operator  with respect to ${\bf W}$,
  $$\Psi:(\otimes_{i=1}^k F^2(H_{n_i}))\otimes
   \cE \to (\otimes_{i=1}^k F^2(H_{n_i}))\otimes \cH,$$
   such that
   $\cM=\Psi\left((\otimes_{i=1}^k F^2(H_{n_i}))\otimes \cE\right)$.

\begin{corollary}\label{Beurling}
Let  $\cM\subset (\otimes_{i=1}^k F^2(H_{n_i}))\otimes \cH$
 be an invariant subspace under the operators
  ${\bf W}_{i,j}\otimes I_\cH$, $i\in \{1,\ldots, k\}$,
  $j\in \{1,\ldots, n_i\}$. Then $\cM$ is of  Beurling type if and only if
 $$
   (id-\Phi_{f_1,{\bf W}_1\otimes I_\cH})^{p_1}\cdots
    (id-\Phi_{f_k,{\bf W}_k\otimes I_\cH})^{p_k}(P_\cM)\geq 0
   $$
for any ${\bf p}:=(p_1,\ldots, p_k)\in \ZZ_+^k$ such that ${\bf p}\leq {\bf m}$, where $P_\cM$ is the orthogonal projection of the Hilbert space  $ (\otimes_{i=1}^k F^2(H_{n_i}))\otimes \cH$ onto $\cM$. In the particular case when ${\bf m}=(1,\ldots,1)$, the condition above is satisfied when ${\bf W}\otimes I_\cH|_\cM:=({\bf W}_1\otimes I_\cH|_\cM,\ldots, {\bf W}_k\otimes I_\cH|_\cM)$ is  doubly commuting.
 \end{corollary}
\begin{proof}
If
$\Psi:(\otimes_{i=1}^k F^2(H_{n_i}))\otimes
\cE \to (\otimes_{i=1}^k F^2(H_{n_i}))\otimes \cH$ is a
 inner multi-analytic operator
 and
 $\cM=\Psi\left((\otimes_{i=1}^k F^2(H_{n_i}))\otimes \cE\right)$,
then $P_\cM=\Psi \Psi^*$. Taking into account
Lemma \ref{univ-model}, we deduce that
$$
   (id-\Phi_{f_1,{\bf W}_1\otimes I_\cH})^{p_1}\cdots
    (id-\Phi_{f_k,{\bf W}_k\otimes I_\cH})^{p_k}(P_\cM)
    =\Psi (P_{\bf C}\otimes I_\cE) \Psi^*\geq 0
   $$
   for any ${\bf p}:=(p_1,\ldots, p_k)\in \ZZ_+^k$ such that ${\bf p}\leq {\bf m}$.
   The converse is a consequence of Theorem \ref{Beur-fact}, when we take $Y=P_\cM$.

   Now, we consider the case when ${\bf m}=(1,\ldots,1)$.
Note that if
    $\cM$
 is  an invariant subspace under the operators
  ${\bf W}_{i,j}\otimes I_\cH$, then
     ${\bf W}\otimes I_\cH|_\cM$
   is doubly commuting if and only if
     $P_\cM({\bf W}_{i_1,j_1}\otimes I_\cH)P_\cM$ commutes with
      $P_\cM({\bf W}_{i_2,j_2}^*\otimes I_\cH)P_\cM$ for any
       $i_1,i_2\in \{1,\ldots,k\}$,
      $i_1\neq i_2$, and any $j_1\in \{1,\ldots, n_{i_1}\}$,
       $j_2\in \{1,\ldots, n_{i_2}\}$.
       The latter condition is equivalent to
       \begin{equation}\label{commutes}
        P_\cM({\bf W}_{i_1,\alpha}\otimes I_\cH)P_\cM \quad \text{ commutes with } \qquad
      P_\cM({\bf W}_{i_2,\beta}^*\otimes I_\cH)P_\cM
       \end{equation}
       for any $\alpha\in \FF_{n_{i_1}}^+$ and $\beta\in \FF_{n_{i_2}}^+$.
     Assume that
   $\cM$
 is   invariant subspace under the operators
  ${\bf W}_{i,j}\otimes I_\cH$ and
     ${\bf W}\otimes I_\cH|_\cM$
   is doubly commuting. Then, due to relation \eqref{commutes},
    for any $\alpha_i\in \FF_{n_i}^+$, $i\in \{1,\ldots, k\}$, we have
    \begin{equation}
    \begin{split}
    &({\bf W}_{1,\alpha_1}\otimes I_\cH)\cdots ({\bf W}_{k,\alpha_k}\otimes I_\cH) P_\cM({\bf W}_{k,\alpha_k}^*\otimes I_\cH)
    \cdots
    ({\bf W}_{1,\alpha_1}^*\otimes I_\cH)\\
    &\qquad\qquad=({\bf W}_{1,\alpha_1}\otimes I_\cH)P_\cM ({\bf W}_{1,\alpha_1}^*\otimes I_\cH)
    \cdots ({\bf W}_{k,\alpha_k}\otimes I_\cH)P_\cM
    ({\bf W}_{k,\alpha_k}^*\otimes I_\cH).
    \end{split}
    \end{equation}
   Consequently, we deduce that
   $$
   (id-\Phi_{f_1,{\bf W}_1\otimes I_\cH})^{p_1}\cdots
    (id-\Phi_{f_k,{\bf W}_k\otimes I_\cH})^{p_k}(P_\cM)
    =\left(P_\cM-\Phi_{f_1,{\bf W}_1\otimes I_\cH}(P_\cM)\right)^{p_1}\cdots
     \left(P_\cM-\Phi_{f_k,{\bf W}_k\otimes I_\cH}(P_\cM)\right)^{p_k}
     $$
     for any ${\bf p}:=(p_1,\ldots, p_k)\in \ZZ_+^k$ such that
      ${\bf p}\leq (1,\ldots,1)$.
 Now, since ${\bf W}_1,\ldots, {\bf W}_k$  are commuting tuples, we deduce that
 $P_\cM-\Phi_{f_i,{\bf W}_i\otimes I_\cH}(P_\cM)$, $i\in \{1,\ldots, k\}$,
  are commuting  operators. On the other hand,  they  are also  positive operators. Indeed,
  let $\{a_{i,\alpha_i}\}_{\alpha\in \FF_{n_i}^+}$ be the coefficients of the positive regular free holomorphic
  function $f_i$, and let $x\in  (\otimes_{i=1}^k F^2(H_{n_i}))\otimes \cH$ have the representation
  $x=x_1+x_2$ with respect to the orthogonal decomposition $\cM\oplus \cM^\perp$.
  Note that
  \begin{equation*}
  \begin{split}
  \left<\Phi_{f_i,{\bf W}_i\otimes I_\cH}(P_\cM)x,x\right>
  &=\left<\Phi_{f_i,{\bf W}_i\otimes I_\cH}(P_\cM)x_1,x_1\right>=
  \sum_{|\alpha|\geq 1} a_{i,\alpha_i}
   \|P_\cM ({\bf W}_{i,\alpha_i}\otimes I_\cH)x_1\|^2\\
   &\leq \left<\Phi_{f_i,{\bf W}_i\otimes I_\cH}(I)x_1,x_1\right>\leq \|x_1\|^2=
   \left<P_\cM x,x\right>,
  \end{split}
  \end{equation*}
which proves our assertion. Therefore, we can deduce that
$$
(id-\Phi_{f_1,{\bf W}_1\otimes I_\cH})^{p_1}\cdots
    (id-\Phi_{f_k,{\bf W}_k\otimes I_\cH})^{p_k}(P_\cM)\geq 0
    $$
for any ${\bf p}:=(p_1,\ldots, p_k)\in \ZZ_+^k$ such that
      ${\bf p}\leq (1,\ldots,1)$. Due to the first part  of this corollary,
       we conclude that  $\cM $
 is a Beurling type invariant subspace under the operators
  ${\bf W}_{i,j}\otimes I_\cH$. The proof is complete.
 \end{proof}

Let ${\bf W}:=({\bf W}_1,\ldots, {\bf W}_k)$ be the universal model
 associated to the abstract  noncommutative domain ${\bf D_f^m}$, and let
   $\Phi:(\otimes_{i=1}^k F^2(H_{n_i}))\otimes \cH\to
(\otimes_{i=1}^k F^2(H_{n_i}))\otimes \cK$  be a multi-analytic operator
with respect to ${\bf W}$, i.e.,
 if $\Phi({\bf W}_{i,j}\otimes I_\cH)=({\bf W}_{i,j}\otimes I_\cK)\Phi$
for any $i\in \{1,\ldots,k\}$ and  $j\in \{1,\ldots, n_i\}$.
We introduce the {\it support} of $\Phi$  as the  smallest reducing subspace
 $\supp (\Phi)\subset \otimes_{i=1}^k F^2(H_{n_i}))\otimes \cH$
 under  each operator ${\bf W}_{i,j}$, containing   the co-invariant
  subspace $\cM:=\overline{\Phi^*((\otimes_{i=1}^k F^2(H_{n_i}))\otimes \cK)}$.
  Using Theorem \ref{cyclic} and its proof, we deduce that
 $$
 \supp (\Phi)=\bigvee_{(\alpha)\in \FF_{n_1}^+\times\cdots
  \times \FF_{n_k}^+}({\bf W}_{(\alpha)}\otimes I_\cH) (\cM)
 =(\otimes_{i=1}^k F^2(H_{n_i}))\otimes \cL,
 $$
 where $\cL:=({\bf P}_\CC\otimes I_\cH)\overline{\Phi^*
 ((\otimes_{i=1}^k F^2(H_{n_i}))\otimes \cK)}$.

Assume that ${\bf W}:=({\bf W}_1,\ldots, {\bf W}_k)$ is the universal model associated to the abstract  noncommutative domain ${\bf D_f^m}$. We remark that if
  $\Psi:(\otimes_{i=1}^k F^2(H_{n_i}))\otimes
\cE \to (\otimes_{i=1}^k F^2(H_{n_i}))\otimes \cH$ is an
 isometric multi-analytic operator and
  $\cM=\Psi\left((\otimes_{i=1}^k F^2(H_{n_i}))\otimes \cE\right)$, then
     ${\bf W}\otimes I_\cH|_\cM$
   is doubly commuting. Since this is a straightforward  computation, we  omit it.
   The converse of this implication holds true for the  noncommutative polyball.

\begin{corollary} \label{polyball-Beurling} Let ${\bf W}:=({\bf W}_1,\ldots, {\bf W}_k)$ be
 the universal model
 associated to the noncommutative polyball
 $[B(\cH)^{n_1}]_{1}^-\times_c \cdots \times_c [B(\cH)^{n_k}]_{1}^-$, i.e.,
 ${\bf m}=(1,\ldots,1)$ and $f_i:=Z_{i,1}+\cdots + Z_{i,n_i}$
  for $i\in \{1,\ldots, k\}$. If $\cM\subset (\otimes_{i=1}^k F^2(H_{n_i}))\otimes
  \cH$
 is a nonzero  invariant subspace under the operators
  ${\bf W}_{i,j}\otimes I_\cH$, then  ${\bf W}\otimes I_\cH|_\cM$
   is doubly commuting if and only if there is a Hilbert space $\cL$
   and an isometric multi-analytic operator
    $\Phi:(\otimes_{i=1}^k F^2(H_{n_i}))\otimes
\cL \to (\otimes_{i=1}^k F^2(H_{n_i}))\otimes \cH$ such that
 $\cM=\Phi((\otimes_{i=1}^k F^2(H_{n_i}))\otimes
\cL )$.
\end{corollary}
\begin{proof} Due to the remarks preceding  this corollary, it remains to prove
 the direct implication.
Assume that  ${\bf W}\otimes I_\cH|_\cM$
   is doubly commuting. Corollary \ref{Beurling} and Theorem \ref{Beur-fact} imply
   the existence of an inner multi-analytic operator
    $\Psi:(\otimes_{i=1}^k F^2(H_{n_i}))\otimes
\cE \to (\otimes_{i=1}^k F^2(H_{n_i}))\otimes \cH$ such that
$\cM=\Psi((\otimes_{i=1}^k F^2(H_{n_i}))\otimes
\cE )$. Since ${\bf W}_{i,j}$ are isometries, the initial space of $\Psi$, i.e.,
$\Psi^*((\otimes_{i=1}^k F^2(H_{n_i}))\otimes \cH)
 =\{x\in (\otimes_{i=1}^k F^2(H_{n_i}))\otimes \cE: \ \|\Psi x\|=\|x\|\}$ is
 reducing under each ${\bf W}_{i,j}$. On the other hand,
 the {\it support} of $\Psi$  is the the smallest reducing subspace
 $\supp (\Psi)\subset F^2(H_{n_i}))\otimes \cH$
 under  each operator ${\bf W}_{i,j}$, containing   the co-invariant
  subspace $\Psi^*((\otimes_{i=1}^k F^2(H_{n_i}))\otimes \cH)$.
   Therefore, we must have
   $\supp (\Psi)=\Psi^*((\otimes_{i=1}^k F^2(H_{n_i}))\otimes \cH)$. Note that
   $\Phi:=\Psi|_{\supp (\Psi)}$ is an isometric multi-analytic operator. Since
   $\supp (\Psi)=(\otimes_{i=1}^k F^2(H_{n_i}))\otimes \cL$, where
   $\cL:=({\bf P}_\CC\otimes I_\cE) \Psi^*
 ((\otimes_{i=1}^k F^2(H_{n_i}))\otimes \cH)$ and $\cM=\Phi((\otimes_{i=1}^k F^2(H_{n_i}))\otimes
\cL )$, the proof is complete.
\end{proof}
We remark that in the particular case when $n_1=\cdots=n_k=1$,
 Corollary \ref{polyball-Beurling}  is a Beurling type result for the
   the Hardy space  $H^2(\DD^k)$ of the polydisc, which seems to be new if $k>2$.

We recall that $\cP({\bf W})$  is the set of all polynomials $p({\bf W}_{i,j})$  in  the operators ${\bf W}_{i,j}$, $i\in \{1,\ldots, k\}$, $j\in \{1,\ldots, n_i\}$,  and the identity.

\begin{lemma}\label{irreducible} If ${\bf W}:=({\bf W}_1,\ldots, {\bf W}_k)$ is the universal model associated to the abstract  noncommutative polydomain ${\bf D_f^m}$, then
the  $C^*$-algebra  $C^*({\bf W}_{i,j})$ is
irreducible.
\end{lemma}

\begin{proof}
  Let $\cM\neq\{0\}$ be a
subspace of $\otimes_{i=1}^k F^2(H_{n_i})$, which is jointly reducing
for each operator ${\bf W}_{i,j}$ for all  $i\in \{1,\ldots,k\}$ and $j\in\{1,\ldots,n_i\}$. Let $\varphi\in \cM$,
$\varphi\neq 0$, and assume that
$$
\varphi=\sum\limits_{{\beta_1\in \FF_{n_1}^+,\ldots, \beta_k\in \FF_{n_k}^+}} a_{\beta_1,\ldots, \beta_k}e^1_{\beta_1}\otimes\cdots \otimes  e^k_{\beta_k}.
$$
 If
 $a_{\beta_1,\ldots, \beta_k}$  is a nonzero coefficient of $\varphi$,
then, using relation  \eqref{WbWb}, we deduce that
 $$
{\bf P}_\CC {\bf W}_{1,\beta_1}^*\cdots {\bf W}_{k,\beta_k}^*\varphi= \frac{1}{\sqrt{b_{1,\beta_1}^{(m_1)}}}\cdots \frac{1}{\sqrt{b_{k,\beta_k}^{(m_k)}}}  a_{\beta_1,\ldots, \beta_k}.
$$
 On the other hand, according to Lemma \ref{univ-model},
 $
  (I-\Phi_{q_1,{\bf W}_1})^{m_1}\cdots (I-\Phi_{q_k,{\bf W}_k})^{m_k}(I)={\bf P}_\CC,
$
where ${\bf P}_\CC$ is the
 orthogonal projection from $\otimes_{i=1}^k F^2(H_{n_i})$ onto $\CC 1\subset \otimes_{i=1}^k F^2(H_{n_i})$.
Hence,  and  using the fact that
 $\cM$ is reducing for each  ${\bf W}_{i,j}$,
we deduce that $a_{\beta_1,\ldots, \beta_k}\in \cM$, so $1\in \cM$. Using  again
that $\cM$ is invariant under the operators ${\bf W}_{i,j}$, we
have $\cM=\otimes_{i=1}^k F^2(H_{n_i})$.
This completes the proof.
\end{proof}

 Let ${\bf T}=({ T}_1,\ldots, { T}_k)\in {\bf D_f^m}(\cH)$  and  ${\bf T}'=({ T}_1',\ldots, { T}_k')\in {\bf D_f^m}(\cH')$ be $k$-tuples with $T_i:=(T_{i,1},\ldots, T_{i,n_i})$ and $T_i':=(T_{i,1}',\ldots, T_{i,n_i}')$. We say that  ${\bf T}$ is unitarily equivalent to ${\bf T}'$ if there is a unitary operator $U:\cH\to \cH'$ such that
$
T_{i,j}=U^* T_{i,j}'U
$
for all $i\in\{1,\ldots,k\}$ and $j\in\{1,\ldots, n_i\}$.

\begin{theorem}\label{dil3}
Let ${\bf T}=({ T}_1,\ldots, { T}_k)$ be  a
 pure $k$-tuple in the noncommutative polydomain ${\bf D_f^m}(\cH)$ and let
  $${\bf K_{f,T}}: \cH \to F^2(H_{n_1})\otimes \cdots \otimes  F^2(H_{n_k}) \otimes  \overline{{\bf \Delta_{f,T}^m}(I)(\cH)}$$
  be the noncommutative Berezin kernel.
 Then the  subspace ${\bf K_{f,T}}\cH$ is   co-invariant  under  each operator
${\bf W}_{i,j}\otimes  I_{\overline{{\bf \Delta_{f,T}^m}\cH}}$ for   any $i\in \{1,\ldots, k\}$,  $  j\in \{1,\ldots, n_i\}$ and
    the dilation
  provided by   the relation
    $$ { \bf T}_{(\alpha)}= {\bf K_{f,T}^*}({\bf W}_{(\alpha)}\otimes  I_{\overline{{\bf\Delta_{f,T}^m}(I)(\cH)}})  {\bf K_{f,T}}, \qquad (\alpha) \in \FF_{n_1}^+\times \cdots \times \FF_{n_k}^+,
    $$
  is minimal.
If  \ ${\bf f}= {\bf q}=(q_1,\ldots, q_k)$ is a $k$-tuple of   positive regular  noncommutative polynomials and
\begin{equation*}
\overline{\text{\rm span}}\,\{{\bf W}_{(\alpha)}  {\bf W}_{(\beta)}^*) :\
 (\alpha), (\beta)\in \FF_{n_1}^+\times \cdots \times \FF_{n_k}^+\}=C^*({\bf W}_{i,j}),
\end{equation*}
then  the   minimal dilation  of ${\bf T}$
 is unique up to an isomorphism.
\end{theorem}
 \begin{proof} Due to Theorem \ref{Berezin-prop}, we have
  ${\bf K_{f,T}} { T}^*_{i,j}= ({\bf W}_{i,j}^*\otimes I)  {\bf K_{f,T}}
    $
    for any $i\in \{1,\ldots, k\}$ and $j\in \{1,\ldots, n_i\}$, where the noncommutative Berezin kernel ${\bf K_{f,T}}$  is an isometry.
On the other hand, the
definition of the Berezin kernel ${\bf K_{f,T}}$  implies
$$
 ({\bf P}_\CC \otimes I_{\overline{{\bf\Delta_{f,T}^m}(I)(\cH)}}) K_{\bf f,T}\cH= \overline{{\bf\Delta_{f,T}^m}(I)(\cH)}.
$$
   Using
Theorem \ref{cyclic} in the particular case when $\cM:=K_{\bf f,T}\cH$
and $\cE:=\overline{{\bf\Delta_{f,T}^m}(I)(\cH)}$, we deduce that the subspace
${\bf K_{f,T}}\cH$ is cyclic for  ${\bf W}_{i,j}\otimes I_\cE$ for $i\in \{1,\ldots, k\}$ and $j\in \{1,\ldots, n_i\}$,
which proves the minimality of the  dilation, i.e.,
\begin{equation}\label{minimal1}
(\otimes_{i=1}^k F^2(H_{n_i}))\otimes \overline{{\bf\Delta_{f,T}^m}(I)(\cH)}=\bigvee_{(\alpha)\in \FF_{n_1}^+\times \cdots \times \FF_{n_k}^+}
({\bf W}_{(\alpha)}\otimes I_{\overline{{\bf\Delta_{f,T}^m}(I)(\cH)}}) {\bf K_{f,T}}\cH.
\end{equation}

To prove the last part of the theorem, assume that
${\bf f}= {\bf q}=(q_1,\ldots, q_k)$ is a $k$-tuple of   positive
 regular  noncommutative polynomials  and that  the relation in the theorem  holds. Consider
another minimal   dilation of ${\bf T}$, i.e.,
\begin{equation}
\label{another} { \bf T}_{(\alpha)}=V^* ({ \bf W}_{(\alpha)}\otimes I_\cD)V, \qquad (\alpha) \in \FF_{n_1}^+\times \cdots \times \FF_{n_k}^+,
\end{equation}
where $V:\cH\to (\otimes_{i=1}^k F^2(H_{n_i}))\otimes \cD$ is an isometry, $V\cH$ is
co-invariant under each operator ${\bf W}_{i,j}\otimes I_\cD$, and
\begin{equation}\label{minimal2}
(\otimes_{i=1}^k F^2(H_{n_i}))\otimes \cD=\bigvee_{(\alpha)\in \FF_{n_1}^+\times \cdots \times \FF_{n_k}^+} ({\bf W}_{(\alpha)}\otimes
I_{\cD}) V\cH.
\end{equation}
Due to Theorem \ref{Poisson-C*}, there exists a unique unital completely positive
linear map
${\bf \Psi_{q,T}}: C^*({\bf W}_{i,j})\to B(\cH)$
 such that
$$
{\bf \Psi_{ q,T}}\left(\sum_{\gamma=1}^s p_\gamma({\bf W}_{i,j})q_\gamma({\bf W}_{i,j})^*\right)= \sum_{\gamma=1}^s p_\gamma(T_{i,j})q_\gamma(T_{i,j})^*
$$
for any $p_\gamma({\bf W}_{i,j}),q_\gamma({\bf W}_{i,j}) \in  \cP({\bf W})$ and $ s\in \NN$.   Consider the $*$-representations
$$
\pi_1:C^*({\bf W}_{i,j})\to B((\otimes_{i=1}^k F^2(H_{n_i}))\otimes \overline{{\bf\Delta_{q,T}^m}(I)(\cH)}),\quad \pi_1(X)
:= X\otimes I_{\overline{{\bf\Delta_{q,T}^m}(I)(\cH)}}
$$
  and
  $$
\pi_2:C^*({\bf W}_{i,j})\to B((\otimes_{i=1}^k F^2(H_{n_i}))\otimes \cD),\quad \pi_2(X):=
X\otimes I_{\cD}.
$$
Since the   subspaces ${\bf K_{q,T}}\cH$  and $V\cH$  are     co-invariant  for each operator
${\bf W}_{i,j}\otimes  I_{\overline{{\bf \Delta_{q,T}^m}(I)\cH}}$,
 the
 relation    \eqref{another}  implies
$$
{\bf \Psi_{q,T}}(X)={\bf K_{q,T}^*}\pi_1(X){\bf K_{q,T}}
=V^*\pi_2(X)V,\qquad
\  X\in C^*({\bf W}_{i,j}).
$$
Due to relations \eqref{minimal1} and
\eqref{minimal2},  we deduce that $\pi_1$
and $\pi_2$ are  minimal Stinespring dilations of the completely
 positive linear map
 ${\bf \Psi_{q,T}}$.  Since
these representations are unique up to an isomorphism, there exists a unitary operator $U:(\otimes_{i=1}^k F^2(H_{n_i}))\otimes
\overline{{\bf\Delta_{q,T}^m}(I)(\cH)}\to (\otimes_{i=1}^k F^2(H_{n_i}))\otimes \cD$ such that
\begin{equation*}
U({\bf W}_{i,j}\otimes I_{\overline{{\bf\Delta_{q,T}^m}(I)(\cH)}})=({\bf W}_{i,j}\otimes
I_\cD)U
\end{equation*}
and $U{\bf K_{q,T}}=V$.  Taking into account that $U$ is unitary, we deduce that
$$
U({\bf W}_{i,j}^*\otimes I_{\overline{{\bf\Delta_{q,T}^m}(I)(\cH)}})=({\bf W}_{i,j}^*\otimes
I_\cD)U.
$$
Since  $C^*({\bf W}_{i,j})$ is irreducible (see
Lemma \ref{irreducible}),  we must have $U=I\otimes Z$, where
$Z\in B(\overline{{\bf\Delta_{q,T}^m}(\cH)},\cD)$ is a unitary operator.
This implies that $\dim \overline{{\bf\Delta_{q,T}^m}(\cH)}=\dim\cD$ and
$U{\bf K_{q,T}}\cH=V\cH$, which proves that the two dilations are
unitarily equivalent.
 The proof is complete.
\end{proof}

 Let $\cD$ be a Hilbert space such that the Hilbert space  $\cH$ can be identified with
a co-invariant subspace of $(\otimes_{i=1}^k F^2(H_{n_i}))\otimes \cD$
 under each operator
${\bf W}_{i,j}\otimes  I_{\cD}$ for   any $i\in \{1,\ldots, k\}$,
 $  j\in \{1,\ldots, n_i\}$ and such that
 ${ \bf T}_{(\alpha)}=V^* ({ \bf W}_{(\alpha)}\otimes I_\cD)V$
 for $ (\alpha) \in \FF_{n_1}^+\times \cdots \times \FF_{n_k}^+$. The dilation index of ${\bf T}$ is the minimum dimension  of $\cD$  with the above mentioned property.
 We remark that
 the dilation index of ${\bf T}$ coincides with $\text{\rm rank}\,
 {\bf \Delta_{f,T}^m}(I)$.
  Indeed, since
 ${\bf \Delta_{f,W}^m}(I)={\bf P}_\CC$, where ${\bf P}_\CC$ is the
 orthogonal projection from $\otimes_{i=1}^k F^2(H_{n_i})$ onto
  $\CC 1\subset \otimes_{i=1}^k F^2(H_{n_i})$,
  we deduce that
 $
 {\bf \Delta_{f,T}^m}(I)= P_\cH \left[  {\bf P}_\CC \otimes I_\cD
\right]|\cH.
 $
Hence, $\rank {\bf \Delta_{f,T}^m}(I)\leq \dim \cD$. Now, Theorem \ref{dil3} implies
that the dilation index of $T$ is equal to $\rank {\bf \Delta_{f,T}^m}(I)$.

\begin{proposition}\label{rank-n} Let ${\bf q}=(q_1,\ldots, q_k)$ is a $k$-tuple of   positive regular  noncommutative polynomials such that
  \begin{equation*}
\overline{\text{\rm span}}\,\{{\bf W}_{(\alpha)}  {\bf W}_{(\beta)}^*) :\
 (\alpha), (\beta)\in \FF_{n_1}^+\times \cdots \times \FF_{n_k}^+\}=C^*({\bf W}_{i,j}).
\end{equation*}
A
 pure $k$-tuple ${\bf T}=({ T}_1,\ldots, { T}_k)\in {\bf D_q^m}(\cH)$
    has
    $\rank {\bf \Delta_{q,T}^m}(I)=n$,  $ n=1,2,\ldots, \infty,
     $
     if and only if it is
unitarily equivalent to one obtained by compressing $({\bf W}_1\otimes
I_{\CC^n},\ldots, {\bf W}_k\otimes I_{\CC^n})$ to a co-invariant subspace
$\cM\subset  (\otimes_{i=1}^k F^2(H_{n_i}))\otimes \CC^n$  under each operator
 ${\bf W}_{i,j}\otimes
I_{\CC^n}$,  $i\in \{1,\ldots, k\}$ and   $j\in \{1,\ldots, n_i\}$,
 with the property that $\dim [({\bf P}_\CC\otimes I_{\CC^n})\cM]=n$,
where ${\bf P}_\CC$ is the orthogonal projection from $\otimes_{i=1}^k F^2(H_{n_i}) $
onto  $\CC 1$.
\end{proposition}
\begin{proof}
  The direct implication  is a consequence
of Theorem \ref{dil3}. To prove the converse, assume that
$$
{ \bf T}_{(\alpha)}=P_\cH ({ \bf W}_{(\alpha)}\otimes I_{\CC^n})|_\cH,\qquad  (\alpha) \in \FF_{n_1}^+\times \cdots \times \FF_{n_k}^+
$$
where $\cH\subset (\otimes_{i=1}^k F^2(H_{n_i})) \otimes \CC^n$ is a co-invariant subspace
under each operator
${\bf W}_{i,j}\otimes  I_{\CC^n}$ for   any $i\in \{1,\ldots, k\}$,
 $  j\in \{1,\ldots, n_i\}$ , such
that $\dim ({\bf P}_\CC\otimes I_{\CC^n})\cH=n$.  Note that ${\bf T}$
 is  a pure element in the noncommutative polydomain ${\bf D_q^m}(\cH)$.
  First, we consider   the case when
$n<\infty$. Since
$({\bf P}_\CC\otimes I_{\CC^n})\cH\subseteq \CC^n$ and $\dim ({\bf P}_\CC\otimes I_{\CC^n})\cH=n$, we
deduce that $({\bf P}_\CC\otimes I_{\CC^n})\cH=\CC^n$. This
condition is equivalent  to the equality
 $
  \cH^\perp\cap \CC^n=\{0\}.
$
Since
 ${\bf \Delta_{q,W}^m}(I)={\bf P}_\CC$, where ${\bf P}_\CC$ is the
 orthogonal projection from $\otimes_{i=1}^k F^2(H_{n_i})$ onto
  $\CC 1\subset \otimes_{i=1}^k F^2(H_{n_i})$,
  we deduce that
 $
 {\bf \Delta_{q,T}^m}(I)= P_\cH \left[  {\bf P}_\CC \otimes I_{\CC^n}
\right]|_\cH= P_\cH \CC^n.
 $
Consequently, we have  $\rank  {\bf \Delta_{q,T}^m}(I)=\dim P_\cH \CC^n$.
 If we assume that $\rank{\bf \Delta_{q,T}^m}(I)<n$, then there exists  $h\in
\CC^n$, $h\neq 0$, with $P_\cH h=0$.  This contradicts the fact that
$\cH^\perp\cap \CC^n=\{0\}$. Therefore, we must have $\rank {\bf \Delta_{q,T}^m}(I)=n$.

Now, we consider the case when $n=\infty$. According to Theorem
\ref{cyclic} and its proof,   we have
$$
(\otimes_{i=1}^k F^2(H_{n_i}))\otimes \cE=\bigvee_{(\alpha)\in \FF_{n_1}^+\times \cdots \times \FF_{n_k}^+} ({\bf W}_{(\alpha)}\otimes
I_{\CC^n}) \cH
$$
where $\cE:=({\bf P}_\CC\otimes I_{\CC^n})\cH$.  Since $(\otimes_{i=1}^k F^2(H_{n_i}))\otimes \cE$  is reducing for each operator
${\bf W}_{i,j}\otimes I_{\CC^m}$,   we deduce that
$
{ \bf T}_{(\alpha)}=P_\cH ({ \bf W}_{(\alpha)}\otimes I_\cE)|_\cH,$
$ (\alpha) \in \FF_{n_1}^+\times \cdots \times \FF_{n_k}^+.
$
The uniqueness of the  minimal   dilation of
${\bf T}$ (see  Theorem \ref{dil3}) implies
$\dim \overline{ {\bf \Delta_{q,T}^m(I)}\cH}=\dim\cE=\infty.
$
This completes the proof.
\end{proof}

 We can
characterize now the pure  $n$-tuples of operators in the
noncommutative  polydomain ${\bf D_q^m}(\cH)$, having  rank one, i.e.,
$\rank {\bf \Delta_{q,T}^m}(I)=1$.

\begin{corollary}\label{rank1} Under the hypothesis  of Proposition \ref{rank-n}, the following statements hold.
 \begin{enumerate}
 \item[(i)]
If $\cM\subset \otimes_{i=1}^k F^2(H_{n_i})$ is  a co-invariant  subspace under each operator  ${\bf W}_{i,j}$, where  $i\in \{1,\ldots, k\}$ and   $j\in \{1,\ldots, n_i\}$,  then
$$
{\bf T}:=(T_1,\ldots, T_k), \quad  T _i:=(P_\cM {\bf W}_{i,1}|_\cM,\ldots,
 P_\cM{\bf W}_{i,n_i}|_\cM),
$$
is a pure    $k$-tuple   in ${\bf D_q^m}(\cM)$ such that
$\rank {\bf \Delta_{q,T}^m}=1$.
\item[(ii)]
If $\cM'$ is another co-invariant subspace under   each operator ${\bf W}_{i,j}$, which gives rise to an $k$-tuple   ${\bf T}'$, then ${\bf T}$
and ${\bf T}'$ are unitarily equivalent if and only if $\cM=\cM'$.
\end{enumerate}
\end{corollary}
\begin{proof} Since  $
 {\bf \Delta_{q,T}^m}(I)= P_\cM  {\bf P}_\CC|_\cM
 $
we have $\rank  {\bf \Delta_{q,T}^m}(I) \leq 1$.
  On the other hand, it is clear that ${\bf T}$ is pure.
  This
also implies that ${\bf \Delta_{q,T}^m}(I)\neq 0$, so
 $\rank {\bf \Delta_{q,T}^m}(I)\geq 1$. Therefore,  $\rank {\bf \Delta_{q,T}^m}(I)=1$.

To prove  (ii), note that,  as in the proof of  Theorem
\ref{dil3}, one can show that ${\bf T}$ and ${\bf T}'$ are unitarily
equivalent if and only if there exists a unitary operator
$\Lambda:\otimes_{i=1}^k F^2(H_{n_i})\to \otimes_{i=1}^k F^2(H_{n_i})$ such that
$
\Lambda {\bf W}_{i,j}={\bf W}_{i,j} \Lambda$,
   $i\in \{1,\ldots, k\}$,  $j\in \{1,\ldots, n_i\},
$
and $\Lambda \cM=\cM'$.
Hence $\Lambda {\bf W}_{i,j}^*={\bf W}_{i,j}^* \Lambda$. Since
$C^*({\bf W}_{i,j})$ is irreducible (see Theorem \ref{compact}),
$\Lambda$ must be a scalar multiple of the identity. Therefore, we have
$\cM=\Lambda \cM=\cM'$.
\end{proof}

\bigskip

\section{Characteristic functions and operator models}

We  provide a characterization for the class of tuples of  operators in  ${\bf D_f^m}(\cH)$ which admit  characteristic functions.
  We prove that  the characteristic function is  a complete unitary invariant  for the class of completely non-coisometric tuples   and provide an operator model for this class of elements  in terms of their   characteristic functions.

Let
${\bf W}:=({\bf W}_1,\ldots, {\bf W}_k)$  be the  the universal
 model associated with  the abstract noncommutative domain ${\bf D_f^m}$.
  We say that  two
multi-analytic operator $\Phi:(\otimes_{i=1}^k F^2(H_{n_i}))\otimes
 \cK_1 \to (\otimes_{i=1}^k F^2(H_{n_i}))\otimes \cK_2$ and
 $\Phi':(\otimes_{i=1}^k F^2(H_{n_i}))\otimes \cK_1' \to
  (\otimes_{i=1}^k F^2(H_{n_i}))\otimes \cK_2'$  coincide if there are two unitary operators $\tau_j\in
B(\cK_j, \cK_j')$ such that
$$
\Phi'(I_{\otimes_{i=1}^k F^2(H_{n_i})}\otimes \tau_1)
=(I_{\otimes_{i=1}^k F^2(H_{n_i})}\otimes \tau_2) \Phi.
$$

\begin{lemma}\label{fifi*}   Let $\Phi_s:(\otimes_{i=1}^k F^2(H_{n_i}))\otimes \cH_s\to
(\otimes_{i=1}^k F^2(H_{n_i}))\otimes \cK$, \ $s=1,2$,    be  multi-analytic operators   with respect to ${\bf W}:=({\bf W}_1,\ldots, {\bf W}_k)$ such that
$
\Phi_1 \Phi_1^*=\Phi_2 \Phi_2^*.
$
Then there is a unique partial isometry $V:\cH_1\to \cH_2$ such that
$$\Phi_1=\Phi_2(I_{\otimes_{i=1}^k F^2(H_{n_i})}\otimes V),
$$
where $(I_{\otimes_{i=1}^k F^2(H_{n_i})}\otimes V)$ is an inner multi-analytic operator  with initial space $\supp (\Phi_1)$ and  final space $\supp (\Phi_2)$.
In particular,  the  multi-analytic operators $\Phi_1|_{\supp (\Phi_1)}$ and $\Phi_2|_{\supp (\Phi_2)}$ coincide.
\end{lemma}

\begin{proof} Due to Lemma \ref{univ-model},
$(id-\Phi_{f_1,{\bf W}_1})^{m_1}\cdots (id-\Phi_{f_k,{\bf W}_k})^{m_k}(I)
={\bf P}_\CC$, where ${\bf P}_\CC$ is the
 orthogonal projection from $\otimes_{i=1}^k F^2(H_{n_i})$ onto $\CC 1\subset \otimes_{i=1}^k F^2(H_{n_i})$. Since $\Phi_1, \Psi_2$ are  multi-analytic operators   with respect to ${\bf W}$, we deduce that
 $\Phi_1({\bf P}_\CC\otimes I_{\cH_1})\Phi_1^*=\Phi_2({\bf P}_\CC\otimes I_{\cH_2})\Phi_2^*$.
 Consequently, we have
 $$
 \|({\bf P}_\CC\otimes I_{\cH_1})\Phi_1^*x\|=\|({\bf P}_\CC\otimes I_{\cH_2})\Phi_2^*x\|, \qquad x\in (\otimes_{i=1}^k F^2(H_{n_i}))\otimes \cK.
 $$
 Set $\cL_s:=({\bf P}_\CC\otimes I_{\cH_s})\overline{\Phi_s^*((\otimes_{i=1}^k F^2(H_{n_i}))\otimes \cK)}$, $s=1,2$, and  define the unitary  operator
 $U:\cL_1\to \cL_2$ by
 $$
 U({\bf P}_\CC\otimes I_{\cH_1})\Phi_1^*x:=({\bf P}_\CC\otimes I_{\cH_2})\Phi_2^*x, \qquad x\in (\otimes_{i=1}^k F^2(H_{n_i}))\otimes \cK.
 $$
 This implies that there is a unique  partial isometry $V:\cH_1\to \cH_2$ with initial space $\cL_1$ and final space $\cL_2$, extending $U$. Moreover, we have
 $\Phi_1 V^*= \Phi_2|_{1\otimes \cH_2}$.   Since $\Phi_1, \Psi_2$ are  multi-analytic operators   with respect to ${\bf W}$, we deduce that
 $\Phi_1(I_{\otimes_{i=1}^k F^2(H_{n_i})}\otimes V^*)=\Phi_2$. Hence, the result follows.
Now, the last part of the lemma is clear.
\end{proof}

We say that  ${\bf T}=({ T}_1,\ldots, { T}_k)\in {\bf D_f^m}(\cH)$  has
  characteristic function   if there is a Hilbert space $\cE$ and
   a multi-analytic operator $\Psi:(\otimes_{i=1}^k F^2(H_{n_i}))\otimes \cE \to (\otimes_{i=1}^k F^2(H_{n_i}))\otimes \overline{{\bf \Delta_{f,T}^m}(I) (\cH)}$ with respect to $W_{i,j}$, $i\in \{1,\ldots,k\}$, $j\in \{1,\ldots, n_i\}$, such that
$$
{\bf K_{f,T}}{\bf K_{f,T}^*} +\Psi \Psi^*=I.
$$
 According to Lemma \ref{fifi*}, if there is a characteristic function
  for ${\bf T}\in {\bf D_f^m}(\cH)$, then it is essentially unique.

We give  now   an example of  a class of elements ${\bf T} \in {\bf D_f^m}(\cH)$
which have characteristic function.  Let $\Psi:(\otimes_{i=1}^k F^2(H_{n_i}))\otimes \cE\to
(\otimes_{i=1}^k F^2(H_{n_i}))\otimes \cG$ be
an inner multi-analytic operator with $\Psi(0)=0$ and consider the subspace  $\cM:=\Psi((\otimes_{i=1}^k F^2(H_{n_i}))\otimes \cE)$.
Note that $\cM$ is invariant under each operator ${\bf W}_{i,j}$ and define
$T_{i,j}:=P_{\cM^\perp}({\bf W}_{i,j}\otimes I_\cG)|_{\cM^\perp}$ for $i\in \{1,\ldots, k\}$ and $j\in \{1,\ldots, n_i\}$.
Set ${\bf T}:=(T_1,\ldots, T_k)$, where $T_i=(T_{i,1},\ldots, T_{i,j})$,  and note that
$$
{\bf \Delta_{f,T}^m}(I_{\cM^\perp})
=P_{\cM^\perp}{\bf \Delta}_{{\bf f,W}\otimes I_\cG}^{\bf m}(I_\cG)|_{\cM^\perp}=P_{\cM^\perp} ({\bf P}_\CC \otimes I_{\cG})|_{\cM^\perp}.
$$
Since $\Psi(0)=0$, we have $1\otimes \cG\subset \cM^\perp$ and, consequently,
${\bf \Delta_{f,T}^m}(I_{\cM^\perp})^{1/2}=({\bf P}_\CC \otimes I_{\cG})|_{\cM^\perp}$.
Consider an arbitrary vector
$$
h=\sum_{\beta_i\in \FF_{n_i}^+, i=1,\ldots,k}e^1_{\beta_1}\otimes \cdots \otimes  e^k_{\beta_k} \otimes h_{\beta_1,\ldots, \beta_k}
$$
   in $\cM^\perp\subset (\otimes_{i=1}^k F^2(H_{n_i}))\otimes \cG$. Using the definition of the noncommutative Berezin kernel and relation \eqref{WbWb}, we obtain
\begin{equation*}
\begin{split}
 {\bf K_{f,T}}h&:=\sum_{\beta_i\in \FF_{n_i}^+, i=1,\ldots,k}
   \sqrt{b_{1,\beta_1}^{(m_1)}}\cdots \sqrt{b_{k,\beta_k}^{(m_k)}}
   e^1_{\beta_1}\otimes \cdots \otimes  e^k_{\beta_k}\otimes
    ({\bf P}_\CC\otimes I_\cG) ({\bf W}_{1,\beta_1}^*\cdots {\bf W}_{k,\beta_k}^*\otimes I_\cG)^*h\\
   &=
   \sum_{\beta_i\in \FF_{n_i}^+, i=1,\ldots,k}
   \sqrt{b_{1,\beta_1}^{(m_1)}}\cdots \sqrt{b_{k,\beta_k}^{(m_k)}}
   e^1_{\beta_1}\otimes \cdots \otimes  e^k_{\beta_k}\otimes  \frac{1}{\sqrt{b_{1,\beta_1}^{(m_1)}}}\cdots \frac{1}{\sqrt{b_{k,\beta_k}^{(m_k)}}}(1\otimes h_{\beta_1,\ldots, \beta_k})=h
   \end{split}
 \end{equation*}
 Consequently, ${\bf K_{f,T}}$ can be identified with the injection of $\cM^\perp$ into $(\otimes_{i=1}^k F^2(H_{n_i}))\otimes \cG$, and
 ${\bf K_{f,T}}{\bf K_{f,T}^*}$ can be identified with the orthogonal projection
 $P_{\cM^\perp}$. Therefore,  ${\bf K_{f,T}}{\bf K_{f,T}^*}+\Psi\Psi^*=I$, which proves our assertion.

We also remark that in the particular case when $k=1$ and $m_1=1$,
all the elements in the
 noncommutative domain ${\bf D}_{f_1}^{1}$ have characteristic functions.

\begin{theorem} A $k$-tuple  ${\bf T}=({ T}_1,\ldots, { T}_k)$  in  the noncommutative polydomain ${\bf D_f^m}(\cH)$  admits a characteristic function if and only if
$$
{\bf \Delta}_{{\bf f,W}\otimes I}^{\bf p}(I -{\bf K_{f,T}}{\bf K_{f,T}^*})\geq 0
$$
for  any ${\bf p}:=(p_1,\ldots, p_k)\in \ZZ_+^k$ such that ${\bf p}\leq {\bf m}$, where ${\bf K_{f,T}}$ is the noncommutative Berezin kernel  associated with ${\bf T}$.
\end{theorem}

\begin{proof}
 If ${\bf T}$
 has characteristic function, then there is
  a multi-analytic operator $\Psi$  with the property  that
$
{\bf K_{f,T}}{\bf K_{f,T}^*} +\Psi \Psi^*=I.
$
Using the multi-analyticity of $\Psi$,
 we have
$$
{\bf \Delta}_{{\bf f,W}\otimes I}^{\bf p}(I -{\bf K_{f,T}}{\bf K_{f,T}^*})
=\Psi{\bf \Delta}_{{\bf f,W}\otimes I}^{\bf p}(I)\Psi^*\geq 0,
$$
for  any ${\bf p}:=(p_1,\ldots, p_k)\in \ZZ_+^k$ such that ${\bf p}\leq {\bf m}$.
 For the converse, we apply Theorem \ref{Beur-fact} to the operator
  $Y=I -{\bf K_{f,T}}{\bf K_{f,T}^*}$ and  complete the proof.
\end{proof}

 If ${\bf T}$ has characteristic function, the multi-analytic operator $M$ provided by the  proof of Theorem \ref{Beur-fact} when $Y=I -{\bf K_{f,T}}{\bf K_{f,T}^*}$, which we denote by $\Theta_{\bf f,T}$,  is called  the {\it characteristic function} of ${\bf T}$. More precisely, $\Theta_{\bf f,T}$
  is the multi-analytic operator
 $$\Theta_{\bf f,T}:(\otimes_{i=1}^k F^2(H_{n_i}))\otimes \overline{{\bf \Delta_{f,M_T}^m}(I)(\cM_T)} \to (\otimes_{i=1}^k F^2(H_{n_i}))\otimes \overline{{\bf \Delta_{f,T}^m}(I)(\cH)}
 $$
defined by $\Theta_{\bf f,T}:=(I -{\bf K_{f,T}}{\bf K_{f,T}^*})^{1/2} {\bf K_{f,M_T}^*}$, where
$${\bf K_{f,T}}: \cH \to F^2(H_{n_1})\otimes \cdots \otimes  F^2(H_{n_k}) \otimes  \overline{{\bf \Delta_{f,T}^m}(I)(\cH)}$$
is the noncommutative Berezin kernel associated with ${\bf T}$ and
$${\bf K_{f,M_T}}: \cH \to F^2(H_{n_1})\otimes
\cdots \otimes  F^2(H_{n_k}) \otimes  \overline{{\bf \Delta_{f,M_T}^m}(I)(\cM_{\bf T})}$$
is the noncommutative Berezin kernel associated
 with ${\bf M_T}\in {\bf D_f^m}(\cM_{\bf T})$.  Here, we have
$$
\cM_{\bf T}:= \overline{{\rm range}\,(I -{\bf K_{f,T}}{\bf K_{f,T}^*}) }
$$
and  ${\bf M_T}:=(M_1,\ldots, M_k)$ is the $k$-tuple
 with $M_i:=(M_{i,1},\ldots, M_{i,n_i})$ and $M_{i,j}\in B(\cM_{\bf T})$   given by $M_{i,j}:=A_{i,j}^*$, where $A_{i,j}\in B(\cM_{\bf T})$ is uniquely defined by
$$
A_{i,j}\left[(I -{\bf K_{f,T}}{\bf K_{f,T}^*})^{1/2}x\right]:=(I -{\bf K_{f,T}}{\bf K_{f,T}^*})^{1/2}({\bf W}_{i,j}\otimes I)x
$$
for any $x\in (\otimes_{i=1}^k F^2(H_{n_i}))\otimes \overline{{\bf \Delta_{f,T}^m}(I)(\cH)}$. According to Theorem \ref{Beur-fact}, we have
$
{\bf K_{f,T}}{\bf K_{f,T}^*}+ \Theta_{\bf f,T}\Theta_{\bf f, T}^*=I.
$

We denote by $\cC_{\bf f}^{\bf m}(\cH)$ the set of all ${\bf T}=({ T}_1,\ldots, { T}_k)\in {\bf D_f^m}(\cH)$  which admit characteristic functions.
\begin{theorem}\label{pure-model} Let  ${\bf T}=({ T}_1,\ldots, { T}_k)$ be a  $k$-tuple  in
$\cC_{\bf f}^{\bf m}(\cH)$. Then ${\bf T}$  is pure if and only if  the characteristic function $\Theta_{\bf f,T}$ is an inner multi-analytic operator. Moreover, in this case  ${\bf T}=({ T}_1,\ldots, { T}_k)$ is unitarily equivalent to ${\bf G}=({ G}_1,\ldots, { G}_k)$, where $G_i:=(G_{i,1},\ldots, G_{i,n_i})$ is defined by
$$ G_{i,j}:=P_{\bf H_{f,T}} \left({\bf W}_{i,j}\otimes I\right)|_{\bf H_{f,T}}
$$
and $P_{\bf H_{f,T}}$ is the orthogonal projection of $(\otimes_{i=1}^k F^2(H_{n_i}))\otimes \overline{{\bf \Delta_{f,T}^m}(I)(\cH)}$ onto
$${\bf H_{f,T}}:=\left\{(\otimes_{i=1}^k F^2(H_{n_i}))\otimes \overline{{\bf \Delta_{f,T}^m}(I)(\cH)}\right\}\ominus {\rm range}\, \Theta_{\bf f,T}.
$$
\end{theorem}
\begin{proof} Assume that ${\bf T}$ is a pure $k$-tuple in $\cC_{\bf f}^{\bf m}(\cH)$.    Theorem \ref{Berezin-prop} shows that
 the noncommutative Berezin kernel associated with ${\bf T}$, i.e.,
$${\bf K_{f,T}}: \cH \to F^2(H_{n_1})\otimes \cdots \otimes  F^2(H_{n_k}) \otimes  \overline{{\bf \Delta_{f,T}^m}(I)(\cH)}$$
is an isometry, the subspace ${\bf K_{f,T}}\cH$ is coinvariant  under  the operators
${\bf W}_{i,j}\otimes I_{\overline{{\bf \Delta_{f,T}^m}(I)(\cH)}}$, $i\in \{1,\ldots, k\}$, $j\in \{1,\ldots, n_i\}$, and
$T_{i,j}={\bf K_{f,T}^*}({\bf W}_{i,j}\otimes I_{\overline{{\bf \Delta_{f,T}^m}(I)(\cH)}}) {\bf K_{f,T}}$.
Since ${\bf K_{f,T}}{\bf K_{f,T}^*}$ is the orthogonal projection of $(\otimes_{i=1}^k F^2(H_{n_i}))\otimes \overline{{\bf \Delta_{f,T}^m}(I)(\cH)}$ onto ${\bf K_{f,T}}\cH$ and  ${\bf K_{f,T}}{\bf K_{f,T}^*}+\Theta_{\bf f,T}\Theta_{\bf f,T}^*=I $, we deduce that $\Theta_{\bf f,T}$ is a partial isometry and
${\bf K_{f,T}}\cH={\bf H_{f,T}}$. Since ${\bf K_{f,T}}$ is an isometry, we can identify $\cH$ with ${\bf K_{f,T}}\cH$. Therefore, ${\bf T}=({ T}_1,\ldots, { T}_k)$ is unitarily equivalent to ${\bf G}=({ G}_1,\ldots, { G}_k)$.

Conversely,  if we assume that  $\Theta_{\bf f,T}$ is inner, then it
is a partial isometry. Due to the fact that $
{\bf K_{f,T}}{\bf K_{f,T}^*}+ \Theta_{\bf f,T}\Theta_{\bf f, T}^*=I,
$
 the noncommutative Berezin kernel ${\bf K_{f,T}}$ is a partial isometry. On the other hand,
 since ${\bf T}$ is completely non-coisometric,
  ${\bf K_{f,T}}$ is a one-to-one partial isometry and,
therefore, isometry. Due to  Theorem  \ref{Berezin-prop}, we have
$$
{\bf K_{f,T}^*}{\bf K_{f,T}}=
\lim_{{\bf q}=(q_1,\ldots, q_k)\in \ZZ_+^k}(id-\Phi_{f_k,T_k}^{q_k})\cdots (id-\Phi_{f_1,T_1}^{q_1})(I)=I
$$
Consequently,  ${\bf T}$ is a pure $k$-tuple. The proof is complete.
\end{proof}

Now, we are able to provide a model theorem for class of  the  completely non-coisometric $k$-tuple  of operators  in
$\cC_{\bf f}^{\bf m}(\cH)$.

\begin{theorem}\label{model}  Let ${\bf T}=({ T}_1,\ldots, { T}_k)$ be a
 completely non-coisometric $k$-tuple  in
$\cC_{\bf f}^{\bf m}(\cH)$  and let  ${\bf W}:=({\bf W}_1,\ldots, {\bf W}_k)$ be the universal model associated to the abstract  noncommutative domain ${\bf D_f^m}$.   Set
$$
\cD:=\overline{{\bf \Delta_{f,T}^m}(I)(\cH)},\quad  \quad \cD_*:=\overline{{\bf \Delta_{f,M_T}^m}(I)(\cM_T)},
$$
and $\Delta_{\Theta_{\bf f,T}}:= \left(I-\Theta_{\bf f,T}^*
\Theta_{\bf f,T}\right)^{1/2}$, where $\Theta_{\bf f,T}$ is the characteristic function of ${\bf T}$.
 Then ${\bf T} $ is unitarily equivalent to
$\TT:=(\TT_1,\ldots, \TT_k)\in \cC_{\bf f}^{\bf m}(\HH_{\bf f,T})$, where $\TT_i:=(\TT_{i,1},\ldots, \TT_{i,n_i})$ and $\TT_{i,j}$ is a bounded operator acting on the
Hilbert space
\begin{equation*}
\begin{split}
\HH_{\bf f,T}&:=\left[\left((\otimes_{i=1}^k F^2(H_{n_i}))\otimes\cD\right)\oplus
\overline{\Delta_{\Theta_{\bf f,T}}((\otimes_{i=1}^k F^2(H_{n_i}))\otimes \cD_*)}\right]\\
& \qquad \qquad\ominus\left\{\Theta_{\bf f,T}\varphi\oplus
\Delta_{\Theta_{\bf f,T}}\varphi:\ \varphi\in (\otimes_{i=1}^k F^2(H_{n_i}))\otimes  \cD_*\right\}
\end{split}
\end{equation*}
 and is uniquely defined by the relation
$$
\left( P_{(\otimes_{i=1}^k F^2(H_{n_i}))\otimes\cD}|_{\HH_{\bf f,T}}\right) \TT_{i,j}^*x=
({\bf W}_{i,j}^*\otimes I_{ \cD})\left( P_{(\otimes_{i=1}^k F^2(H_{n_i}))\otimes \cD}|_{\HH_{\bf f,T}}\right)x
$$
for any $x\in \HH_{\bf f,T}$. Here,
   $ P_{(\otimes_{i=1}^k F^2(H_{n_i}))\otimes  \cD}$ is the orthogonal
projection of the Hilbert space
$$\cK_{\bf f,T}:=\left((\otimes_{i=1}^k F^2(H_{n_i}))\otimes\cD\right)\oplus
\overline{\Delta_{\Theta_{\bf f,T}}((\otimes_{i=1}^k F^2(H_{n_i}))\otimes \cD_*)}$$
 onto
the subspace $(\otimes_{i=1}^k F^2(H_{n_i}))\otimes \cD$.
 \end{theorem}

\begin{proof}
   First, we show that there is a unique  unitary
operator $\Gamma:\cH\to \HH_{\bf f,T}$ such that
\begin{equation}\label{Ga}
\Gamma({\bf K_{f,T}^*} g)=P_{\HH_{\bf f,T}}(g\oplus 0), \qquad  g\in
(\otimes_{i=1}^k F^2(H_{n_i}))\otimes \cD,
\end{equation}
where $P_{\HH_{\bf f,T}}$ the orthogonal projection of
$\cK_{\bf f,T}$ onto the subspace $\HH_{\bf f,T}$.
 Indeed, note that the operator $\Phi: (\otimes_{i=1}^k F^2(H_{n_i}))\otimes \cD\to
\cK_{\bf f,T}$ defined by
$$
\Phi \varphi:=\Theta_{\bf f,T} \varphi\oplus \Delta_{\Theta_{\bf f,T}}
\varphi,\quad \varphi\in (\otimes_{i=1}^k F^2(H_{n_i}))\otimes \cD_*,
$$
 is an isometry and
\begin{equation}\label{fi}
\Phi^*(g\oplus 0)=\Theta_{\bf f,T}^*g, \qquad g\in (\otimes_{i=1}^k F^2(H_{n_i}))\otimes
\cD.
\end{equation}
This leads to
\begin{equation*}
\begin{split}
\|g\|^2&= \|P_{\HH_{\bf f,T}}(g\oplus 0)\|^2+\|\Phi \Phi^*(g\oplus 0)\|^2
=\|P_{\HH_{\bf f,T}}(g\oplus 0)\|^2+\|\Theta_{\bf f,T}^*g\|^2
\end{split}
\end{equation*}
for any $g\in (\otimes_{i=1}^k F^2(H_{n_i}))\otimes \cD$. Now, taking into account that
\begin{equation*}
 \|{\bf K_{f,T}^*}  g\|^2+ \|\Theta_{\bf f,T}^*g\|^2=\|g\|^2,
\quad g\in(\otimes_{i=1}^k F^2(H_{n_i}))\otimes \cD,
\end{equation*}
 we deduce
that
\begin{equation}\label{K*P}
\|{\bf K_{f,T}^*} g\|=\|P_{\HH_{\bf f,T}}(g\oplus 0)\|,  \quad g\in
(\otimes_{i=1}^k F^2(H_{n_i}))\otimes \cD.
\end{equation}
Since the $k$-tuple  ${\bf T}=(T_1,\ldots, T_k)$ is  completely non-coisometric,  the noncommutative Berezin kernel ${\bf K_{f,T}}$ is  a one-to-one
operator   and, consequently,  $\text{\rm
range}\, {\bf K_{f,T}^*}$ is dense in $\cH$. Now,  let
$x\in \HH_{\bf f,T}$ and assume that $\left<x, P_{\HH_{\bf f,T}}(g\oplus
0)\right>=0$ for any $g\in (\otimes_{i=1}^k F^2(H_{n_i}))\otimes \cD $. Using the definition of
$\HH_{\bf f,T}$ and the fact that $\cK_{\bf f,T}$ coincides with the
span of all vectors
$g\oplus 0$ for $ g\in (\otimes_{i=1}^k F^2(H_{n_i}))\otimes \cD$ and
$\Theta_{\bf f,T} \varphi\oplus \Delta_{\Theta_{\bf f,T}}
\varphi$ for  $ \varphi\in (\otimes_{i=1}^k F^2(H_{n_i}))\otimes \cD$,
we deduce that $x=0$. This shows that
$$
\HH_{\bf f,T}=\left\{P_{\HH_{\bf f,T}}(g\oplus 0):\ g\in (\otimes_{i=1}^k F^2(H_{n_i}))\otimes
\cD\right\}
$$
Using  relation \eqref{K*P},  we conclude that there is a
unique unitary operator $\Gamma$ satisfying relation \eqref{Ga}.
For each $i\in \{1,\ldots, k\}$ and $j\in\{1,\ldots, n_i\}$, let $\TT_{i,j}:\HH_{\bf f,T}\to \HH_{\bf f,T}$ be
  defined  by
$$\TT_{i,j}:=\Gamma
T_{i,j}\Gamma^*,\qquad  i\in \{1,\ldots,k \}, j\in \{1,\ldots,n_i\}.
$$
 In what follows,  we prove that
\begin{equation}\label{PN-intert}
 \left( P_{(\otimes_{i=1}^k F^2(H_{n_i}))\otimes \cD}|_{\HH_{\bf f,T}}\right) \TT_{i,j}^*x=
({\bf W}_{i,j}^*\otimes I_\cD)\left( P_{(\otimes_{i=1}^k F^2(H_{n_i}))\otimes
\cD}|_{\HH_{\bf f,T}}\right)x
\end{equation}
for any   $i\in \{1,\ldots, k\}$, $j\in\{1,\ldots, n_i\}$, and $x\in \HH_{\bf f,T}$.
Using  relations \eqref{Ga} and \eqref{fi},
and the fact that $\Phi$ is an isometry, we deduce that
\begin{equation*}
\begin{split}
P_{(\otimes_{i=1}^k F^2(H_{n_i}))\otimes \cD} \Gamma {\bf K_{f,T}^*} g&= P_{(\otimes_{i=1}^k F^2(H_{n_i}))\otimes
\cD} P_{\HH_{\bf f,T}}(g\oplus 0)
 =
g-P_{(\otimes_{i=1}^k F^2(H_{n_i}))\otimes \cD} \Phi \Phi^*(g\oplus 0)\\
&=g-\Theta_{\bf f,T} \Theta_{\bf f,T}^* g ={\bf K_{f,T}} {\bf K_{f,T}^*}g
\end{split}
\end{equation*}
for any $g\in (\otimes_{i=1}^k F^2(H_{n_i}))\otimes \cD$. Taking into account
that  the range of ${\bf K_{f,T}^*}$ is dense in $\cH$, we deduce that
\begin{equation}\label{PGK} P_{(\otimes_{i=1}^k F^2(H_{n_i}))\otimes \cD} \Gamma={\bf K_{f,T}}.
\end{equation}
Hence,  and using the fact that the  noncommutative Berezin kernel
${\bf K_{f,T}}$ is one-to-one, we  can see that
\begin{equation*}
 P_{(\otimes_{i=1}^k F^2(H_{n_i}))\otimes \cD} |_{\HH_{\bf f,T}}={\bf K_{f,T}}
\Gamma^*
\end{equation*}
is a one-to-one operator acting from $\HH_{\bf f,T}$ to $(\otimes_{i=1}^k F^2(H_{n_i}))\otimes
\cD$. Relation \eqref{PGK} and  Theorem \ref{Berezin-prop} imply
\begin{equation*}
\begin{split}
\left(P_{(\otimes_{i=1}^k F^2(H_{n_i}))\otimes \cD} |_{\HH_{\bf,T}}\right)
\TT_{i,j}^*\Gamma h&= \left(P_{(\otimes_{i=1}^k F^2(H_{n_i}))\otimes \cD}
|_{\HH_{\bf f,T}}\right) \Gamma T_{i,j}^* h ={\bf K_{f,T}} T_{i,j}^*h\\&=
\left( {\bf W}_{i,j}^*\otimes I_{\cD}\right) {\bf K_{f,T}}h= \left( {\bf W}_{i,j}^*\otimes I_{\cD}\right)
\left(P_{(\otimes_{i=1}^k F^2(H_{n_i}))\otimes \cD} |_{\HH_{\bf f,T}}\right)\Gamma h
\end{split}
\end{equation*}
for any $i\in \{1,\ldots, k\}$, $j\in\{1,\ldots, n_i\}$, and $h\in \cH$. Now, we can  deduce   relation \eqref{PN-intert}.
Note that, since the operator $P_{(\otimes_{i=1}^k F^2(H_{n_i}))\otimes
\cD}|_{\HH_{\bf f,T}}$ is one-to-one, the
relation \eqref{PN-intert} uniquely determines each operator
$\TT_{i,j}^*$ for all $i\in \{1,\ldots, k\}$ and $j\in \{1,\ldots, n_i\}$.
This completes the proof.
\end{proof}

In what follows, we show  that the  characteristic function
$\Theta_{\bf f,T}$  is a complete unitary invariant for the completely non-coisometric
part of the noncommutative domain $\cC_{\bf f}^{\bf m}$.

\begin{theorem}\label{u-inv}
Let   ${\bf T}:=(T_1,\ldots,
T_k)\in \cC_{\bf f}^{\bf m}(\cH)$ and ${\bf T}':=(T_1',\ldots, T_k')\in
\cC_{\bf f}^{\bf m}(\cH')$ be two completely non-coisometric $k$-tuples. Then ${\bf T}$ and
${\bf T}'$ are unitarily equivalent if and only if their
characteristic functions $\Theta_{\bf f,T}$ and $\Theta_{\bf f,T'}$
coincide.
\end{theorem}

\begin{proof}
Assume that  the $k$-tuples ${\bf T}$ and ${\bf T}'$ are unitarily
equivalent and let $U:\cH\to \cH'$ be a unitary operator such that
$T_{i,j}=U^*T_{i,j}'U$ for any $i\in\{1,\ldots,k\}$ and $j\in\{1,\ldots, n_i\}$.
It is easy to see  that $U{\bf \Delta_{f,T}^m}(I)={\bf \Delta_{f,T'}^m}(I) U$ and, consequently,
$U\cD=\cD'$, where
$$
\cD:=\overline{{\bf \Delta_{f,T}^m}(I)(\cH)},\quad
 \quad \cD':=\overline{{\bf \Delta_{f,T'}^m}(I)(\cH')}.
$$
 Using the definition of the noncommutative Berezin kernel
  associated with ${\bf D_f^m}$, one can easily check that $(I_{\otimes_{i=1}^k F^2(H_{n_i})}\otimes U){\bf K_{f,T}}={\bf K_{f,T'}} U$. This implies
\begin{equation}
\label{UKKU}
(I_{\otimes_{i=1}^k F^2(H_{n_i})}\otimes U)(I- {\bf K_{f,T}}{\bf K_{f,T}^*}) (I_{\otimes_{i=1}^k F^2(H_{n_i})}\otimes U)=I- {\bf K_{f,T'}}{\bf K_{f,T'}^*}
\end{equation}
 and $(I_{\otimes_{i=1}^k F^2(H_{n_i})}\otimes U)\cM_{\bf T}= \cM_{\bf T'}$,
 where
 $
\cM_{\bf T}:= \overline{{\rm range}\,(I -{\bf K_{f,T}}{\bf K_{f,T}^*}) }
$
and $\cM_{T'}$ is defined similarly. Recall that  ${\bf M_T}:=(M_1,\ldots, M_k)$ is
 the $k$-tuple with $M_i:=(M_{i,1},\ldots, M_{i,n_i})$ and
  $M_{i,j}\in B(\cM_{\bf T})$, and it  is  given by $M_{i,j}:=A_{i,j}^*$, where $A_{i,j}\in B(\cM_{\bf T})$ is uniquely defined by
$$
A_{i,j}\left[(I -{\bf K_{f,T}}{\bf K_{f,T}^*})^{1/2}x\right]:=(I -{\bf K_{f,T}}{\bf K_{f,T}^*})^{1/2}({\bf W}_{i,j}\otimes I)x
$$
for any $x\in (\otimes_{i=1}^k F^2(H_{n_i}))\otimes \overline{{\bf \Delta_{f,T}^m}(I)(\cH)}$. Similarly, we define the $k$-tuple ${\bf M_{T'}}$ and the operators $A_{i,j}'\in B(\cM_{\bf T'})$.
Note that
\begin{equation*}
\begin{split}
A_{i,j} (I -{\bf K_{f,T}}{\bf K_{f,T}^*})^{1/2}x
&=(I_{\otimes_{i=1}^k F^2(H_{n_i})}\otimes U^*) A_{i,j}'(I -{\bf K_{f,T'}}{\bf K_{f,T'}^*})^{1/2}(I_{\otimes_{i=1}^k F^2(H_{n_i})}\otimes U^*)x\\
&=(I_{\otimes_{i=1}^k F^2(H_{n_i})}\otimes U^*) A_{i,j}' (I_{\otimes_{i=1}^k F^2(H_{n_i})}\otimes U)(I -{\bf K_{f,T}}{\bf K_{f,T}^*})^{1/2}x
\end{split}
\end{equation*}
 for any $x\in (\otimes_{i=1}^k F^2(H_{n_i}))\otimes \overline{{\bf \Delta_{f,T}^m}(I)(\cH)}$.
Hence, we deduce that
 $$
 A_{i,j}=(I_{\otimes_{i=1}^k F^2(H_{n_i})}\otimes U^*) A_{i,j}' (I_{\otimes_{i=1}^k F^2(H_{n_i})}\otimes U).
$$
 Now, we can see that $(I_{\otimes_{i=1}^k F^2(H_{n_i})}\otimes U)\cD_*=\cD_*'$, where
 $\cD_*:=\overline{{\bf \Delta_{f,M_T}^m}(I)(\cM_T)}$ and $\cD_*'$ is defined similarly.
 We introduce the
unitary operators $\tau$ and $\tau'$ by setting
$$\tau:=U|_\cD:\cD\to \cD' \ \text{ and }\
\tau_*:=(I_{\otimes_{i=1}^k F^2(H_{n_i})}\otimes U)|_{\cD_*}:\cD_*\to {\cD_*'}.
$$
Using the definition of the  characteristic function, it
is easy to show that
$$
(I_{\otimes_{i=1}^k F^2(H_{n_i})}\otimes
\tau)\Theta_{\bf f,T}=\Theta_{\bf f,T'}(I_{\otimes_{i=1}^k F^2(H_{n_i})}\otimes \tau_*).
$$

To prove the converse, assume that the  characteristic functions  of
${\bf T}$ and ${\bf T'}$ coincide.  Then  there exist unitary operators
$\tau:\cD\to \cD'$ and $\tau_*:\cD_*\to
\cD_*'$ such that
\begin{equation*}
(I_{\otimes_{i=1}^k F^2(H_{n_i})}\otimes \tau)\Theta_{\bf f,T}
=\Theta_{\bf f,T'}(I_{\otimes_{i=1}^k F^2(H_{n_i})}\otimes
\tau_*).
\end{equation*}
It is clear  that this relation   implies
$$
\Delta_{\Theta_{\bf f,T}}=\left(I_{\otimes_{i=1}^k F^2(H_{n_i})}\otimes \tau_*\right)^*
\Delta_{\Theta_{\bf f,T'}}\left(I_{(\otimes_{i=1}^k F^2(H_{n_i})}\otimes \tau_*\right)
$$
and
$$
\left(I_{\otimes_{i=1}^k F^2(H_{n_i})}\otimes
\tau_*\right)\overline{\Delta_{\Theta_{\bf f,T}}((\otimes_{i=1}^k F^2(H_{n_i})\otimes
\cD_*)}= \overline{\Delta_{\Theta_{\bf f,T'}}(\otimes_{i=1}^k F^2(H_{n_i})\otimes
\cD_*')}.
$$
 Define now the unitary operator $U:\cK_{\bf f,T}\to \cK_{\bf f,T'}$
by setting
$$U:=(I_{\otimes_{i=1}^k F^2(H_{n_i})}\otimes \tau)\oplus (I_{\otimes_{i=1}^k F^2(H_{n_i})}\otimes \tau_*).
$$
Note that the operator $\Phi:(\otimes_{i=1}^k F^2(H_{n_i})\otimes
\cD_*\to \cK_{\bf f,T}$, defined  by $$ \Phi\varphi:=
\Theta_{\bf f,T}\varphi \oplus \Delta_{\Theta_{\bf f,T}}\varphi,\quad
\varphi\in (\otimes_{i=1}^k F^2(H_{n_i})\otimes \cD_*, $$
 and the corresponding $\Phi'$ satisfy the following relations:
\begin{equation}
\label{Uni1} U \Phi\left(I_{\otimes_{i=1}^k F^2(H_{n_i})}\otimes \tau_*\right)^*=\Phi'
\end{equation}
and
\begin{equation}
\label{Uni2} \left(I_{\otimes_{i=1}^k F^2(H_{n_i})}\otimes \tau\right)
 P_{\otimes_{i=1}^k F^2(H_{n_i})\otimes
\cD}^{\cK_{\bf f,T}} U^*=P_{\otimes_{i=1}^k F^2(H_{n_i})\otimes
\cD'}^{\cK_{\bf f,T'}},
\end{equation}
where $P_{(\otimes_{i=1}^k F^2(H_{n_i})\otimes \cD}^{\cK_{\bf f,T}}$ is the orthogonal
projection of $\cK_{\bf f,T}$ onto $(\otimes_{i=1}^k F^2(H_{n_i})\otimes \cD$. Note
also that relation \eqref{Uni1} implies
\begin{equation*}
\begin{split}
U\HH_{\bf f,T}&=U\cK_{\bf f,T}\ominus U\Phi((\otimes_{i=1}^k F^2(H_{n_i})\otimes \cD_*)\\
&=\cK_{\bf f,T'}\ominus \Phi'(I_{\otimes_{i=1}^k F^2(H_{n_i})}\otimes \tau_*)((\otimes_{i=1}^k F^2(H_{n_i})\otimes
 \cD_*)\\
&=\cK_{\bf f,T'}\ominus \Phi' ((\otimes_{i=1}^k F^2(H_{n_i})\otimes \cD_*').
\end{split}
\end{equation*}
This shows that  the operator $U|_{\HH_{\bf f,T}}:\HH_{\bf f,T}\to
\HH_{\bf f,T'}$ is unitary.
Note also that
\begin{equation}
\label{intertw} ({\bf W}_{i,j}^*\otimes I_{\cD'})(I_{(\otimes_{i=1}^k F^2(H_{n_i})}\otimes
\tau)= (I_{(\otimes_{i=1}^k F^2(H_{n_i})}\otimes \tau)({\bf W}_{i,j}^*\otimes I_{\cD}).
\end{equation}
Let $\TT:=(\TT_1,\ldots \TT_n)$ and $\TT':=(\TT_1',\ldots \TT_n')$
be the model operators provided by Theorem \ref{model}  for ${\bf T}$ and
${\bf T}'$, respectively. Using the relation \eqref{PN-intert}  for ${\bf T}'$
and ${\bf T}$, as well as \eqref{Uni2} and \eqref{intertw}, we have
\begin{equation*}
\begin{split}
P_{(\otimes_{i=1}^k F^2(H_{n_i})\otimes \cD'}^{\KK_{\bf f,T'}}{\TT_{i,j}'}^*Ux
&=  ({\bf W}_{i,j}^*\otimes I_{\cD'}) P_{(\otimes_{i=1}^k F^2(H_{n_i})\otimes
\cD}^{\cK_{\bf f,T}}Ux\\
&=({\bf W}_{i,j}^*\otimes I_{\cD'})(I_{\otimes_{i=1}^k F^2(H_{n_i})}\otimes \tau)
P_{(\otimes_{i=1}^k F^2(H_{n_i})\otimes \cD}^{\cK_{\bf f,T}}x\\
&=(I_{\otimes_{i=1}^k F^2(H_{n_i})}\otimes \tau)({\bf W}_{i,j}^*\otimes I_{\cD})
P_{(\otimes_{i=1}^k F^2(H_{n_i})\otimes \cD}^{\cK_{\bf f,T}}x\\
&=(I_{\otimes_{i=1}^k F^2(H_{n_i})}\otimes \tau) P_{(\otimes_{i=1}^k F^2(H_{n_i})\otimes \cD}^{\cK_{\bf f,T}}
\TT_i^*x\\
&= P_{(\otimes_{i=1}^k F^2(H_{n_i})\otimes \cD'}^{\cK_{\bf f,T'}}U \TT_{i,j}^*x
\end{split}
\end{equation*}
for any $i=\{1,\ldots, k\}$, $j\in\{1,\ldots, n_i\}$, and $x\in \HH_{\bf f,T}$. Since
$P_{(\otimes_{i=1}^k F^2(H_{n_i})\otimes \cD'}^{\cK_{\bf f,T'}}|_{\HH_{\bf f,T'}}$ is an
one-to-one operator (see Theorem \ref{model}), we obtain
$
\left(U|_{\HH_{\bf f,T}}\right)
\TT_{i,j}^*=({\TT_{i,j}'})^*\left(U|_{\HH_{\bf f,T}}\right).
$
 Due to  Theorem \ref{model}, we conclude that the $k$-tuples
  ${\bf T}$ and ${\bf T}'$ are
 unitarily equivalent. The proof is complete.
\end{proof}

\begin{proposition} If  ${\bf T}=({ T}_1,\ldots, { T}_k)\in {\bf D_f^m}(\cH)$, then the following statements hold.
\begin{enumerate}
\item[(i)]
${\bf T}$  is unitarily equivalent to
$({\bf W}_1\otimes I_\cK,\ldots, {\bf W}_k\otimes I_\cK)$
for some Hilbert space $\cK$ if and only if
${\bf T}\in \cC_{\bf f}^{\bf m}(\cH)$ is completely non-coisometric  and
the characteristic function $\Theta_{\bf f,T}=0$.
\item[(ii)] If ${\bf T}\in \cC_{\bf f}^{\bf m}(\cH)$, then $\Theta_{\bf f,T}$  has dense range if and only if there is no nonzero vector $h\in \cH$ such that
    $$
    \lim_{q=(q_1,\ldots, q_k)\in \NN^k} \left< (id-\Phi_{f_1,T_1}^{q_1})\cdots (id-\Phi_{f_k,T_k}^{q_k})(I_\cH)h, h\right>=\|h\|.
    $$
\end{enumerate}
\end{proposition}
\begin{proof}
Note that if ${\bf T} =({\bf W}_1\otimes I_\cK,\ldots, {\bf W}_k\otimes I_\cK)$ for some Hilbert space $\cK$, then ${\bf K_{f,T}}=I$. Since
$
{\bf K_{f,T}}{\bf K_{f,T}^*}+ \Theta_{\bf f,T}\Theta_{\bf f, T}^*=I,
$
we deduce that $\Theta_{\bf f,T}=0$. Conversely,
 if ${\bf T}\in \cC_{\bf f}^{\bf m}(\cH)$ is completely non-coisometric  and
the characteristic function $\Theta_{\bf f,T}=0$, then ${\bf K_{f,T}}{\bf K_{f,T}^*}=I$. Using Theorem \ref{model}, the result follows.

Due to Theorem \ref{Berezin-prop}, the condition in item (ii) is equivalent to
$\ker (I-{\bf K_{f,T}^*}{\bf K_{f,T}})=\{0\}$, which is equivalent to
 $\ker (I-{\bf K_{f,T}}{\bf K_{f,T}^*})=\{0\}$ and, therefore, to
  $\ker \Theta_{\bf f,T}\Theta_{\bf f, T}^*=\{0\}$. Hence, the result follows. The
  proof is complete.
\end{proof}

\section{Dilation theory on noncommutative polydomains}

We develop a dilation theory on the noncommutative polydomain  ${\bf D_q^m}(\cH)$ and  obtain  Wold type
decompositions for non-degenerate $*$-representations of the
$C^*$-algebra $C^*({\bf W}_{i,j})$.

We recall that $\cP({\bf W})$  is the set of all polynomials $p({\bf W}_{i,j})$  in  the operators ${\bf W}_{i,j}$, $i\in \{1,\ldots, k\}$, $j\in \{1,\ldots, n_i\}$,  and the identity.

\begin{lemma}\label{compact}  Let ${\bf q}=(q_1,\ldots, q_k)$ be a $k$-tuple of   positive regular  noncommutative polynomials
 and let ${\bf W}=({\bf W}_1,\ldots, {\bf W}_k)$ be the universal model
associated with the noncommutative polydomain ${\bf D_q^m}$. Then all
the compact operators in $B(\otimes_{i=1}^k F^2(H_{n_i}))$ are contained in the operator
space
$$\cS:=\overline{\text{\rm  span}} \{ p({\bf W}_{i,j})q({\bf W}_{i,j})^*:\
p({\bf W}_{i,j}),q({\bf W}_{i,j}) \in  \cP({\bf W})\},
$$
where the closure is in the operator norm.
\end{lemma}

\begin{proof} According to Lemma \ref{univ-model}, we have
\begin{equation}\label{proj}
  (I-\Phi_{q_1,{\bf W}_1})^{m_1}\cdots (I-\Phi_{q_k,{\bf W}_k})^{m_k}(I)={\bf P}_\CC,
\end{equation}
where ${\bf P}_\CC$ is the
 orthogonal projection from $\otimes_{i=1}^k F^2(H_{n_i})$ onto $\CC 1\subset \otimes_{i=1}^k F^2(H_{n_i})$.
Fix
$$g({\bf W}_{i,j}):= \sum\limits_{{\beta_1\in \FF_{n_1}^+,\ldots, \beta_k\in \FF_{n_k}^+}\atop{|\beta_1|+\cdots +|\beta_k|\leq n}} d_{\beta_1,\ldots, \beta_k}
{\bf W}_{1,\beta_1}\cdots {\bf W}_{k,\beta_k} \quad
\text{ and } \quad \xi:=\sum\limits_{{\beta_1\in \FF_{n_1}^+,\ldots, \beta_k\in \FF_{n_k}^+}} c_{\beta_1,\ldots, \beta_k}e^1_{\beta_1}\otimes\cdots \otimes  e^k_{\beta_k}
$$
and note  that
$
{\bf P}_\CC g({\bf W}_{i,j})^*\xi= \left< \xi,g({\bf W}_{i,j})(1)\right>.
$
Consequently, we have
\begin{equation}\label{rankone}
\chi({\bf W}_{i,j}){\bf P}_\CC g({\bf W}_{i,j})^*\xi= \left<
\xi,g({\bf W}_{i,j})(1)\right>\chi({\bf W}_{i,j})(1)
\end{equation}
for any polynomial $\chi({\bf W}_{i,j})$.  Using relation \eqref{proj}, we deduce that the
operator $\chi({\bf W}_{i,j}){\bf P}_\CC g({\bf W}_{i,j})^*$
 has rank one  and  it is   in the
operator space $\cS$. On the other hand, due to the fact that
the set of all vectors  of the form
 $  \sum\limits_{{\beta_1\in \FF_{n_1}^+,\ldots, \beta_k\in \FF_{n_k}^+}\atop{|\beta_1|+\cdots +|\beta_k|\leq n}} d_{\beta_1,\ldots, \beta_k}
{\bf W}_{1,\beta_1}\cdots {\bf W}_{k,\beta_k}(1)$ with  $n\in \NN$, $d_{\beta_1,\ldots, \beta_k}\in \CC$, is  dense in
$\otimes_{i=1}^k F^2(H_{n_i})$, relation \eqref{rankone} implies that   all
the compact operators in $B(\otimes_{i=1}^k F^2(H_{n_i}))$ are contained in  $\cS$.
This completes the proof.
\end{proof}

Let $C^*(\Gamma)$ be the $C^*$-algebra generated by  a set of operators $\Gamma\subset B(\cK)$ and the identity.
A subspace $\cH\subset \cK$ is called $*$-cyclic for $\Gamma$ if
$\cK=\overline{\text{\rm span}}\{Xh, X\in C^*(\Gamma), h\in \cH\}$.
The main result of this section is the following dilation theorem
 for the elements of
the  noncommutative polydomain ${\bf D_q^m}(\cH)$.

\begin{theorem}\label{dil2}  Let ${\bf q}=(q_1,\ldots, q_k)$ be a $k$-tuple of   positive regular  noncommutative polynomials
 and let ${\bf W}=({\bf W}_1,\ldots, {\bf W}_k)$ be the universal model
associated with the abstract noncommutative polydomain ${\bf D_q^m}$.
If ${\bf T}=({ T}_1,\ldots, { T}_k)$ is  a  $k$-tuple  in
${\bf D}_{\bf q}^{\bf m}(\cH)$,
then there exists  a $*$-representation $\pi:C^*({\bf W}_{i,j})\to
B(\cK_\pi)$  on a separable Hilbert space $\cK_\pi$,  which
annihilates the compact operators and
$$
(I-\Phi_{q_1,\pi({\bf W}_1)}) \cdots (I-\Phi_{q_k,\pi({\bf W}_k)}) (I_{\cK_\pi})=0,
$$
where $\pi({\bf W}_{i}):=(\pi({\bf W}_{i,1}),\ldots, \pi({\bf W}_{i,n_i}))$,
such that
$\cH$ can be identified with a $*$-cyclic co-invariant subspace of
$$\tilde\cK:=\left[(\otimes _{i=1}^k F^2(H_{n_i}))\otimes
 \overline{{\bf \Delta_{q,T}^m}(I)(\cH)}\right]\oplus
\cK_\pi$$ under  each operator
$$
V_{i,j}:=\left[\begin{matrix} {\bf W}_{i,j}\otimes
I_{\overline{{\bf \Delta_{q,T}^m}(I)(\cH)}}&0\\0&\pi({\bf W}_{i,j})
\end{matrix}\right],\qquad i\in \{1,\ldots,k\}, j\in\{1,\ldots, n_i\},
$$
where ${\bf \Delta_{q,T}^m}(I)
:= (id-\Phi_{q_1,T_1})^{m_1}\cdots (id-\Phi_{q_k,T_k})^{m_k}(I)$,
and such that
$ T_{i,j}^*=V_{i,j}^*|_{\cH}$
  for all   $ i\in \{1,\ldots,k\}$ and  $j\in\{1,\ldots, n_i\}$.
\end{theorem}
 \begin{proof}
 Applying Arveson extension theorem
\cite{Arv-acta} to the map $\Psi_{\bf q,T}$ of Theorem \ref{Poisson-C*}, we find a unital
completely positive linear map $\Psi_{\bf q,T}:C^*({\bf W}_{i,j})\to B(\cH)$ such that $\Psi_{\bf q,T}({\bf W}_{(\alpha)} {\bf W}_{(\beta)})^*={\bf T}_{(\alpha)} {\bf T}_{(\beta)}^*$,
where
 ${\bf T}_{(\alpha)}:=T_{1,\alpha_1}\cdots T_{k,\alpha_k}$ for
$(\alpha):=(\alpha_1,\ldots, \alpha_k)\in \FF_{n_1}^+\times \cdots \times \FF_{n_k}^+$, and ${\bf W}_{(\alpha)}$ is defined similarly.
 Let $\tilde\pi:C^*({\bf W}_{i,j})\to
B(\tilde\cK)$ be the minimal Stinespring dilation \cite{St}  of
$\Psi_{\bf q,T}$. Then we have
$$\Psi_{\bf q,T}(X)=P_{\cH} \tilde\pi(X)|\cH,\quad X\in C^*({\bf W}_{i,j}),
$$
and $\tilde\cK=\overline{\text{\rm span}}\{\tilde\pi(X)h:\ X\in C^*({\bf W}_{i,j}), h\in
\cH\}.$  Now,  we prove  that
  that $P_\cH \tilde\pi({\bf W}_{(\alpha)})|_{\cH^\perp}=0$ for any
$(\alpha):=(\alpha_1,\ldots, \alpha_k)\in \FF_{n_1}^+\times \cdots \times \FF_{n_k}^+$. Indeed, we have
\begin{equation*}
\begin{split}
\Psi_{\bf q,T}({\bf W}_{(\alpha)}{\bf W}_{(\alpha)}^*)&={\bf T}_{(\alpha)}{\bf T}_{(\alpha)}^*=P_\cH\tilde\pi({\bf W}_{(\alpha)})   \tilde\pi({\bf W}_{(\alpha)}^*)|_\cH\\
&=P_\cH\tilde\pi({\bf W}_{(\alpha)})(P_\cH+P_{\cH^\perp})   \tilde\pi({\bf W}_{(\alpha)}^*)|_\cH\\
&=\Psi_{\bf q,T}({\bf W}_{(\alpha)}{\bf W}_{(\alpha)}^*)+ (P_\cH\tilde\pi({\bf W}_{(\alpha)})|_{\cH^\perp})(P_\cH\tilde\pi({\bf W}_{(\alpha)})|_{\cH^\perp})^*.
\end{split}
\end{equation*}
Consequently, we deduce that $P_\cH \tilde\pi({\bf W}_{(\alpha)})|_{\cH^\perp}=0$ and, therefore, $\cH$
 is an
invariant subspace under each  operator $\tilde\pi({\bf W}_{i,j})^*$
and
\begin{equation}\label{coiso}
\tilde\pi({\bf W}_{i,j})^*|_\cH=\Psi_{\bf q,T}({\bf W}_{i,j}^*)=T_{i,j}^*
\end{equation}
for any $i\in \{1,\ldots,k\}$ and $ j\in\{1,\ldots, n_i\}$.

According to   Theorem \ref{compact},   all the
compact operators $ \cC(\otimes _{i=1}^k F^2(H_{n_i}))$ in
$B(\otimes _{i=1}^k F^2(H_{n_i}))$ are contained in the
$C^*$-algebra $C^*({\bf W}_{i,j})$.
 Due to  standard theory of
representations of  $C^*$-algebras \cite{Arv-book}, the
representation $\tilde\pi$ decomposes into a direct sum
$\tilde\pi=\pi_0\oplus \pi$ on $\tilde \cK=\cK_0\oplus \cK_\pi$,
where $\pi_0$, $\pi$  are disjoint representations of
$C^*({\bf W}_{i,j})$ on the Hilbert spaces
$$\cK_0:=\overline{\text{\rm span}}\{\tilde\pi(X)\tilde\cK:\ X\in \cC(\otimes _{i=1}^k F^2(H_{n_i}))\}
\quad \text{ and  }\quad  \cK_\pi:=\cK_0^\perp,
$$
respectively,
such that
 $\pi$ annihilates  the compact operators in $B(\otimes _{i=1}^k F^2(H_{n_i}))$, and
  $\pi_0$ is uniquely determined by the action of $\tilde\pi$ on the
  ideal $\cC(\otimes _{i=1}^k F^2(H_{n_i}))$ of compact operators.
Since every representation of $\cC(\otimes _{i=1}^k F^2(H_{n_i}))$ is equivalent to a
multiple of the identity representation, we deduce
 that
\begin{equation}\label{sime}
\cK_0\simeq(\otimes _{i=1}^k F^2(H_{n_i}))\otimes \cG, \quad  \pi_0(X)=X\otimes I_\cG, \quad
X\in C^*({\bf W}_{i,j}),
\end{equation}
 for some Hilbert space $\cG$.
 Using Theorem \ref{compact} and its proof, one can
easily see
that
\begin{equation*}\begin{split}
\cK_0&:=\overline{\text{\rm span}}\{\tilde\pi(X)\tilde\cK:\ X\in \cC(\otimes _{i=1}^k F^2(H_{n_i}))\}\\
&=\overline{\text{\rm span}}\{\tilde\pi({\bf W}_{(\alpha)} {\bf P}_\CC  {\bf W}_{(\beta)}^*)\tilde\cK:\
 (\alpha), (\beta)\in \FF_{n_1}^+\times \cdots \times \FF_{n_k}^+\}\\
&= \overline{\text{\rm span}}\left\{\tilde\pi({\bf W}_{(\alpha)}) \left[(I-\Phi_{q_1,\tilde\pi({\bf W}_1)})^{m_1}\cdots (I-\Phi_{q_k,\tilde\pi({\bf W}_k)})^{m_k}
(I_{\tilde\cK})\right] \tilde\cK:\ (\alpha)\in \FF_{n_1}^+\times \cdots \times \FF_{n_k}^+\right\}.
\end{split}
\end{equation*}
Since $(I-\Phi_{q_1, {\bf W}_1) })^{m_1}\cdots (I-\Phi_{q_k, {\bf W}_k)}^{m_k}
(I )={\bf P}_\CC$,    is a
 projection  of rank one in $C^*({\bf W}_{i,j})$,
  we deduce  that
$
 (I-\Phi_{q_1,\pi({\bf W}_1)})^{m_1}\cdots (I-\Phi_{q_k,\pi({\bf W}_k)})^{m_k}
(I_{\cK_\pi})= 0$
  and
$
\dim \cG=\dim \left[\text{\rm range}\,\tilde\pi({\bf P}_\CC )\right].
$
On the other hand, since  the  Stinespring representation $\tilde\pi$  is minimal, we
can use the proof of Theorem \ref{compact} to  deduce that
\begin{equation*}
\text{\rm range}\,\tilde\pi({\bf P}_\CC )=
 \overline{\text{\rm
span}}\{\tilde\pi({\bf P}_\CC )\tilde\pi({\bf W}_{(\beta)}^*)h:\ (\beta)\in \FF_{n_1}^+\times \cdots \times \FF_{n_k}^+, h\in \cH\}.
\end{equation*}
Indeed, we have
\begin{equation*}\begin{split}
\text{\rm range}\,\tilde\pi({\bf P}_\CC)&=
\overline{\text{\rm span}}\{\tilde\pi({\bf P}_\CC)\tilde\pi(X)h:\ X\in C^*({\bf W}_{i,j}), h\in \cH\}\\
&=
\overline{\text{\rm span}}\{\tilde\pi({\bf P}_\CC)\tilde\pi(Y)h:\ Y\in \cC(\otimes _{i=1}^k F^2(H_{n_i})), h\in \cH\}\\
&=
\overline{\text{\rm span}}\{\tilde\pi({\bf P}_\CC)\tilde\pi({\bf W}_{(\alpha)} {\bf P}_\CC  {\bf W}_{(\beta)}^*)h:\ (\alpha), (\beta)\in \FF_{n_1}^+\times \cdots \times \FF_{n_k}^+, h\in \cH\}\\
&=
\overline{\text{\rm span}}\{\tilde\pi({\bf P}_\CC)\tilde\pi({\bf W}_{(\beta)}^*)h:\ (\beta)\in \FF_{n_1}^+\times \cdots \times \FF_{n_k}^+, h\in \cH\}.
\end{split}
\end{equation*}
 Now,  using the fact that
 \begin{equation*}
\begin{split}
\Psi_{\bf q,T}({\bf W}_{(\alpha)} X)&=P_\cH(\tilde\pi({\bf W}_{(\alpha)}) \tilde\pi(X))|\cH\\
&=(P_\cH\tilde\pi({\bf W}_{(\alpha)})|\cH)(P_\cH\tilde\pi(X)
|\cH)\\
&=\Psi_{\bf q,T}({\bf W}_{(\alpha)})
\Psi_{\bf q,T}(X)\end{split}
\end{equation*}
for any $X\in C^*({\bf W}_{i,j})$ and $(\alpha) \in \FF_{n_1}^+\times \cdots \times \FF_{n_k}^+$, it is easy to see that
\begin{equation*}\begin{split}
\left<\tilde\pi({\bf P}_\CC)\tilde\pi({\bf W}_{(\alpha)}^*)h,
\tilde\pi({\bf P}_\CC)\tilde\pi({\bf W}_{(\beta)}^*)k\right>
& = \left<h,{\bf T}_{(\alpha)}\left[(id-\Phi_{q_1,T_1})^{m_1}\cdots (id-\Phi_{q_k,T_k})^{m_k}(I_\cH) \right]{\bf T}_{(\beta)}^*h\right>\\
& =
\left<{\bf \Delta_{q,T}^m}(I){\bf T}_{(\alpha)}^*h,{\bf \Delta_{q,T}^m}(I){\bf T}_{(\beta)}^*k\right>
\end{split}
\end{equation*}
for any $h, k \in \cH$ and $(\alpha), (\beta) \in \FF_{n_1}^+\times \cdots \times \FF_{n_k}^+$. This implies the existence of      a unitary
operator $\Lambda:\text{\rm range}\,\tilde\pi({\bf P}_\CC)\to
\overline{{\bf \Delta_{q,T}^m}(I)\cH}$ defined by
$$
\Lambda[\tilde\pi({\bf P}_\CC)\tilde\pi({\bf W}_{(\alpha)}^*)h]:={\bf \Delta_{q,T}^m}(I)
{\bf T}_{(\alpha)}^*h,\quad h\in \cH, \,\alpha\in \FF_n^+.
$$
 This shows that
$$
\dim[\text{\rm range}\,\pi({\bf P}_\CC)]= \dim
\overline{{\bf \Delta_{q,T}^m}(I)\cH}=\dim \cG.
$$
 Using relations \eqref{coiso} and
\eqref{sime},  and identifying    $\cG$ with
$ \overline{{\bf \Delta_{q,T}^m}(I)\cH}$, we obtain the required dilation.
On the other hand, due to the fact that $
 (I-\Phi_{q_1,\pi({\bf W}_1)})^{m_1}\cdots (I-\Phi_{q_k,\pi({\bf W}_k)})^{m_k}
(I_{\cK_\pi})= 0$,
  we can use  Proposition \ref{exemp}  to deduce that
     $
 (I-\Phi_{q_1,\pi({\bf W}_1)}) \cdots (I-\Phi_{q_k,\pi({\bf W}_k)})
(I_{\cK_\pi})= 0$.
 The proof is complete.
\end{proof}

We remark that if we replace   ${\bf q}=(q_1,\ldots, q_k)$, in Theorem \ref{dil2},  by    a $k$-tuple ${\bf f}:=(f_1,\ldots, f)$ of  positive regular free holomorphic functions  we obtain   a dilation theorem for any ${\bf T}=({ T}_1,\ldots, { T}_k)$    in
${\bf D}_{\bf f}^{\bf m}(\cH)$. More precisely, one can show  that there is a
  $*$-representation $\tilde\pi:C^*({\bf W}_{i,j})\to
B(\tilde\cK)$  such that  $\cH$ is an invariant subspace  under each operator
$\tilde\pi({\bf W}_{i,j})^*$  and $T_{i,j}^*=\tilde\pi({\bf W}_{i,j})^*|_{\cH}$ for any $i\in \{1,\ldots,k\}, j\in\{1,\ldots, n_i\}$.

 On the other hand, note that, using  the proof of Theorem \ref{dil2}
  and due to  the  standard theory of
representations of   $C^*$-algebras,   one can deduce
 the following   Wold type
decomposition for non-degenerate $*$-representations of the
$C^*$-algebra $C^*({\bf W}_{i,j})$.

\begin{corollary}\label{wold} Let ${\bf q}=(q_1,\ldots, q_k)$ be a $k$-tuple of   positive regular  noncommutative polynomials
 and let ${\bf W}=({\bf W}_{i,j})$ be the universal model
associated with the abstract noncommutative polydomain ${\bf D_q^m}$. If
\ $\pi:C^*({\bf W}_{i,j})\to B(\cK)$ is  a nondegenerate
$*$-representation  of \ $C^*({\bf W}_{i,j})$ on a separable
Hilbert space  $\cK$, then $\pi$ decomposes into a direct sum
$$
\pi=\pi_0\oplus \pi_1 \  \text{ on  } \ \cK=\cK_0\oplus \cK_1,
$$
where $\pi_0$ and  $\pi_1$  are disjoint representations of \
$C^*({\bf W}_{i,j})$ on the Hilbert spaces
\begin{equation*}
\cK_0:=\overline{\text{\rm span}}\left\{\pi({\bf W}_{(\alpha)}) \left[(I-\Phi_{q_1,\pi({\bf W}_1)})^{m_1}\cdots (I-\Phi_{q_k,\pi({\bf W}_k)})^{m_k}
(I_{\cK})\right] \cK:\ (\alpha)\in \FF_{n_1}^+\times \cdots \times \FF_{n_k} \right\}
\end{equation*}
and $\cK_1:=\cK_0^\perp,$
 respectively, where $\pi({\bf W}_{i}):=(\pi({\bf W}_{i,1}),\ldots, \pi({\bf W}_{i,n_i}))$. Moreover,
  up to an isomorphism,
\begin{equation*}
\cK_0\simeq (\otimes _{i=1}^k F^2(H_{n_i}))\otimes \cG, \quad  \pi_0(X)=X\otimes I_\cG \quad
\text{ for  any } \  X\in C^*({\bf W}_{i,j}),
\end{equation*}
 where $\cG$ is  a Hilbert space with
$$
\dim \cG=\dim \left\{\text{\rm range}\, \left[(I-\Phi_{q_1,\pi({\bf W}_1)})^{m_1}\cdots (I-\Phi_{q_k,\pi({\bf W}_k)})^{m_k}
(I_{\cK})\right]\right\},
$$
 and $\pi_1$ is a $*$-representation  which annihilates the compact operators   and
$$
 (I-\Phi_{q_1,\pi_1({\bf W}_1)}) \cdots (I-\Phi_{q_k,\pi_1({\bf W}_k)})
(I_{\cK_1})= 0.
$$
If  $\pi'$ is another nondegenerate  $*$-representation of
$C^*({\bf W}_{i,j})$ on a separable  Hilbert space  $\cK'$, then
$\pi$ is unitarily equivalent to $\pi'$ if and only if
$\dim\cG=\dim\cG'$ and $\pi_1$ is unitarily equivalent to $\pi_1'$.
\end{corollary}
 Note that  in the particular case when ${\bf m}=(1,\ldots,1)$, $q_i:=Z_{i,1}+\cdots + Z_{i,n_i}$
  for $i\in \{1,\ldots, k\}$, and $V_i=(V_{i,1},\ldots, V_{i,n_i})$ are row isometries  such that ${\bf V}=(V_{i,j})$ are doubly commuting, Corollary \ref{wold} provides a Wold type decomposition for ${\bf V}$.
We also remark that under
 the hypotheses and notations of Corollary $\ref{wold}$, and setting
$V_{i,j}:=\pi({\bf W}_{i,j})$ for any $i\in \{1,\ldots,k\}$ and $ j\in\{1,\ldots, n_i\}$,   the following statements are
equivalent:
\begin{enumerate}
\item[(i)]
${\bf V}:=(V_1,\ldots, V_k)$ is a pure element in ${\bf D_q^m}(\cK)$ ;
\item[(ii)]  for each $i\in\{1,\ldots,k\}$,
 $\lim\limits_{p\to\infty} \Phi^p_{q_i,{V}_i}(I)=0$ in the strong operator
 topology;
 \item[(iii)]
$
\cK:=\overline{\text{\rm span}}\left\{ {V}_{(\alpha)} \left[(I-\Phi_{q_1,{V}_1})^{m_1}\cdots (I-\Phi_{q_k,{V}_k})^{m_k}
(I_{\cK})\right] (\cK):\ (\alpha)\in \FF_{n_1}^+\times \cdots \times \FF_{n_k} \right\}.
$
\end{enumerate}

 We mention that,  under the
additional  condition that
$
\overline{\text{\rm span}}\,\{{\bf W}_{(\alpha)}  {\bf W}_{(\beta)}^* :\
 (\alpha), (\beta)\in \FF_{n_1}^+\times \cdots \times \FF_{n_k}^+\}$ is equel to $C^*({\bf W}_{i,j}),
$
(eg. for the polyball)
  the map $\Psi_{\bf q,T}$ in the
proof of  Theorem \ref{dil2} is unique and the dilation of ${\bf T}$ is minimal, i.e.,
$\tilde \cK$ is the closed span of all $ {\bf V}_{(\alpha)} \cH$, $(\alpha)\in \FF_{n_1}^+\times \cdots \times \FF_{n_k}^+$.
 Taking into account the uniqueness of the
minimal Stinespring representation    and the
Wold type decomposition mentioned above,  one can prove  the
uniqueness, up to  unitary equivalence, of the minimal dilation provided by Theorem \ref{dil2}.
  Moreover, let ${\bf T}'=({ T}_1',\ldots, { T}_k')$ be another    $k$-tuple  in
${\bf D}_{\bf q}^{\bf m}(\cH')$   and let ${\bf V}'=({ V}_1',\ldots, { V}_k')$   be the corresponding dilation.    Using  standard  arguments concerning the
representation theory  of $C^*$-algebras,  one can prove that
 ${\bf T}$ and ${\bf T}'$ are unitarily equivalent if and only if
$\dim \overline{{\bf\Delta_{q,T}^m}(I)(\cH)}
=\dim \overline{{\bf\Delta_{q,T'}^m}(I)(\cH')}
$
and there are unitary operators
 $U:(\otimes _{i=1}^k F^2(H_{n_i}))
 \otimes\overline{{\bf\Delta_{q,T}^m}(I)(\cH)}\to
  (\otimes _{i=1}^k F^2(H_{n_i}))\otimes\overline{{\bf\Delta_{q,T'}^m}(I)(\cH')}$ and
$\Gamma:\cK_\pi\to \cK_{\pi '}$ such that
$$
 U({\bf W}_{i,j}\otimes I_{\overline{{\bf\Delta_{q,T}^m}(I)(\cH)}})
 =({\bf W}_{i,j}\otimes I_{\overline{{\bf\Delta_{q,T'}^m}(I)(\cH')}})U, \qquad  \Gamma\pi({\bf W}_{i,j})=\pi'({\bf W}_{i,j}) \Gamma
 $$
 for any $i\in \{1,\ldots,k\}$ and $ j\in\{1,\ldots, n_i\}$,
and
$\left[ \begin{matrix} U&0\\
0&\Gamma\end{matrix} \right] \cH=\cH'.
$

\begin{corollary} \label{part-cases}
Let ${\bf V}:=(V_1,\ldots, V_k)\in {\bf D_q^m}(\tilde\cK)$ be the dilation
 of ${\bf T}:=(T_1,\ldots,T_k)\in {\bf D_q^m}(\cH)$, given by Theorem $\ref{dil2}$.
 Then,
\begin{enumerate}
\item[(i)]
 ${\bf V}$ is a pure element in ${\bf D_q^m}(\tilde\cK)$ if and only if
${\bf T}$ is  a pure element in ${\bf D_q^m}(\cH)$;
\item[(ii)]
   $
 (I-\Phi_{q_1,{V}_1}) \cdots (I-\Phi_{q_k, {V}_k})
(I_{\tilde\cK})= 0
$ \
if and only if \ $
 (I-\Phi_{q_1, {T}_1}) \cdots (I-\Phi_{q_k,{T}_k})
(I_{\cH})= 0
$.
\end{enumerate}
\end{corollary}
\begin{proof} According to Theorem \ref{dil2}, we have
$$
(id-\Phi_{q_k,T_k}^{p_k})\cdots (id-\Phi_{q_1,T_1}^{p_1})(I_\cH)=P_\cH \left[\begin{matrix}
(id-\Phi_{q_k,{\bf W}_k}^{p_k})\cdots (id-\Phi_{q_1,{\bf W}_1}^{p_1})(I_{\otimes_{i=1}^k F^2(H_{n_i})})\otimes
I_{\overline{{\bf\Delta_{q,T}^m}(\cH)}}&0\\0&
0\end{matrix}\right]|\cH.
$$
Hence,  we deduce that
$\lim_{{\bf q}=(q_1,\ldots, q_k)\in \ZZ_+^k}(id-\Phi_{q_k,T_k}^{p_k})\cdots (id-\Phi_{q_1,T_1}^{p_1})(I_\cH)=I$
if and only if
$P_\cH
\left[\begin{matrix}
 I&0\\0& 0\end{matrix}\right]|\cH=I$.
Consequently,   ${\bf T}$ is pure   if and only if  $\cH\perp (0\oplus
\cK_\pi)$.  According to Theorem \ref{dil2}, this is equivalent to
  $\cH\subset (\otimes _{i=1}^k F^2(H_{n_i}))\otimes
   \overline{{\bf \Delta_{q,T}^m}(I)(\cH)}$. On
the other hand, since $(\otimes _{i=1}^k F^2(H_{n_i}))\otimes
\overline{{\bf \Delta_{q,T}^m}(I)(\cH)}$
is reducing for each $V_{i,j}$, and $\widetilde \cK$ is the
smallest reducing subspace for   $V_{i,j}$, which contains
$\cH$, we must have
$\widetilde\cK=(\otimes _{i=1}^k F^2(H_{n_i}))\otimes
\overline{{\bf \Delta_{q,T}^m}(I)(\cH)}$.
Therefore, item (i) holds.

   To prove part (ii), note that
$$
{\bf\Delta_{q,V}^m}(I_{\tilde\cK})= \left[\begin{matrix}
{\bf\Delta_{q,W}^m}(I_{\otimes_{i=1}^k F^2(H_{n_i})}) \otimes
I_{\overline{{\bf\Delta_{q,T}^m}(\cH)}}&0\\0&
0\end{matrix}\right].
$$
Hence, we  deduce that
 ${\bf\Delta_{q,V}^m}(I_{\tilde\cK})=0$ if and only if
 $
 {\bf\Delta_{q,W}^m}(I_{\otimes_{i=1}^k F^2(H_{n_i})}) \otimes
I_{\overline{{\bf\Delta_{q,T}^m}(\cH)}}=0.
$
 On the other hand,  we know
that   ${\bf\Delta_{q,W}^m}(I_{\otimes_{i=1}^k F^2(H_{n_i})})={\bf P}_\CC$.
Consequently,  the relation above holds
if and only if ${\bf\Delta_{q,T}^m}=0$.
 Now, using  Proposition \ref{exemp},  we obtain the
equivalence in part (ii). The proof is complete.
\end{proof}

We remark that
 every pure  $k$-tuple ${\bf T}\in {\bf D_q^m}(\cH)$ with $\rank {\bf \Delta_{q,T}^m}=1$ is
unitarily equivalent to one obtained by compressing $({\bf W}_1,\ldots,
{\bf W}_n)$ to a co-invariant subspace  under  ${\bf W}_{i,j}$, where
  $i\in \{1,\ldots, k\}$ and $  j\in \{1,\ldots, n_i\}$.
Indeed, this follows from Theorem
\ref{dil2},   Corollary \ref{part-cases}, and  the remarks preceding Corollary \ref {part-cases}.

       %

      \end{document}